\documentclass[11pt]{amsart}

\usepackage[margin=0.9in]{geometry}
\usepackage{amsfonts}
\usepackage{amsmath}
\usepackage{amssymb}
\usepackage{mathtools}
\usepackage{booktabs,longtable,array}
\usepackage{enumitem}
\usepackage{hyperref}
\usepackage[mathscr]{eucal}
\newcommand{\os}{\mathrm{os}}
\newcommand{\cyc}{\mathrm{cyc}}
\newcommand{\sub}{\mathrm{sub}}

\DeclareMathOperator{\lcm}{lcm}
\DeclareMathOperator{\Rad}{Rad}
\DeclareMathOperator{\Soc}{Soc}
\DeclareMathOperator{\Aut}{Aut}
\DeclareMathOperator{\Out}{Out}
\DeclareMathOperator{\PSL}{PSL}
\DeclareMathOperator{\PSU}{PSU}
\DeclareMathOperator{\PSp}{PSp}
\newcommand{\POmega}{\mathrm{P}\Omega}
\DeclareMathOperator{\Sp}{Sp}

\newtheorem{theorem}{Theorem}[section]
\newtheorem{proposition}[theorem]{Proposition}
\newtheorem{corollary}[theorem]{Corollary}
\newtheorem{question}[theorem]{Question}
\newtheorem{remark}[theorem]{Remark}
\newtheorem{lemma}[theorem]{Lemma}

\title{Group Structure from Subgroup and Cyclic Subgroup Counts}

\author{Angsuman Das}
\address{Department of Mathematics, Presidency University, Kolkata, India}
\email{angsuman.maths@presiuniv.ac.in}

\author{Hiranya Kishore Dey}
\address{Department of Mathematics, Indian Institute of Technology, Jammu, India}
\email{hiranya.dey@iitjammu.ac.in}

\author{Cesar Galindo}
\address{Departamento de Matem\'aticas, Universidad de los Andes, Bogot\'a, D.C. 111711, Colombia}
\email{cn.galindo1116@uniandes.edu.co}

\author{Khyati Sharma}
\address{Department of Mathematical Sciences, Indian Institute of Science Education and Research, Berhampur, India}
\email{khyatisharma0907@gmail.com}



\begin{document}

\begin{abstract}
For a finite group \(G\), let \(\sub(G)\) be the number of subgroups of \(G\),
let \(\cyc(G)\) be the number of cyclic subgroups, and let \(\pi(G)\) be the
number of distinct prime divisors of \(|G|\).  We study the normalized counts
\(\lambda(G)=\sub(G)/2^{\pi(G)}\) and
\(\eta(G)=\cyc(G)/2^{\pi(G)}\).  We prove that \(\eta(G)<5/4\) or
\(\lambda(G)<3/2\) implies that \(G\) is cyclic of squarefree order.
The inequalities \(\eta(G)<2\) and \(\lambda(G)<5/2\) each force all Sylow
subgroups to be cyclic, and hence imply metacyclicity.  For the subgroup count,
we give an exact arithmetic criterion in the parameters of the corresponding
\(ZM\)-presentation.  We determine all values with \(1<\eta(G)<2\) and
\(1<\lambda(G)<5/2\), and prove that \(\lambda(G)<59/8\) or \(\eta(G)<4\)
implies solvability.  Both solvability bounds are sharp.  We also describe how
cyclic direct factors of coprime order affect the two normalized counts.
\end{abstract}

\maketitle
\setcounter{tocdepth}{1}
\tableofcontents

\section{Introduction}
\label{sec:intro}
Several numerical invariants have been used to study the structure of finite
groups, including commuting probability
\cite{eberhard, guralnick-robinson, gustafson, neumann-com}, sums and averages
of element orders
\cite{asad-joa, her-joa, khukhro-joa-2021, tarnauceanu-israel}, and order
sequences \cite{cameron-dey}.

We study two subgroup-counting invariants:
\[
        \sub(G)=|\{H\le G\}|,
        \qquad
        \cyc(G)=|\{H\le G:H\text{ is cyclic}\}|.
\]
Here \(H\le G\) means that \(H\) is a subgroup of \(G\).  For a positive
integer \(n\), let \(d(n)\) denote the number of positive divisors of \(n\), and
let \(\pi(n)\) denote the number of distinct prime divisors of \(n\).  If \(G\)
is finite, we put \(\pi(G)=\pi(|G|)\).  We use \(\pi(1)=0\),
\(d(1)=1\), and the convention that empty products equal one.
In products of the form \(\prod_{p\mid n}\), the index runs over prime
divisors only.
The subgroup counts grow with the number of distinct prime divisors of
\(|G|\).  Indeed, Theorem~\ref{thm:Richards} gives
\[
        \cyc(G)\ge d(|G|)\ge 2^{\pi(G)}
\]
for every finite group \(G\).  Thus the cyclic groups of squarefree order have
the smallest possible values of both \(\sub(G)\) and \(\cyc(G)\), once we
normalize by \(2^{\pi(G)}\).

We therefore define
\[
        \lambda(G)=\frac{\sub(G)}{2^{\pi(G)}} ,
        \qquad
        \eta(G)=\frac{\cyc(G)}{2^{\pi(G)}} .
\]
Thus \(\lambda(1)=\eta(1)=1\).  They satisfy
\[
        \lambda(G)\ge1,\qquad \eta(G)\ge1,
\]
and cyclic groups of squarefree order have \(\lambda=\eta=1\).  This
normalization also respects coprime direct products: if \(p\nmid |G|\), then direct
product with \(C_p\) doubles both \(\sub(G)\) and \(\cyc(G)\), while also
increasing \(\pi(G)\) by one.  Hence \(\lambda(G)\) and \(\eta(G)\) are
unchanged by adjoining a cyclic factor of new prime order.

Our first results give gaps above the minimum value \(1\).  If
\(\eta(G)<5/4\), or if \(\lambda(G)<3/2\), then the corresponding normalized
count is already equal to \(1\), and \(G\) is cyclic of squarefree order.  

Each of the inequalities
\(\eta(G)<2\) and \(\lambda(G)<5/2\) forces all Sylow subgroups of \(G\)
to be cyclic.  Hence \(G\) is metacyclic and has a \(ZM(m,n,r)\)-presentation,
recalled in Section~\ref{sec:preli}.  The subgroup-count formula for these
groups gives an exact arithmetic characterization of the parameters satisfying
\(\lambda(G)<5/2\).  These structural reductions also determine the first
nontrivial normalized values:
\[
        1<\eta(G)<2,
        \qquad
        1<\lambda(G)<\frac52.
\]
They are
\[
        \frac54,\frac32,\frac74,\frac{15}{8}
        \quad\text{for }\eta,
        \qquad
        \frac32,2,\frac94
        \quad\text{for }\lambda.
\]

We also prove that each of the following inequalities implies solvability:
\[
        \eta(G)<4,\qquad \lambda(G)<\frac{59}{8}.
\]
Both constants are sharp.  The subgroup-count theorem is proved by reducing to
almost simple groups and using existing results on groups with few subgroups.
For the cyclic-subgroup theorem, we prove a lower bound for finite nonabelian
simple groups, treating each family separately.

The main conclusions are summarized in Table~\ref{tab:main-results}.

\begin{table}[ht]
    \centering
    \begingroup
    \renewcommand{\arraystretch}{1.0}
    \begin{tabular}{||c|c|>{\raggedright\arraybackslash}p{0.72\textwidth}||}
    \hline \hline
    Invariant & Threshold & Conclusion \\
    \hline
    \(\eta\) & \(<5/4\) &
    \(\eta(G)=1\); \(G\) is cyclic of squarefree order (Theorem~\ref{thm:nilpotency-criterion}).\\
    \(\eta\) & \(<2\) &
    \(G\) has cyclic Sylow subgroups and is metacyclic; the values in \((1,2)\) are listed in Theorem~\ref{thm:small-eta-values}.\\
    \(\eta\) & \(<4\) &
    \(G\) is solvable (Theorem~\ref{thm:main}).\\
    \hline
    \(\lambda\) & \(<3/2\) &
    \(\lambda(G)=1\); \(G\) is cyclic of squarefree order (Theorem~\ref{thm:nilpotency-criterion-Sub}).\\
    \(\lambda\) & \(<5/2\) &
    \(G\) is a \(ZM(m,n,r)\)-group satisfying an explicit parameter inequality (Corollary~\ref{cor:sub-bound-zm-classification}).\\
    \(\lambda\) & \(<59/8\) &
    \(G\) is solvable (Theorem~\ref{thm:solvable}).\\
    \hline \hline
    \end{tabular}
    \endgroup
    \caption{Normalized threshold results}
    \label{tab:main-results}
\end{table}

The groups \(S_3\), \(A_4\), and \(A_5\) show that these inequalities
must be strict.  For \(S_3\), we have
\[
        \eta(S_3)=\frac54,\qquad \lambda(S_3)=\frac32.
\]
The group \(A_4\) lies on the next boundary:
\[
        \eta(A_4)=2,\qquad \lambda(A_4)=\frac52,
\]
and \(A_5\) lies on both solvability boundaries:
\[
        \eta(A_5)=4,\qquad \lambda(A_5)=\frac{59}{8}.
\]
Multiplying these examples by cyclic groups of new prime order preserves the
corresponding normalized value.

The first-gap and metacyclicity proofs use induction, order-sequence comparisons,
and direct subgroup counts.  For the solvability theorem for \(\lambda\),
we reduce to a list of almost simple groups and use the subgroup-count theorem
of Das--Mandal \cite[arXiv version, Theorem~2.1]{das-com} when the group order
has exactly three distinct prime divisors.  For \(\eta\), the main input is a lower bound for
the number of cyclic subgroups of every finite nonabelian simple group.  The
estimate needed for the reduction is slightly stronger than \(\eta(S)\ge4\):
we prove
\[
        \cyc(S)\ge 2^{\pi(\Aut(S))+2}
\]
for every finite nonabelian simple group \(S\).

To study the groups below the solvability bounds, we separate cyclic direct
factors of coprime order.  The contribution of such a factor \(C\) to either
normalized count is
\[
        \delta(C)=\frac{d(|C|)}{2^{\pi(C)}}.
\]
After separating such a factor, the subgroup-count condition becomes
\[
        \lambda(S)\delta(C)<\frac{59}{8};
\]
and the cyclic-subgroup condition becomes
\[
        \eta(S)\delta(C)<4.
\]
These formulas reduce the problem to groups without coprime cyclic direct
factors and specify which cyclic factors can be adjoined without exceeding
the bounds.

The paper is organized as follows.  In Section~\ref{sec:preli}, we collect the
preliminaries used throughout the paper.  Section~\ref{sec:nilpo} proves the
first gaps for \(\eta\) and \(\lambda\), isolating the cyclic groups of
squarefree order.  Section~\ref{sec:sup-solv} treats the next thresholds and
culminates in the arithmetic characterization of \(ZM\)-groups below
\(\lambda<5/2\).  Section~\ref{sec:solv} proves the solvability theorem for
\(\lambda<59/8\), and Section~\ref{sec:sub-threshold-structure} records the
corresponding decomposition by coprime cyclic direct factors.
Section~\ref{sec:cyc-solv} states the solvability theorem for \(\eta<4\)
and the required lower bound for finite nonabelian simple groups.
Section~\ref{sec:from-simple} assembles the simple-group estimates and proves
the solvability theorem.  Section~\ref{sec:cyc-threshold-structure} records the
parallel decomposition below \(\eta<4\).  The detailed auxiliary estimates
and Lie-type notation are collected in Appendix~\ref{app:auxiliary-estimates}.
Appendices~\ref{sec:alternating}--\ref{sec:remaining-cases} treat the alternating
groups, the groups \(\PSL(2,q)\), the remaining Lie-type groups, and the
residual finite cases, including the sporadic groups.  The GAP verification
code and complete reference output are provided in Online Resource~1.

\section{Preliminaries}
\label{sec:preli}

We collect the counting results needed below and use standard finite-group terminology
as in \cite{Isac-ams, Scott}.  We write \(C_n\) for the cyclic group of order
\(n\), \(S_n\) for the symmetric group on \(n\) letters, and \(o(g)\) for the
order of an element \(g\).  We denote Euler's totient function by \(\phi\)
and the product of the distinct prime divisors of \(n\) by
\(\operatorname{rad}(n)\).
The following result of Richards \cite{richards}
says that, among groups of a fixed order \(n\), the cyclic group has the
smallest number of cyclic subgroups.

\begin{theorem}[Richards]
\label{thm:Richards}
Let \(G\) be a group of order \(n\).  Then
\[
        \cyc(G)\ge \cyc(C_n),
\]
with equality if and only if \(G\cong C_n\).
\end{theorem}

The following lemma from \cite{Garonzi-lima} connects the orders of the
elements of \(G\) with the number of cyclic subgroups.

\begin{lemma}
\label{lem:cyclic-subgroup-number-in-terms-order}
Let \(G\) be a finite group.  Then
\[
        \cyc(G)=\sum_{g\in G}\frac{1}{\phi(o(g))}.
\]
\end{lemma}

\begin{theorem}
\label{thm:product-cyclic-grps}
Let \(G\) and \(H\) be finite groups.  Then
\[
        \cyc(G \times H) \geq \cyc(G)\cyc(H).
\]
If \((|G|,|H|)=1\), then equality holds.
\end{theorem}
\begin{proof}
For any \((g,h) \in G \times H\), we have
\[
        o(g,h)=\lcm(o(g),o(h)).
\]
Therefore,
\[
\cyc(G \times H)
= \sum_{(g,h) \in G \times H} \frac{1}{\phi(\lcm(o(g),o(h)))}.
\]
We now use the inequality \(\phi(\lcm(a,b))\le \phi(a)\phi(b)\) to get
\begin{align*}
\cyc(G) \cyc(H)
= & \sum _{g \in G} \frac{1}{\phi(o(g))}  \sum _{h \in H} \frac{1}{\phi(o(h))}   \\
= & \sum _{g \in G}  \sum _{h \in H} \frac{1}{\phi(o(g))}  \frac{1}{\phi(o(h))} \\
\leq & \sum _{g \in G, h \in H} \frac{1}{\phi(o(g,h))}
 = \cyc(G \times H).\\
\end{align*}
If \((|G|,|H|)=1\), then \(o(g)\) and \(o(h)\) are coprime for all
\(g\in G\) and \(h\in H\), and hence
\(\phi(o(g,h))=\phi(o(g))\phi(o(h))\).  The preceding inequalities are then
equalities.
\end{proof}

We next recall the order sequence of a group from \cite{cameron-dey}.  For a
finite group \(G\) of order \(n\), its order sequence is
\[
        \os(G)=(o(g_1),o(g_2),\ldots,o(g_n)),
\]
where \(o(g_i)\le o(g_{i+1})\) for \(1\le i\le n-1\).  For two groups \(G\)
and \(H\) of order \(n\), we say that \(\os(G)\) dominates \(\os(H)\) if
\(o(g_i)\ge o(h_i)\) for \(1\le i\le n\).  We say that \(\os(G)\) strongly
dominates \(\os(H)\) if there is a bijection \(f:G\to H\) such that
\(o(f(g))\mid o(g)\) for all \(g\in G\).  The following observation is
immediate from the definition.

 \begin{lemma}
\label{lem:strong-dom-implies-less-cyc}
Let \(G\) and \(H\) be finite groups of the same order.  If \(\os(G)\)
strongly dominates \(\os(H)\), then \(\cyc(G)\le \cyc(H)\).
 \end{lemma}

Amiri \cite{amiri-jpaa} showed that for every positive integer \(n\),
\(\os(C_n)\) strongly dominates \(\os(G)\) for every non-cyclic group \(G\) of
order \(n\).  This gives another proof of Theorem~\ref{thm:Richards}.
Answering \cite[Question~1.5]{amiri-amiri}, Amiri \cite{amiri-jaco} also
proved the following result.

\begin{theorem}
\label{thm:amiri-jaco}
Let \(G\) be a finite group of order \(n\), and let \(p\mid n\) be prime.  If a
Sylow \(p\)-subgroup of \(G\) is neither cyclic nor generalized quaternion,
then the order sequence of \(C_{n/p}\times C_p\) strongly dominates the order
sequence of \(G\).
 \end{theorem}

The following result from \cite[Theorem~8]{cameron-dey} compares the order
sequences of coprime extensions with those of direct products.

\begin{theorem}
\label{t:extension}
Let \(G\) and \(H\) be groups of coprime order, and suppose that \(G\) is
abelian.  Let \(K\) be any extension of \(G\) by \(H\).  Then
\(\os(G\times H)\) strongly dominates \(\os(K)\).
\end{theorem}

\begin{corollary}
\label{cor:cyc-nm}\label{lem:coprime-extension}
Let \(G\) and \(H\) be groups of coprime order, and suppose that \(G\) is
abelian.  Let \(K\) be any extension of \(G\) by \(H\).  Then
\[
        \cyc(G\times H)\le \cyc(K).
\]
\end{corollary}

We also use the following elementary observation.
\begin{lemma}
\label{lem:distinct-Sub-semi-distinct-Sub}
Let \(G=H\rtimes K\), where \(H\triangleleft G\).  If \(L\le K\), then \(HL\)
is a subgroup of \(G\).  Moreover, distinct subgroups \(L\le K\) give distinct
subgroups \(HL\).
\end{lemma}

We also use known formulas for groups whose Sylow subgroups are all
cyclic.  Such groups are often called \(ZM\)-groups.  By the theorem of
Zassenhaus, every \(ZM\)-group is metacyclic \cite{Scott}, but the converse is
not true in general.  Thus \(ZM\)-groups are the metacyclic groups with the
additional condition that all Sylow subgroups are cyclic.  In the notation of
T\u{a}rn\u{a}uceanu--T\'{o}th, they have presentations
\[
        ZM(m,n,r)=\langle a,b\mid a^m=b^n=1,\ b^{-1}ab=a^r\rangle,
\]
where
\[
        (m,n)=(m,r-1)=1,\qquad r^n\equiv1\pmod m.
\]
Calhoun described the subgroup lattice of these cyclic extensions of cyclic
groups \cite{Calhoun1987}; T\u{a}rn\u{a}uceanu--T\'{o}th record the resulting
formulas for both the number of subgroups and the number of cyclic subgroups
of \(ZM(m,n,r)\) \cite[Section~3.4]{TarnauceanuToth2015}.

\begin{proposition}\label{prop:zm-counts}
Let \(G=ZM(m,n,r)\), and write
\[
        R(n_1)=\sum_{j=0}^{n/n_1-1}r^{jn_1}
        \qquad (n_1\mid n).
\]
When \(r^{n_1}\ne1\), this equals \((r^n-1)/(r^{n_1}-1)\); throughout,
any such quotient denotes this geometric sum even when its denominator
vanishes.
Then
\[
        \sub(G)=
        \sum_{m_1\mid m}\sum_{n_1\mid n}
        \gcd(m_1,R(n_1)).
\]
Moreover,
\[
        \cyc(G)=
        \sum_{\substack{m_1\mid m,\ n_1\mid n\\ m/m_1\mid r^{n_1}-1}}
        \gcd(m_1,R(n_1)).
\]

For the local factorizations, write \(p^{a_p}\Vert m\), let \(d_p\) be the
order of \(r\) modulo \(p\), and put
\[
(c_p(n_1),h_p(n_1))=
\begin{cases}
(a_p+1,a_p+1),&d_p\mid n_1,\\
\left(p^{a_p},\displaystyle\sum_{i=0}^{a_p}p^i\right),&d_p\nmid n_1.
\end{cases}
\]
Then
\[
        \cyc(G)=\sum_{n_1\mid n}\prod_{p\mid m}c_p(n_1),\qquad
        \sub(G)=\sum_{n_1\mid n}\prod_{p\mid m}h_p(n_1).
\]
For \(m=1\), the products are empty and equal \(1\); both counts are
\(d(n)\), as required for \(ZM(1,n,1)\cong C_n\).
\end{proposition}

\begin{proof}
The first formula is the subgroup-count formula for cyclic extensions of
cyclic groups proved in \cite{Calhoun1987}.  The second is the
cyclic-subgroup formula for \(ZM\)-groups recorded by
T\u{a}rn\u{a}uceanu--T\'{o}th \cite[Equation~(8)]{TarnauceanuToth2015}.

We prove the local factorizations.  Since \(r^n\equiv1\pmod m\), we have
\(d_p\mid n\) for every \(p\mid m\).  If \(d_p\mid n_1\), the geometric sum gives
\(R(n_1)\equiv n/n_1\not\equiv0\pmod p\), because \((m,n)=1\).
The identity \((r^{n_1}-1)R(n_1)=r^n-1\) then implies
\(p^{a_p}\mid r^{n_1}-1\).  Thus every choice \(p^i\), \(0\le i\le a_p\),
for the \(p\)-part of \(m_1\) is admissible in both sums and has local weight
\(1\).  If \(d_p\nmid n_1\), then \(p\nmid r^{n_1}-1\), so
\(p^{a_p}\mid R(n_1)\).  In the cyclic sum only \(i=a_p\) is admissible,
with local weight \(p^{a_p}\); in the subgroup sum all \(i\) are admissible,
with local weight \(p^i\).

For fixed \(n_1\), these choices of exponents form a Cartesian product:
the divisibility condition is checked separately at each prime, and
\(\gcd(m_1,R(n_1))=\prod_{p\mid m}\gcd(p^{v_p(m_1)},R(n_1))\).
Multiplying the local weights and summing over \(n_1\mid n\) gives the
asserted factorizations.
\end{proof}

\subsection{\texorpdfstring{The group \(\PSL(2,q)\)}{The group PSL(2,q)} and an inequality}

In this subsection, we record the number of involutions of \(\PSL(2,q)\) for
an odd prime power \(q\) and state an inequality used later to bound its
subgroup count.  The involution count follows from the rank-one structure of
\(\PSL(2,q)\); see Malle--Testerman
\cite[Theorem~14.2, Corollaries~24.6 and~24.11, Propositions~25.3 and~26.6]{MalleTesterman2011}.

\begin{lemma}
\label{lem:num-invol}
The number of involutions in \(\PSL(2,q)\), for \(q\) an odd prime power, is
\[
        \frac{q(q+1)}{2}\quad\text{if }q\equiv1\pmod4,
        \qquad
        \frac{q(q-1)}{2}\quad\text{if }q\equiv3\pmod4.
\]
\end{lemma}

The involution count is also recovered in the proof of
Lemma~\ref{lem:psl2-formula} in Appendix~\ref{sec:psl2}.

\begin{lemma}
\label{lem:ineq-primes}
Let \(q\) be an odd prime power, put \(N=q(q+1)(q-1)/2\), and put
\(t=\pi(N)\).  Then
\[
        q\ge7 \Longrightarrow q(q-1)\ge5\cdot2^t,
        \qquad
        q\ge37 \Longrightarrow q(q-1)\ge59\cdot2^{t-1}.
\]
\end{lemma}

The proof is given in Appendix~\ref{app:auxiliary-estimates}.

\section{First gaps for the normalized counts}
\label{sec:nilpo}

We prove that \(\eta(G)<5/4\) or \(\lambda(G)<3/2\) implies that \(G\)
is cyclic of squarefree order.  In either case, \(\eta(G)=\lambda(G)=1\).

\begin{lemma}\label{lem:squarefree-semidirect-cyclic-count}
Let \(p\) be a prime, let \(m\) be squarefree with \((p,m)=1\), and let
\[
        X=C_p\rtimes_\varphi C_m.
\]
Suppose that the action is nontrivial.  Let \(k=|\operatorname{Im}\varphi|\).
Then \(k>1\), \(k\mid m\), and
\[
        \cyc(X)=p\,d(m)-(p-2)d(m/k).
\]
\end{lemma}

\begin{proof}
This is the specialization of Proposition~\ref{prop:zm-counts} to the group
\(X=C_p\rtimes_\varphi C_m\).  The kernel of \(\varphi\) has
order \(m/k\), and
\[
        X\cong (C_p\rtimes C_k)\times C_{m/k},
\]
with faithful action on the first factor; the cyclic-subgroup formula in
Proposition~\ref{prop:zm-counts} gives
\[
        \cyc(C_p\rtimes C_k)=2+p(d(k)-1).
\]
Since the two factors have coprime orders,
Theorem~\ref{thm:product-cyclic-grps} gives
\[
        \cyc(X)=d(m/k)\bigl(2+p(d(k)-1)\bigr)
        =p\,d(m)-(p-2)d(m/k).
\]
\end{proof}

\begin{theorem}
\label{thm:nilpotency-criterion}
Let \(G\) be a finite group.  If
\[
        \eta(G)<\frac54,
\]
then \(\eta(G)=1\), and \(G\) is cyclic of squarefree order.
\end{theorem}

\begin{proof}
The case \(G=1\) is immediate, so assume \(G\ne1\).
Put \(t=\pi(G)\).  The hypothesis reads
\[
        \cyc(G)<5\cdot2^{t-2}.
\]
We prove first that \(G\) is nilpotent.  We argue by contrapositive induction
on \(t\).  The case \(t=1\) is clear.  Assume \(t\ge2\), and suppose
that \(G\) is not nilpotent.  By
Theorem~\ref{thm:Richards},
\[
        d(|G|)=\cyc(C_{|G|})\le \cyc(G).
\]
If \(|G|\) is not squarefree, then
\[
        d(|G|)\ge 3\cdot 2^{t-1}>5\cdot 2^{t-2}.
\]
Thus we may assume that \(|G|\) is squarefree.

Write
\[
        |G|=p_1p_2\cdots p_t,\qquad p_1<p_2<\cdots<p_t.
\]
A group of squarefree order is supersolvable, and its Sylow subgroup
corresponding to the largest prime divisor is normal.  Therefore
\[
        G=C_{p_t}\rtimes H,
        \qquad |H|=m=p_1p_2\cdots p_{t-1}.
\]
We first show that \(H\) is nilpotent.  If not, the induction hypothesis gives
\[
        \cyc(H)\ge 5\cdot 2^{t-3}.
\]
By Corollary~\ref{cor:cyc-nm} and Theorem~\ref{thm:product-cyclic-grps},
\[
        \cyc(G)\ge \cyc(C_{p_t}\times H)
        =2\cyc(H)
        \ge 5\cdot 2^{t-2},
\]
as required.  Hence \(H\) is nilpotent.  Since \(H\) has squarefree order, it
is cyclic, and therefore
\[
        G=C_{p_t}\rtimes_\varphi C_m.
\]
If the action is trivial, then \(G\) is cyclic, hence nilpotent, a
contradiction.  Thus the action is nontrivial.  Let
\(k=|\operatorname{Im}\varphi|\).  By
Lemma~\ref{lem:squarefree-semidirect-cyclic-count},
\begin{equation}
\label{eqn:cyc-p-nbyp}
        \cyc(G)=p_t\cdot 2^{t-1}-(p_t-2)\cdot 2^{t-1-\pi(k)},
\end{equation}
because \(m\) is squarefree and \(k>1\).  Hence
\begin{equation}
\label{eqn:cyc-p-nbyp-2}
        \frac{\cyc(G)}{2^{t-1}}
        =p_t-\frac{p_t-2}{2^{\pi(k)}}.
\end{equation}
If \(p_t\ge5\), then the right hand side of \eqref{eqn:cyc-p-nbyp-2} is at
least \((p_t+2)/2\ge7/2\).  If \(p_t=3\), then it is at least
\[
        3-\frac12=\frac52.
\]
Therefore \(\cyc(G)\ge 5\cdot 2^{t-2}\) in all cases.
Hence the assumed inequality forces \(G\) to be nilpotent.  The same divisor
comparison at the beginning of the proof also forces \(|G|\) to be squarefree.
A nilpotent group of squarefree order is cyclic.  For a cyclic group of
squarefree order, \(\cyc(G)=2^{\pi(G)}\), and hence \(\eta(G)=1\).
\end{proof}

\begin{remark}
\label{rem:nilpotency}
Theorem~\ref{thm:nilpotency-criterion} is a rigidity statement: below the first
gap one gets the extremal cyclic squarefree groups.  The constant \(5/4\) is
sharp, since \(S_3\) and its products with cyclic groups of new prime order have
\(\eta=5/4\) and are not nilpotent.
\end{remark}

\begin{lemma}
\label{lem:squarefree-semidirect-subgroup-count}
Let \(p\) be a prime, let \(m\) be a squarefree positive integer with
\((p,m)=1\), and let
\[
        X=C_p\rtimes_{\varphi}C_m.
\]
Assume that the action is nontrivial, and put
\[
        k=|\operatorname{Im}\varphi|.
\]
Then
\[
        \sub(X)=(p+1)d(m)-(p-1)d(m/k).
\]
\end{lemma}

\begin{proof}
This is the corresponding specialization of Proposition~\ref{prop:zm-counts}.
As above,
\[
        X\cong (C_p\rtimes C_k)\times C_{m/k},
\]
with faithful action on the first factor.  The subgroup-count formula in
Proposition~\ref{prop:zm-counts} gives
\[
        \sub(C_p\rtimes C_k)=(p+1)d(k)-(p-1).
\]
Since the second factor is cyclic and has coprime order, subgroups split over
the direct product.  Hence
\[
        \sub(X)=d(m/k)\bigl((p+1)d(k)-(p-1)\bigr)
        =(p+1)d(m)-(p-1)d(m/k).
\]
\end{proof}

The proof of the next result follows the same reduction as the proof of
Theorem~\ref{thm:nilpotency-criterion}, but the final semidirect-product
count uses Lemma~\ref{lem:squarefree-semidirect-subgroup-count}.

\begin{theorem}
\label{thm:nilpotency-criterion-Sub}
Let \(G\) be a finite group.  If
\[
        \lambda(G)<\frac32,
\]
then \(\lambda(G)=1\), and \(G\) is cyclic of squarefree order.
\end{theorem}

\begin{proof}
The case \(G=1\) is immediate, so assume \(G\ne1\).
Put \(t=\pi(G)\).  The hypothesis reads
\[
        \sub(G)<6\cdot2^{t-2}.
\]
We argue by induction on \(t\).  First, Theorem~\ref{thm:Richards}
gives
\[
        d(|G|)\le \cyc(G)\le \sub(G).
\]
If \(|G|\) is not squarefree, then
\[
        d(|G|)\ge 3\cdot 2^{t-1}=6\cdot 2^{t-2}.
\]
Thus the strict inequality in the hypothesis forces \(|G|\) to be squarefree.
For \(t=1\), this already implies that \(G\) has prime order, and the result
follows.

Assume now that \(t\ge2\), and write
\[
        |G|=p_1p_2\cdots p_t,\qquad p_1<p_2<\cdots<p_t.
\]
A group of squarefree order is supersolvable, and its Sylow subgroup
corresponding to the largest prime divisor is normal.  Thus
\[
        G=P\rtimes H,\qquad P\cong C_{p_t},\quad |H|=p_1p_2\cdots p_{t-1}.
\]
We claim that \(H\) is nilpotent.  If \(H\) is not nilpotent, then the
induction hypothesis, applied contrapositively to \(H\), gives
\[
        \sub(H)\ge 6\cdot 2^{t-3}.
\]
Inside \(G=P\rtimes H\), the subgroups \(K\le H\) and the subgroups \(PK\) with
\(K\le H\) form two disjoint families.  Hence
\[
        \sub(G)\ge 2\sub(H)\ge 6\cdot 2^{t-2},
\]
contrary to the hypothesis.  Therefore \(H\) is nilpotent.  Since \(H\) has
squarefree order, it is cyclic, and we may write
\[
        G\cong C_{p_t}\rtimes_{\varphi}C_m,
        \qquad m=p_1p_2\cdots p_{t-1}.
\]

If the action is trivial, then \(G\) is cyclic of squarefree order.  Suppose
that the action is nontrivial, and put
\[
        k=|\operatorname{Im}\varphi|.
\]
By Lemma~\ref{lem:squarefree-semidirect-subgroup-count},
\[
        \sub(G)
        =(p_t+1)d(m)-(p_t-1)d(m/k).
\]
Since \(m\) is squarefree, \(d(m)=2^{t-1}\), and therefore
\[
        \frac{\sub(G)}{2^{t-1}}
        =
        p_t+1-\frac{p_t-1}{2^{\pi(k)}}.
\]
The action is nontrivial; hence \(\pi(k)\ge1\).  Also \(p_t\ge3\).  Hence
\[
        \frac{\sub(G)}{2^{t-1}}
        \ge p_t+1-\frac{p_t-1}{2}
        =\frac{p_t+3}{2}\ge3.
\]
Thus \(\sub(G)\ge 6\cdot 2^{t-2}\), again contradicting the hypothesis.
Therefore the action is trivial, and \(G\) is cyclic of squarefree order.  For
such a group, \(\sub(G)=2^{\pi(G)}\), and hence \(\lambda(G)=1\).
\end{proof}

\begin{remark}
\label{rem:sub-nilpotency}
Theorem~\ref{thm:nilpotency-criterion-Sub} has the same rigidity feature.  The
first nonabelian examples occur at the boundary: \(S_3\), and its products with
cyclic groups of new prime order, have \(\lambda=3/2\) and are not nilpotent.
\end{remark}

Combining Theorems~\ref{thm:nilpotency-criterion}
and~\ref{thm:nilpotency-criterion-Sub}, we obtain the following immediate
consequence.

\begin{corollary}\label{cor:initial-threshold-lower-bounds}
If \(G\) is a finite group and either
\[
        \eta(G)<\frac54
        \qquad\text{or}\qquad
        \lambda(G)<\frac32,
\]
then \(\eta(G)=\lambda(G)=1\), and \(G\) is cyclic of squarefree order.
\end{corollary}

\section{\texorpdfstring{Cyclic Sylow subgroups and \(ZM\)-group criteria}{Cyclic Sylow subgroups and ZM-group criteria}}
 \label{sec:sup-solv}

We now consider the thresholds \(\eta(G)<2\) and \(\lambda(G)<5/2\).
Both force every Sylow subgroup of \(G\) to be cyclic, so \(G\) is a
\(ZM\)-group and, in particular, metacyclic and supersolvable \cite{Scott}.
Nonabelian examples already occur, such as \(S_3\), so neither threshold
forces nilpotency.

We first prove the cyclic-subgroup result and determine the values of \(\eta\)
in \((1,2)\).  For the subgroup count, we establish supersolvability, then
cyclic Sylow subgroups, and finally an arithmetic criterion in the \(ZM\)
parameters and the values of \(\lambda\) in \((1,5/2)\).

\begin{theorem}
\label{thm:sup-solv-criterion-cyclic-subgroups}
Let \(G\) be a group with
\[
        |G|=\prod_{i=1}^k p_i^{r_i},
        \qquad p_1<p_2<\cdots<p_k,
\]
and assume that at least one \(r_i\) is larger than \(1\).  Suppose also that
no Sylow subgroup of \(G\) is generalized quaternion.  If
\[
        \cyc(G)
        < d(|G|)\min_{r_i>1}\left\{\frac{2r_i}{r_i+1}\right\},
\]
then every Sylow subgroup of \(G\) is cyclic.
\end{theorem}

\begin{proof}
Put \(n=|G|\).  Suppose that the Sylow \(p_i\)-subgroup of \(G\) is not
cyclic.  Since no Sylow subgroup is generalized quaternion, Theorem
\ref{thm:amiri-jaco} gives a bijection from \(G\) onto
\(C_{n/p_i}\times C_{p_i}\) such that element orders in \(G\) divide the
orders of their images.  Hence
\[
        \cyc(G)\ge \cyc(C_{n/p_i}\times C_{p_i}).
\]
By Theorem~\ref{thm:product-cyclic-grps},
\[
\begin{aligned}
        \cyc(G)
        &\ge \cyc(C_{n/p_i}\times C_{p_i}) \\
        &\ge \cyc(C_{n/p_i})\cyc(C_{p_i})
         =2r_i\prod_{j\ne i}(r_j+1)\\
        &=d(|G|)\frac{2r_i}{r_i+1}.
\end{aligned}
\]
In particular,
\[
        \cyc(G)\ge
        d(|G|)\min_{r_j>1}\left\{\frac{2r_j}{r_j+1}\right\},
\]
contrary to the hypothesis.  Therefore every Sylow subgroup of \(G\) is
cyclic.
\end{proof}

The preceding estimate gives the threshold \(\eta<2\).

\begin{theorem}
\label{thm:super-solvability-criterion}
Let \(G\) be a finite group.  If
\[
        \eta(G)<2,
\]
then every Sylow subgroup of \(G\) is cyclic.  In particular, \(G\) is
metacyclic.
\end{theorem}

\begin{proof}
Put \(t=\pi(G)\).  The hypothesis reads
\[
        \cyc(G)<2^{t+1}.
\]
Write
\[
        |G|=\prod_{i=1}^t p_i^{r_i},
        \qquad p_1<p_2<\cdots<p_t.
\]
If all \(r_i=1\), then all Sylow subgroups have prime order and hence are
cyclic.

Assume now that some \(r_i>1\).  We first rule out generalized quaternion Sylow
subgroups.  If \(G\) has such a Sylow subgroup, then \(8\mid |G|\).  By
Theorem~\ref{thm:Richards},
\[
        \cyc(G)\ge d(|G|)\ge 4\cdot2^{t-1}=2^{t+1},
\]
contrary to the hypothesis.  Thus no Sylow subgroup of \(G\) is generalized
quaternion.

It remains only to compare the hypothesis with the threshold in
Theorem~\ref{thm:sup-solv-criterion-cyclic-subgroups}.  Since some \(r_i>1\),
\[
        d(|G|)
        \min_{r_i>1}\left\{\frac{2r_i}{r_i+1}\right\}
        =
        \min_{r_i>1}\left\{2r_i\prod_{j\ne i}(r_j+1)\right\}
        \ge 2^{t+1}.
\]
Therefore
\[
        \cyc(G)<2^{t+1}
        \le
        d(|G|)
        \min_{r_i>1}\left\{\frac{2r_i}{r_i+1}\right\}.
\]
Theorem~\ref{thm:sup-solv-criterion-cyclic-subgroups} now implies that every
Sylow subgroup of \(G\) is cyclic.

In all cases, every Sylow subgroup of \(G\) is cyclic.  Hence \(G\) is
metacyclic \cite{Scott}.
\end{proof}

\begin{remark}
\label{rem:supersolvability}
The constant \(2\) is sharp for the metacyclic conclusion: \(A_4\) has
\(\eta(A_4)=2\) and is not metacyclic.  Below this boundary nonabelian examples
still occur, for instance \(S_3\) and its products with cyclic groups of new
prime order.
\end{remark}

Combining Theorem~\ref{thm:super-solvability-criterion} with the
cyclic-subgroup formula in Proposition~\ref{prop:zm-counts} determines the
possible values of \(\eta\) below \(2\).

\begin{theorem}\label{thm:small-eta-values}
Let \(G\) be a finite group.  If
\[
        1<\eta(G)<2,
\]
then
\[
        \eta(G)\in\left\{\frac54,\,\frac32,\,\frac74,\,\frac{15}{8}\right\}.
\]
All four values occur.
\end{theorem}

\begin{proof}
Put \(t=\pi(G)\).  Since
\[
        \cyc(G)<2^{t+1},
\]
Theorem~\ref{thm:Richards} gives
\[
        d(|G|)\le\cyc(G)<2^{t+1}.
\]
It follows that \(|G|\) is either squarefree or has exactly one prime divisor
with exponent \(2\), and all other prime divisors occur to exponent \(1\).
Indeed, an exponent at least \(3\), or two distinct exponents at least \(2\),
would give \(d(|G|)\ge2^{t+1}\).

If \(G\) is cyclic and \(|G|=\prod_i p_i^{a_i}\), then
\[
        \eta(G)=\frac{d(|G|)}{2^{\pi(G)}}
        =\prod_i\frac{a_i+1}{2}.
\]
The order pattern just obtained shows that the only value of this product in
\((1,2)\) is \(3/2\), attained when exactly one exponent is \(2\).

Assume now that \(G\) is not cyclic.  By
Theorem~\ref{thm:super-solvability-criterion}, all Sylow subgroups of \(G\) are
cyclic.  Hence
\[
        G=ZM(m,n,r)
\]
with \(m,n>1\).  We use the notation of Proposition~\ref{prop:zm-counts}.

We first restrict the possible values of \(m\).  The defining conditions
force \(m\) to be odd: if \(2\mid m\), then \(r\) would be odd and
\((m,r-1)>1\).  Put \(h=\pi(m)\), \(k=\pi(n)\), and use the factors
\(c_p(n_1)\) defined in Proposition~\ref{prop:zm-counts}.
For \(p^a\Vert m\), we have \(d_p>1\), since \((m,r-1)=1\).  Put
\(s=\pi(d_p)\ge1\), and restrict the outer sum to
\[
\mathcal D(n)=\left\{\prod_{q\mid n}q^{\varepsilon_q v_q(n)}:
                    \varepsilon_q\in\{0,1\}\right\}.
\]
Unique factorization makes these \(2^k\) divisors distinct.  Since
\(d_p\mid n\), precisely \(2^{k-s}\) of them are divisible by \(d_p\):
they are the choices with \(\varepsilon_q=1\) for every \(q\mid d_p\).
For the other choices \(c_p(n_1)=p^a\), whereas on these \(2^{k-s}\)
choices \(c_p(n_1)=a+1\ge2\).  Every other factor \(c_\ell(n_1)\) is at
least \(2\).  Therefore the exact product formula gives
\[
        \cyc(G)\ge
        2^{h-1}\bigl((2^k-2^{k-s})p^a+2^{k-s+1}\bigr),
\]
and hence
\[
        \eta(G)\ge\frac{(2^s-1)p^a+2}{2^{s+1}}
                  \ge\frac{p^a+2}{4}.
\]
The last inequality uses \(s\ge1\) and \(p^a\ge3\).  Since \(\eta(G)<2\),
we obtain \(a=1\) and \(p\in\{3,5\}\).  Thus \(m\mid15\).

To exclude \(m=15\), write \(n=2^b u\) with \(u\) odd.  The order of
\(r\) modulo \(3\) is \(2\); its order modulo \(5\) is \(2\) or \(4\).
For each squarefree divisor \(v\mid u\), if \(d_5=2\), the distinct
choices \(n_1=v,2v\) contribute respectively
\(c_3(n_1)c_5(n_1)=15,4\).  There are \(2^{\pi(u)}\) such divisors \(v\),
and the resulting choices are all distinct, so
\[
        \eta(G)\ge\frac{(15+4)2^{\pi(u)}}{2^{\pi(u)+3}}
                  =\frac{19}{8}>2.
\]
If \(d_5=4\), then \(b\ge2\), and the choices \(n_1=v,2v,4v\) give
products \(15,10,4\), respectively.  The same count gives
\(\eta(G)\ge29/8>2\).  Thus \(m=3\) or \(m=5\).

If \(m=5\) and \(d_5=4\), the choices \(n_1=v,2v,4v\), again for every
squarefree \(v\mid u\), have local weights \(5,5,2\).  This gives
\(\eta(G)\ge(5+5+2)/4=3\), also impossible.  Therefore in every remaining
case \(d_m=2\), and \(r\equiv-1\pmod m\).

Write
\[
        n=2^b u,\qquad b\ge1,\qquad u\text{ odd}.
\]
Let \(p=m\); thus \(p=3\) or \(p=5\).  Since \(d_p=2\), the local
factorization in Proposition~\ref{prop:zm-counts} gives
\(c_p(n_1)=p\) for odd \(n_1\), and \(c_p(n_1)=2\) for even \(n_1\).
There are \(d(u)\) odd divisors of \(n\) and \(bd(u)\) even ones.
Therefore
\[
        \cyc(G)=p\,d(u)+2b\,d(u)=(p+2b)d(u).
\]
Consequently
\[
        \eta(G)
        =
        \frac{p+2b}{4}\cdot
        \frac{d(u)}{2^{\pi(u)}}.
\]
Put
\[
        \Delta(u)=\frac{d(u)}{2^{\pi(u)}}.
\]
The order pattern at the beginning of the proof implies \(b\le2\).  If
\(b=2\), then \(u\) is squarefree and \(\Delta(u)=1\).  If \(b=1\), then
\(\Delta(u)\in\{1,3/2\}\).

If \(p=3\), then
\[
        \eta(G)=\frac{3+2b}{4}\Delta(u).
\]
The condition \(\eta(G)<2\) leaves only the following possibilities:
\[
        b=1,\ \Delta(u)=1;\qquad
        b=1,\ \Delta(u)=\frac32;\qquad
        b=2,\ \Delta(u)=1.
\]
These give respectively
\[
        \eta(G)=\frac54,\qquad \frac{15}{8},\qquad \frac74.
\]
If \(p=5\), then
\[
        \eta(G)=\frac{5+2b}{4}\Delta(u).
\]
The strict inequality forces \(b=1\) and \(\Delta(u)=1\), giving
\[
        \eta(G)=\frac74.
\]
Together with the cyclic case, this gives exactly the asserted list.

Finally, all four values occur:
\[
        \eta(S_3)=\frac54,\qquad
        \eta(C_4)=\frac32,\qquad
        \eta(D_{10})=\frac74.
\]
Also \(ZM(3,50,2)\) has \(m=3\), \(b=1\), and
\(d(25)/2=3/2\).  The displayed expression for \(\eta(G)\) gives
\[
        \eta(ZM(3,50,2))=\frac{15}{8}.
\]
\end{proof}

\subsection{\texorpdfstring{The threshold \(\lambda<5/2\) and its arithmetic criterion}{The threshold lambda<5/2 and its arithmetic criterion}}

For the subgroup-count threshold, we begin with supersolvability.  The lemmas
following the next theorem establish the reductions needed for its proof.

\begin{theorem}
\label{thm:sup-solv-and-subgroups}
Let \(G\) be a finite group.  If
\[
        \lambda(G)<\frac52,
\]
then \(G\) is supersolvable.
\end{theorem}

For the proof of this theorem, a counterexample means a finite
non-supersolvable group \(G\) satisfying the hypothesis of
Theorem~\ref{thm:sup-solv-and-subgroups}, with \(t=\pi(G)\).  Thus
\[
        \sub(G)<5\cdot2^{t-1}.
\]
We first reduce a minimal counterexample to three possible orders.

\begin{lemma}\label{lem:sup-sub-order-reduction}
Let \(G\) be a counterexample of minimal order, and put \(t=\pi(G)\).  Then
\(t\ge4\).  Moreover, if
\[
        |G|=p_1^{\alpha_1}\cdots p_t^{\alpha_t},
        \qquad \alpha_1\ge\alpha_2\ge\cdots\ge\alpha_t,
\]
then one of the following holds:
\[
        |G|=p_1^3p_2\cdots p_t,\qquad
        |G|=p_1^2p_2^2p_3\cdots p_t,\qquad
        |G|=p_1^2p_2\cdots p_t.
\]
\end{lemma}

\begin{proof}
For \(t=1\), the group is nilpotent.  If \(t=2\), the hypothesis gives
\(\sub(G)<10\), and if \(t=3\), it gives \(\sub(G)<20\).  The classification
of groups with fewer than \(20\) subgroups by Betz--Nash
\cite{betz-nash} shows that no such group is a counterexample; \(A_4\) is the
boundary nonsupersolvable example for \(t=2\), and it has \(10\) subgroups.
Thus \(t\ge4\).

If all \(\alpha_i=1\), then \(G\) has squarefree order and is supersolvable.
Therefore some \(\alpha_i>1\).  By Theorem~\ref{thm:Richards},
\[
        \sub(G)\ge \cyc(G)\ge d(|G|).
\]
If \(\alpha_1\ge4\), then
\[
        d(|G|)\ge5\cdot2^{t-1},
\]
contrary to the hypothesis.  Thus \(\alpha_1\le3\).  If
\(\alpha_1=3\) and \(\alpha_2\ge2\), then
\[
        d(|G|)\ge4\cdot3\cdot2^{t-2}=6\cdot2^{t-1},
\]
again impossible.  Hence \(\alpha_1=3\) forces
\(\alpha_i=1\) for \(i\ge2\).  Finally, if
\(\alpha_1=\alpha_2=\alpha_3=2\), then
\[
        d(|G|)\ge27\cdot2^{t-3}>5\cdot2^{t-1}.
\]
The three displayed possibilities follow.
\end{proof}

We also use the following classification consequence.

\begin{lemma}\label{lem:simple-four-squarefree}
Let \(S\) be a nonabelian simple group of order \(4m\), where \(m\) is odd and
squarefree.  Then
\[
        S\cong \PSL(2,q),
\]
where \(q\) is an odd prime and \((q-1)(q+1)/4\) is squarefree.  Moreover,
\(\Aut(S)\cong \operatorname{PGL}(2,q)\), and \(\Aut(S)\) has no prime divisor
outside \(|S|\).
\end{lemma}

\begin{proof}
A Sylow \(2\)-subgroup \(P\) of \(S\) has order \(4\).  It cannot be cyclic:
otherwise Burnside's normal \(2\)-complement theorem would give a nontrivial
proper normal subgroup of \(S\).  Thus \(P\cong C_2\times C_2\).
Walter's classification \cite{Walter1969}, in the form stated in
\cite[Theorem~10]{Craven2009}, leaves the possibilities
\[
\begin{gathered}
\PSL(2,2^f)\ (f\ge2),\qquad
\PSL(2,u)\ (u\ge5\text{ an odd prime power},\ u\equiv3,5\pmod8),\\
J_1,\qquad {}^2G_2(3^{2e+1})\ (e\ge1).
\end{gathered}
\]
The Sylow \(2\)-subgroups of \(J_1\) and of the Ree groups in this list have
order \(8\) \cite[Section~3]{Craven2009}, so neither can occur.  Since
\(|\PSL(2,2^f)|=2^f(2^{2f}-1)\), the even-field case forces \(f=2\), and
\(\PSL(2,4)\cong A_5\cong\PSL(2,5)\).  In the odd-field case, write
\(u=p^f\).  The factor \(p^f\) in
\(|\PSL(2,p^f)|=p^f(p^{2f}-1)/2\) divides the squarefree integer \(m\),
so \(f=1\).

Hence in every case \(S\cong\PSL(2,q)\) for an odd prime \(q\ge5\).
The order identity \(4m=q(q^2-1)/2\) gives
\[
        \frac{(q-1)(q+1)}4=\frac{2m}{q},
\]
which is squarefree because \(m/q\) is odd and squarefree.
Finally, the automorphisms of \(\PSL(2,p^f)\) are generated by the
\(\operatorname{PGL}(2,p^f)\)-action and field automorphisms
\cite{White2013}.  Since \(q\) is prime, there are no nontrivial field
automorphisms, and
\[
        \Aut(S)\cong\operatorname{PGL}(2,q),\qquad
        |\Aut(S)|=q(q^2-1)=2|S|.
\]
As \(2\mid |S|\), this introduces no new prime divisor.
\end{proof}

\begin{lemma}\label{lem:sup-sub-normal-prime-sylow}
A minimal counterexample has no normal Sylow subgroup of prime order.
\end{lemma}

\begin{proof}
Let \(G\) be a minimal counterexample, let \(t=\pi(G)\), and suppose that
\(P\) is a normal Sylow subgroup of prime order.  By Schur--Zassenhaus,
\[
        G=P\rtimes H,\qquad \pi(H)=t-1.
\]
If \(H\) is non-supersolvable, then the minimality of \(G\) gives
\[
        \sub(H)\ge5\cdot2^{t-2}.
\]
The subgroups \(L\le H\) and \(PL\), with \(L\le H\), form two disjoint
families of subgroups of \(G\), and the second family has cardinality
\(\sub(H)\) by Lemma~\ref{lem:distinct-Sub-semi-distinct-Sub}.  Therefore
\[
        \sub(G)\ge2\sub(H)\ge5\cdot2^{t-1},
\]
a contradiction.  If \(H\) is supersolvable, then \(G/P\cong H\) is
supersolvable and \(P\) is cyclic of prime order, whence \(G\) is
supersolvable, again a contradiction.
\end{proof}

\begin{lemma}\label{lem:sup-sub-double-square}
A minimal counterexample cannot have order
\[
        p_1^2p_2^2p_3\cdots p_t.
\]
\end{lemma}

\begin{proof}
For this order pattern, no Sylow subgroup is generalized quaternion.  Since \(G\) is not
supersolvable, Theorem~\ref{thm:sup-solv-criterion-cyclic-subgroups} gives
\[
        \sub(G)\ge\cyc(G)
        \ge d(|G|)\frac43
        =3^2\cdot2^{t-2}\cdot\frac43
        =6\cdot2^{t-1},
\]
contrary to the hypothesis.
\end{proof}

\begin{lemma}\label{lem:sup-sub-cube}
A minimal counterexample cannot have order
\[
        p_1^3p_2\cdots p_t.
\]
\end{lemma}

\begin{proof}
If the Sylow \(p_1\)-subgroup is not generalized quaternion, then
Theorem~\ref{thm:sup-solv-criterion-cyclic-subgroups} gives
\[
        \sub(G)\ge\cyc(G)
        \ge d(|G|)\frac64
        =4\cdot2^{t-1}\cdot\frac64
        =6\cdot2^{t-1},
\]
a contradiction.  Hence \(p_1=2\), and the Sylow \(2\)-subgroup \(P_1\) is
isomorphic to \(Q_8\).

Assume first that \(G\) has a normal Sylow subgroup.  If this subgroup is
\(P_r\) with \(r>1\), then it has prime order, contradicting
Lemma~\ref{lem:sup-sub-normal-prime-sylow}.

It remains, under the assumption that some Sylow subgroup is normal, to consider
the case where the normal Sylow subgroup is \(P_1\cong Q_8\).  Then
Schur--Zassenhaus gives \(G=P_1\rtimes H\), where \(|H|=p_2\cdots p_t\).  By
Theorem~\ref{thm:Richards},
\[
        \cyc(G)\ge d(|G|)=4\cdot2^{t-1}.
\]
Also \(H\) has at least \(d(|H|)=2^{t-1}\) subgroups.  For every \(L\le H\), the
subgroup \(P_1L\) is non-cyclic, and these subgroups are distinct by
Lemma~\ref{lem:distinct-Sub-semi-distinct-Sub}.  They are disjoint from the
cyclic subgroups counted above.  Thus
\[
        \sub(G)\ge4\cdot2^{t-1}+2^{t-1}=5\cdot2^{t-1},
\]
a contradiction.

We may therefore assume that \(G\) has no normal Sylow subgroup.  If \(G\) is
solvable and \(N\) is a minimal normal subgroup of \(G\), then \(N\) is
elementary abelian.  It cannot have odd order, because the odd Sylow subgroups
have prime order and no Sylow subgroup is normal.  It also cannot have order
\(4\), because \(Q_8\) contains no subgroup isomorphic to Klein's four-group.
Hence \(|N|=2\).  Then \(G/N\) is non-supersolvable; otherwise \(G\) would be
supersolvable.  Since \(\pi(G/N)=t\), the minimality of \(G\) gives
\[
        \sub(G)\ge\sub(G/N)\ge5\cdot2^{t-1},
\]
a contradiction.

Thus \(G\) is nonsolvable.  The same quotient argument rules out abelian
minimal normal subgroups.  Hence a minimal normal subgroup has the form
\(S^m\), where \(S\) is nonabelian simple.  Every nonabelian simple group has
order divisible by \(4\): this follows from Feit--Thompson and Burnside's
normal \(p\)-complement theorem.  Since \(|G|_2=8\), we must have \(m=1\).  Let
\(S\) be this simple normal subgroup.  Its Sylow \(2\)-subgroup has order \(4\)
or \(8\).  The order \(8\) case is ruled out by Brauer--Suzuki, since then the
Sylow \(2\)-subgroup is generalized quaternion.  In the order \(4\) case, the
Sylow \(2\)-subgroup of \(S\) is cyclic, being contained in \(Q_8\), and all odd
Sylow subgroups of \(S\) have prime order.  Thus all Sylow subgroups of \(S\)
are cyclic; hence \(S\) is supersolvable, a contradiction.
\end{proof}

\begin{lemma}\label{lem:sup-sub-square}
A minimal counterexample cannot have order
\[
        p_1^2p_2\cdots p_t.
\]
\end{lemma}

\begin{proof}
Suppose first that \(G\) has a normal Sylow subgroup.  If this subgroup is
\(P_r\) with \(r>1\), then it has prime order, contradicting
Lemma~\ref{lem:sup-sub-normal-prime-sylow}.

Now suppose that the normal Sylow subgroup is \(P_1\).  Then
\[
        G=P_1\rtimes H,\qquad |H|=|G|/p_1^2.
\]
If \(P_1\) is cyclic, let \(N\) be its unique subgroup of order \(p_1\).  Then
\(N\) is characteristic in \(P_1\), hence normal in \(G\), and \(G/N\) has
squarefree order.  Thus \(G/N\) is supersolvable, and since \(N\) is cyclic of
prime order, \(G\) is supersolvable, a contradiction.  Hence
\[
        P_1\cong C_{p_1}\times C_{p_1},
        \qquad \cyc(P_1)=p_1+2.
\]
Using Corollary~\ref{cor:cyc-nm}, Theorem~\ref{thm:product-cyclic-grps}, and
Theorem~\ref{thm:Richards} for \(H\), we obtain
\[
        \cyc(G)\ge \cyc(P_1)\cyc(H)
        \ge (p_1+2)\cdot2^{t-1}
        \ge4\cdot2^{t-1}.
\]
Moreover, \(H\) has at least \(d(|H|)=2^{t-1}\) subgroups.  For each
\(L\le H\), the subgroup \(P_1L\) is non-cyclic, and these subgroups are
distinct by Lemma~\ref{lem:distinct-Sub-semi-distinct-Sub}.  They are disjoint
from the cyclic subgroups counted above.  Therefore
\[
        \sub(G)\ge4\cdot2^{t-1}+2^{t-1}=5\cdot2^{t-1},
\]
a contradiction.

We may assume that no Sylow subgroup of \(G\) is normal.  If \(G\) is solvable
and \(N\) is a minimal normal subgroup, then \(N\) is elementary abelian.  It
cannot have order \(p_r\) for \(r>1\), and it cannot have order \(p_1^2\),
because either case would give a normal Sylow subgroup.  Hence \(|N|=p_1\).
Then \(G/N\) has squarefree order and is supersolvable, forcing \(G\) to be
supersolvable, a contradiction.  Thus \(G\) is nonsolvable.

The same quotient argument rules out abelian minimal normal subgroups.  Hence
a minimal normal subgroup has the form \(S^m\), where \(S\) is nonabelian
simple.  Since every nonabelian simple group has order divisible by \(4\), the
order condition forces \(p_1=2\) and \(m=1\).  Thus \(G\) has a nonabelian
simple normal subgroup \(S\) of order
\[
        |S|=4p_2\cdots p_k
        \qquad (3\le k\le t).
\]
By Lemma~\ref{lem:simple-four-squarefree}, \(S\cong \PSL(2,q)\), where \(q\) is
an odd prime and \((q-1)(q+1)/4\) is squarefree.

By Schur--Zassenhaus, \(G=S\rtimes_\varphi K\), where
\(|K|=p_{k+1}\cdots p_t\) and \(\varphi:K\to\Aut(S)\).  By
Lemma~\ref{lem:simple-four-squarefree}, \(\Aut(S)\) has no prime divisor
outside \(|S|\).  Since \((|S|,|K|)=1\), the homomorphism \(\varphi\) is
trivial, and \(G\cong S\times K\).  If \(K\ne1\), then the Sylow subgroup of
\(K\) for the largest prime divisor of \(|K|\) is normal in \(K\), and hence is
a normal Sylow subgroup of \(G\), a contradiction.  Therefore \(K=1\), and
\[
        G=S\cong \PSL(2,q).
\]
Since \(t\ge4\), we have \(q\ge7\).  Put
\[
        N_q=|\PSL(2,q)|=\frac{q(q^2-1)}2.
\]
Then \(t=\pi(N_q)\).  By Lemma~\ref{lem:ineq-primes},
\[
        q(q-1)\ge5\cdot2^t.
\]
By Lemma~\ref{lem:num-invol}, \(G\) has at least \(q(q-1)/2\) involutions, and
each involution generates a distinct cyclic subgroup.  Thus
\[
        \sub(G)\ge\cyc(G)\ge\frac{q(q-1)}2\ge5\cdot2^{t-1},
\]
the final contradiction.
\end{proof}

\begin{proof}[Proof of Theorem~\ref{thm:sup-solv-and-subgroups}]
Assume that a counterexample exists, and choose one of minimal order.  By
Lemma~\ref{lem:sup-sub-order-reduction}, its order has one of the three forms
listed there.  These three possibilities are excluded respectively by
Lemmas~\ref{lem:sup-sub-cube}, \ref{lem:sup-sub-double-square}, and
\ref{lem:sup-sub-square}.  This contradiction proves the theorem.
\end{proof}

\begin{lemma}\label{lem:sup-v4-q8-bound}
Let \(G\) be supersolvable.  If a Sylow \(2\)-subgroup of \(G\) is isomorphic
to \(V_4\) or to \(Q_8\), then
\[
        \sub(G)\ge 5\cdot 2^{\pi(G)-1}.
\]
\end{lemma}

\begin{proof}
We argue by induction on \(|G|\).  If \(G\) is a \(2\)-group, then
\(G\cong V_4\) or \(G\cong Q_8\), and the assertion follows from
\[
        \sub(V_4)=5,\qquad \sub(Q_8)=6.
\]

Suppose that \(G\) has an odd prime divisor.  Let \(q\) be the largest
prime divisor of \(|G|\), and let \(P\) be a Sylow \(q\)-subgroup of \(G\).
Since \(G\) is supersolvable, its ordered Sylow tower implies that
\(P\triangleleft G\) \cite{Scott}.  Thus \(P\) is a normal Hall subgroup,
and Schur--Zassenhaus gives
\[
        G=P\rtimes K.
\]
The subgroup \(K\) is supersolvable, its Sylow \(2\)-subgroup is isomorphic
to that of \(G\), and \(\pi(K)=\pi(G)-1\).
For every \(L\le K\), both \(L\) and \(PL\) are subgroups of \(G\).
The two families are disjoint, since \(P\ne1\) and \(P\cap K=1\).
Moreover, \(PL\cap K=L\), so the subgroups in the second family are
pairwise distinct.  By induction,
\[
        \sub(G)\ge 2\sub(K)
        \ge 2\cdot5\cdot2^{\pi(K)-1}
        =5\cdot2^{\pi(G)-1}.
\]
This proves the lemma.
\end{proof}

\begin{theorem}\label{prop:sub-bound-metacyclic}
Let \(G\) be a finite group.  If
\[
        \lambda(G)<\frac52,
\]
then every Sylow subgroup of \(G\) is cyclic.  In particular, \(G\) is
metacyclic.
\end{theorem}

\begin{proof}
Put \(t=\pi(G)\).  The hypothesis reads
\[
        \sub(G)<5\cdot2^{t-1}.
\]
By Theorem~\ref{thm:sup-solv-and-subgroups}, the group \(G\) is
supersolvable.  Suppose that \(G\) has a non-cyclic Sylow \(p\)-subgroup \(P\).
Write
\[
        |G|=p^r m,\qquad (p,m)=1.
\]

Assume first that \(P\) is not generalized quaternion.  By
Theorem~\ref{thm:amiri-jaco} and Lemma~\ref{lem:strong-dom-implies-less-cyc},
\[
        \cyc(G)\ge \cyc(C_{|G|/p}\times C_p).
\]
Since
\[
        C_{|G|/p}\times C_p
        \cong C_m\times C_{p^{r-1}}\times C_p,
\]
and the factors \(C_m\) and \(C_{p^{r-1}}\times C_p\) have coprime orders,
Theorem~\ref{thm:product-cyclic-grps} gives
\[
        \cyc(C_{|G|/p}\times C_p)
        =d(m)\,\cyc(C_{p^{r-1}}\times C_p).
\]
A direct count in \(C_{p^{r-1}}\times C_p\) gives
\[
        \cyc(C_{p^{r-1}}\times C_p)=(r-1)p+2.
\]
Indeed, besides the trivial subgroup, there are \(p+1\) cyclic subgroups of
order \(p\), and for each \(2\le j\le r-1\) there are \(p\) cyclic subgroups of
order \(p^j\).  Since \(d(m)\ge2^{t-1}\) and \(\sub(G)\ge\cyc(G)\), the
strict hypothesis forces
\[
        (r-1)p+2<5.
\]
Thus \(p=2\) and \(r=2\), and \(P\cong V_4\).

Now assume that \(P\) is generalized quaternion.  Then \(p=2\) and
\(|P|=2^r\) with \(r\ge3\).  Theorem~\ref{thm:Richards} gives
\[
        \sub(G)\ge\cyc(G)\ge d(|G|)\ge (r+1)2^{t-1}.
\]
The strict hypothesis therefore forces \(r=3\), and hence \(P\cong Q_8\).

We have shown that any non-cyclic Sylow subgroup of \(G\) must be a Sylow
\(2\)-subgroup isomorphic to \(V_4\) or to \(Q_8\).  This is impossible by
Lemma~\ref{lem:sup-v4-q8-bound}, because \(G\) is supersolvable and satisfies
\(\sub(G)<5\cdot2^{t-1}\).  Hence every Sylow subgroup of \(G\) is cyclic.

Finite groups with cyclic Sylow subgroups are metacyclic \cite{Scott}.
\end{proof}

Theorem~\ref{prop:sub-bound-metacyclic} and the subgroup-count formula now
give an exact arithmetic criterion for the \(ZM(m,n,r)\) parameters below
\(\lambda<5/2\).

\begin{corollary}\label{cor:sub-bound-zm-classification}
Let \(G\) be a finite group.  Then
\[
        \lambda(G)<\frac52
\]
if and only if \(G\cong ZM(m,n,r)\), where
\[
        (m,n)=(m,r-1)=1,\qquad r^n\equiv1\pmod m,
\]
and
\[
        \sum_{m_1\mid m}\sum_{n_1\mid n}
        \gcd(m_1,R(n_1))
        <5\cdot2^{\pi(mn)-1}.
\]
Here \(R(n_1)\) is the geometric sum in Proposition~\ref{prop:zm-counts}.
\end{corollary}

\begin{proof}
By Theorem~\ref{prop:sub-bound-metacyclic},
every group with \(\lambda<5/2\) has cyclic Sylow subgroups, and hence is a
\(ZM\)-group.  The displayed
inequality is exactly the condition obtained by applying
Proposition~\ref{prop:zm-counts}.  The converse is immediate from the same
proposition.  This includes \(m=1\), with \(r=1\): the inequality then
reduces to \(d(n)<5\cdot2^{\pi(n)-1}\) for the cyclic group \(C_n\).
\end{proof}

\begin{theorem}\label{thm:small-lambda-values}
Let \(G\) be a finite group.  If
\[
        1<\lambda(G)<\frac52,
\]
then
\[
        \lambda(G)\in\left\{\frac32,\,2,\,\frac94\right\}.
\]
All three values occur.
\end{theorem}

\begin{proof}
By Corollary~\ref{cor:sub-bound-zm-classification}, the hypothesis
\(\lambda(G)<5/2\) implies that \(G\) has cyclic Sylow subgroups.  If \(G\) is
cyclic and
\[
        |G|=\prod_i p_i^{a_i},
\]
then
\[
        \lambda(G)=\frac{d(|G|)}{2^{\pi(G)}}
        =\prod_i\frac{a_i+1}{2}.
\]
The only values of this product in the interval \((1,5/2)\) are
\[
        \frac32,\qquad 2,\qquad \frac94.
\]
Indeed, an exponent \(a_i\ge4\) already gives a factor at least \(5/2\); an
exponent \(3\) together with any other exponent larger than \(1\) gives a
factor at least \(3\); and three exponents equal to \(2\) give
\((3/2)^3>5/2\).  Thus the cyclic case is settled.

Assume now that \(G\) is not cyclic.  Write
\[
        G=ZM(m,n,r)
\]
with \(m,n>1\), and put \(t=\pi(mn)\).  For each prime \(p\mid m\),
let \(d_p>1\) be the order of \(r\) modulo \(p\), and fix a prime
\(q_p\mid d_p\).  In particular \(q_p\mid n\), and \(q_p\ne p\).

In the subgroup sum of Proposition~\ref{prop:zm-counts}, restrict to the
\(2^t\) pairs
\[
        m_1\mid\operatorname{rad}(m),\qquad
        n_1\mid\operatorname{rad}(n).
\]
Each pair contributes at least \(1\).  For such a pair define
\[
        E(m_1,n_1)=\{p\mid m:p\mid m_1,\ q_p\nmid n_1\}.
\]
If \(p\in E(m_1,n_1)\), then \(d_p\nmid n_1\), so
\(p\mid R(n_1)\).  Thus, including the empty-set case,
\[
\begin{aligned}
\gcd(m_1,R(n_1))-1
 &\ge\prod_{p\in E(m_1,n_1)}p-1\\
 &\ge\sum_{p\in E(m_1,n_1)}(p-1).
\end{aligned}
\]
The second inequality follows by expanding
\(\prod_{p\in E}(1+(p-1))\).  For each fixed \(p\mid m\), precisely
\(2^{t-2}\) of the selected pairs have \(p\mid m_1\) and
\(q_p\nmid n_1\).  Summing the displayed termwise inequality therefore
justifies adding all these excesses, even when their sets of pairs overlap:
\[
        \sub(G)\ge2^t+2^{t-2}\sum_{p\mid m}(p-1),\qquad
        \lambda(G)\ge1+\frac14\sum_{p\mid m}(p-1).
\]
Since \(\lambda(G)<5/2\), the sum \(\sum_{p\mid m}(p-1)\) is less than \(6\).
Also \(2\nmid m\), because \((m,r-1)=1\).  Thus \(m\) is divisible by exactly
one prime, and that prime is \(3\) or \(5\).

We next show that this prime occurs only to the first power.  Suppose
\(m=p^a\), with \(p\in\{3,5\}\), and choose \(q\mid d_p\).  Write
\(n=q^b v\), where \(q\nmid v\).  Restrict the outer sum in the local
factorization of Proposition~\ref{prop:zm-counts} to
\[
        n_1=q^{\varepsilon b}w,\qquad
        \varepsilon\in\{0,1\},\quad w\mid\operatorname{rad}(v).
\]
These divisors are distinct.  Since \(q\nmid w\), we have
\(h_p(w)=\sum_{i=0}^a p^i\), whereas \(h_p(q^bw)\ge a+1\).
Consequently
\[
        \sub(G)\ge2^{\pi(v)}\left(\sum_{i=0}^a p^i+a+1\right),\qquad
        \lambda(G)\ge\frac{1+p+\cdots+p^a+(a+1)}4.
\]
For \(p=3\) or \(p=5\), this lower bound is already at least \(5/2\) when
\(a\ge2\).  Hence
\[
        m=p\in\{3,5\}.
\]

Let \(d\) be the order of \(r\) modulo \(p\).  If \(p=3\), then \(d=2\).  If
\(p=5\), then \(d=2\) or \(4\).  The case \(d=4\) is impossible under the
strict bound.  Indeed, then write \(n=2^b v\) with \(b\ge2\) and \(v\) odd.
For each \(w\mid\operatorname{rad}(v)\), the three distinct divisors
\(n_1=w,2w,4w\) have local weights
\(h_p(n_1)=p+1,p+1,2\), respectively.  These choices give
\[
        \lambda(G)\ge
        \frac{2^{\pi(v)}(2p+4)}{2^{\pi(v)+2}}=\frac{2p+4}{4}.
\]
For \(p=5\) this is \(7/2\).  Thus in all remaining cases
\[
        r\equiv -1\pmod p.
\]

Write
\[
        n=2^b u,
        \qquad b\ge1,\qquad u\text{ odd}.
\]
Since \(d_p=2\), the local factorization in
Proposition~\ref{prop:zm-counts} gives \(h_p(n_1)=p+1\) for odd \(n_1\)
and \(h_p(n_1)=2\) for even \(n_1\).  Counting these divisors gives
\[
        \sub(G)=(p+1)d(u)+2bd(u)=(p+2b+1)d(u).
\]
Consequently
\[
        \lambda(G)
        =\frac{p+2b+1}{4}\cdot
        \frac{d(u)}{2^{\pi(u)}}.
\]
Put
\[
        \Delta(u)=\frac{d(u)}{2^{\pi(u)}}.
\]
Then \(\Delta(u)\ge1\), and the only values of \(\Delta(u)\) below \(5/3\)
are \(1\) and \(3/2\), while the only value below \(5/4\) is \(1\).

If \(p=3\), then
\[
        \lambda(G)=\frac{b+2}{2}\Delta(u).
\]
The condition \(\lambda(G)<5/2\) leaves only the following possibilities:
\[
        b=1,\ \Delta(u)=1;\qquad
        b=1,\ \Delta(u)=\frac32;\qquad
        b=2,\ \Delta(u)=1.
\]
These give respectively
\[
        \lambda(G)=\frac32,\qquad \frac94,\qquad 2.
\]
If \(p=5\), then
\[
        \lambda(G)=\frac{b+3}{2}\Delta(u).
\]
The strict inequality forces \(b=1\) and \(\Delta(u)=1\), and hence
\[
        \lambda(G)=2.
\]
This proves the asserted list of possible values.

Finally, all three values occur.  For example,
\[
        \lambda(C_4)=\frac32,\qquad
        \lambda(C_8)=2,\qquad
        \lambda(C_{36})=\frac94.
\]
They also occur among nonabelian \(ZM\)-groups:
\[
        ZM(3,2,2),\qquad ZM(3,4,2),\qquad ZM(3,50,2)
\]
have normalized subgroup counts \(3/2,2,\) and \(9/4\), respectively.
\end{proof}

\begin{question}\label{ques:next-normalized-values}
Are there only finitely many normalized values in the intervals
\[
        2\le\eta(G)<4,\qquad \frac52\le\lambda(G)<\frac{59}{8}?
\]
\end{question}

\section{\texorpdfstring{The subgroup-count solvability threshold \(\lambda<59/8\)}{The subgroup-count solvability threshold lambda<59/8}}
\label{sec:solv}

We next prove the subgroup-count solvability threshold \(\lambda(G)<59/8\).
The sharp example is \(A_5\), which has \(59\) subgroups.  As before,
multiplying by cyclic groups of new prime order preserves the normalized value
\(\lambda(G)\), which leads to the scale \(59\cdot2^{\pi(G)-3}\).

The proof is a minimal-counterexample argument.  First we remove normal
\(p\)-subgroups and normal Sylow subgroups.  Then the socle is a product of
nonabelian simple groups.  The numerical hypothesis restricts the
prime-power pattern of \(|G|\), and this reduces the possible simple factors to
a short list.  The final step treats the almost simple case.

\begin{proposition}\label{suzuki-prop}
The Suzuki groups \(Sz(2^{2n+1})\), with \(n\ge1\), are the only nonabelian
finite simple groups whose orders are not divisible by \(3\)
\cite{suzuki-groups-not-divisible-by-3}.
\end{proposition}

\begin{lemma}\label{almost-simple-lemma}
If \(G\) is a finite group with a unique minimal normal subgroup \(S\), and
\(S\) is nonabelian simple, then \(G\) is almost simple.
\end{lemma}

\begin{proof}
By the \(N/C\)-theorem, \(G/C_G(S)\) embeds in \(\Aut(S)\).  The centralizer
\(C_G(S)\) is normal in \(G\).  If it were nontrivial, it would contain a
minimal normal subgroup of \(G\), necessarily equal to \(S\).  This is
impossible because \(S\) is nonabelian.  Hence \(C_G(S)=1\), and conjugation
gives
\[
        S\le G\le \Aut(S).
\]
\end{proof}

\begin{lemma}\label{lem:solv-minimal-reductions}
Let \(G\) be a nonsolvable group of minimal order among all groups satisfying
\[
        \sub(G)<59\cdot2^{\pi(G)-3}.
\]
Put \(t=\pi(G)\), and assume \(t\ge4\).  Then:
\begin{enumerate}[label=(\alph*)]
\item \(G\) has no normal Sylow subgroup;
\item \(G\) has no nontrivial normal \(p\)-subgroup;
\item every minimal normal subgroup of \(G\) is a direct product \(S^k\), where
      \(S\) is nonabelian simple;
\item \(G\) has no nontrivial normal subgroup of odd order, and a Sylow
      \(2\)-subgroup of \(G\) is not generalized quaternion.
\end{enumerate}
\end{lemma}

\begin{proof}
Suppose first that \(P\) is a normal Sylow subgroup of \(G\).  Then
Schur--Zassenhaus gives \(G=P\rtimes H\), with \(H\cong G/P\).  The quotient
\(H\) is nonsolvable; otherwise \(G\) would be solvable.  Since \(H\) has
smaller order than \(G\), the minimal choice of \(G\) gives
\[
        \sub(H)\ge 59\cdot2^{\pi(H)-3}
        =59\cdot2^{t-4}.
\]
For each subgroup \(L\le H\), the subgroups \(L\) and \(PL\) of \(G\) are
distinct, and the subgroups \(PL\) are pairwise distinct by
Lemma~\ref{lem:distinct-Sub-semi-distinct-Sub}.  Hence
\[
        \sub(G)\ge2\sub(H)\ge59\cdot2^{t-3},
\]
a contradiction.  This proves (a).

Let \(N\) be a nontrivial normal \(p\)-subgroup of \(G\).  If \(N\) is Sylow,
then this contradicts (a).  Otherwise \(\pi(G/N)=t\).  The quotient \(G/N\) is
nonsolvable, and the quotient map gives \(\sub(G/N)\le\sub(G)\).  By
minimality,
\[
        \sub(G)\ge \sub(G/N)\ge59\cdot2^{t-3},
\]
again a contradiction.  This proves (b).

A minimal normal subgroup of a finite group is a direct product of isomorphic
simple groups.  Part (b) excludes elementary abelian minimal normal subgroups,
and hence every minimal normal subgroup has the form \(S^k\) with \(S\)
nonabelian simple.  This proves (c).

If \(G\) had a nontrivial normal subgroup \(M\) of odd order, then \(M\) would
be solvable by the Feit--Thompson theorem \cite{FeitThompson1963}.  It would therefore contain a
minimal normal subgroup of \(G\), necessarily elementary abelian, contradicting
(b).  Finally, if a Sylow \(2\)-subgroup of \(G\) were generalized quaternion,
then \(O_{2'}(G)=1\), where \(O_{2'}(G)\) is the largest normal subgroup
of odd order, and the Brauer--Suzuki theorem
\cite{BrauerSuzuki1959} would force \(G\) to have a central involution, again
contradicting (b).  This proves (d).
\end{proof}

\begin{lemma}\label{lem:solv-order-patterns}
Let \(G\) be as in Lemma~\ref{lem:solv-minimal-reductions}, and write
\[
        |G|=p_1^{\alpha_1}\cdots p_t^{\alpha_t},
        \qquad \alpha_1\ge\alpha_2\ge\cdots\ge\alpha_t.
\]
Then \(|G|\) has one of the forms listed in Table~\ref{poss-order}.
\end{lemma}

\begin{proof}
The group \(G\) is nonsolvable, hence not supersolvable.  By
Lemma~\ref{lem:solv-minimal-reductions}, its Sylow \(2\)-subgroup is not
generalized quaternion.  Moreover its Sylow \(2\)-subgroup is not cyclic:
otherwise Burnside's normal \(p\)-complement theorem would give a nontrivial
normal subgroup of odd order \cite{Isac-ams}, contradicting
Lemma~\ref{lem:solv-minimal-reductions}.  Thus Theorem~\ref{thm:sup-solv-criterion-cyclic-subgroups}
applies by contraposition: since not all Sylow subgroups of \(G\) are cyclic,
\[
        \sub(G)\ge \cyc(G)\ge
        d(|G|)\min_{\alpha_i>1}\left\{\frac{2\alpha_i}{\alpha_i+1}\right\}.
\]
Combining this with \(\sub(G)<59\cdot2^{t-3}\) gives
\[
        d(|G|)\min_{\alpha_i>1}\left\{\frac{2\alpha_i}{\alpha_i+1}\right\}
        <59\cdot2^{t-3}.
\]

Let \(\alpha_1,\ldots,\alpha_\ell\) be the exponents larger than \(1\), still
ordered decreasingly.  Since \(2a/(a+1)\) is increasing for \(a\ge1\), the
minimum is attained at \(\alpha_\ell\), and the last inequality becomes
\[
        2\alpha_\ell\prod_{i<\ell}(\alpha_i+1)2^{t-\ell}
        <59\cdot2^{t-3}.
\]
If \(\ell=1\), this gives \(\alpha_1\le7\).  In this case the unique repeated
prime is \(2\), by the preceding paragraph.  If \(\ell=2\), it gives either
\((\alpha_1,\alpha_2)=(a,2)\) with \(2\le a\le6\), or
\((\alpha_1,\alpha_2)=(3,3)\).  If \(\ell=3\), it gives
\((\alpha_1,\alpha_2,\alpha_3)=(a,2,2)\) with \(2\le a\le3\).  If
\(\ell=4\), it gives \((2,2,2,2)\).  No solution is possible for
\(\ell\ge5\).  These are exactly the five rows of Table~\ref{poss-order}.
\end{proof}

\begin{table}[ht]
    \centering
    \begin{tabular}{@{}cl@{}}
\toprule
Case & Order \\
\midrule
1 & \(2^{\alpha_1}p_2p_3\cdots p_t\), with \(2\le\alpha_1\le7\) \\
2 & \(p_1^{\alpha_1}p_2^2p_3\cdots p_t\), with \(2\le\alpha_1\le6\) \\
3 & \(p_1^{\alpha_1}p_2^2p_3^2p_4\cdots p_t\), with \(2\le\alpha_1\le3\) \\
4 & \(p_1^3p_2^3p_3\cdots p_t\) \\
5 & \(p_1^2p_2^2p_3^2p_4^2p_5\cdots p_t\) \\
\bottomrule
    \end{tabular}
    \caption{Possible orders of a minimal counterexample.}
    \label{poss-order}
\end{table}

\begin{lemma}\label{lem:solv-simple-minimal-normal}
Let \(G\) satisfy the hypotheses of Lemma~\ref{lem:solv-minimal-reductions}.
Then every minimal normal subgroup of \(G\) is nonabelian simple.
\end{lemma}

\begin{proof}
Let \(N\cong S^k\) be a minimal normal subgroup, as in
Lemma~\ref{lem:solv-minimal-reductions}.  If \(k\ge2\), then \(2^4\mid |N|\):
indeed \(4\mid |S|\) for every nonabelian simple group, by Feit--Thompson
\cite{FeitThompson1963} and Burnside's normal \(p\)-complement theorem
\cite{Isac-ams}.  Also, by Burnside's
\(p^aq^b\)-theorem, \(|S|\) has at least two odd prime divisors, and both of
them occur with exponent at least \(2\) in \(|S|^k\).  No order in
Table~\ref{poss-order} has these divisibility properties.  Hence \(k=1\).
\end{proof}

\begin{lemma}\label{lem:solv-possible-simple-socles}
Let \(G\) satisfy the hypotheses of Lemma~\ref{lem:solv-minimal-reductions},
and let \(S\) be a minimal normal subgroup of \(G\).  Then \(S\) is one of the
groups listed in Table~\ref{poss-S}.
\end{lemma}

\begin{proof}
By Lemma~\ref{lem:solv-simple-minimal-normal}, \(S\) is nonabelian simple, and
\(|S|\) divides \(|G|\).  Table~\ref{poss-order} therefore gives the following
three obstructions, for any distinct primes \(r,s\):
\[
        r^8\nmid |S|,\qquad r^4s^3\nmid |S|,\qquad r^7s^2\nmid |S|.
\]
We apply these to the families in the classification of finite simple groups.

If \(S=A_n\) with \(n\ge9\), then
\(2^6\cdot3^4\mid |A_9|\mid |A_n|\), contradicting the second obstruction.
Thus only \(A_5,A_6,A_7,A_8\) remain.  Among the sporadic groups, the
\textit{Atlas} orders \cite{ATLAS} show that the only ones not divisible by
\(2^8\) are
\[
        M_{11},\ M_{12},\ M_{22},\ M_{23},\ J_1,\ J_2,\ J_3,\ McL.
\]
Now \(2^6\cdot3^3\mid |M_{12}|\), and \(2^7\cdot3^2\) divides the order of
each of \(M_{22},M_{23},J_2,J_3,McL\).  The second and third obstructions
exclude these groups, leaving only \(M_{11}\) and \(J_1\).

For Lie type, write \(q=p^f\).  The order formulas
\cite[Theorem~1.6.5, Table~1.3, and Remark~1.6.6(a)]{GeckMalle2016}
give the defining-characteristic exponents below.  Here the parameter in
\({}^2B_2(q),{}^2G_2(q),{}^2F_4(q)\) is the field size, the square of the
parameter used for those three families in the cited table.  Passing from the
simply connected group to its central quotient does not change the \(p\)-part.
The exceptional small derived groups are treated separately below.
\[
\begin{array}{c|c|l}
\text{type} & v_p(|S|) & \text{possibilities with }v_p(|S|)\le7\\ \hline
A_{n-1},\ {}^2A_{n-1} & fn(n-1)/2
    & n=2;\ n=3,\ f\le2;\ n=4,\ f=1\\
B_n,\ C_n\ (n\ge2) & fn^2 & n=2,\ f=1\\
D_n,\ {}^2D_n\ (n\ge4) & fn(n-1) & \text{none}\\
G_2 & 6f & f=1\\
{}^2B_2\ (f\ge3\text{ odd}) & 2f & f=3\\
{}^2G_2\ (f\ge3\text{ odd}) & 3f & \text{none}\\
{}^3D_4,\ {}^2F_4 & 12f & \text{none}\\
F_4,\ E_6,\ {}^2E_6,\ E_7,\ E_8 & Nf,\ N\ge24 & \text{none}
\end{array}
\]
The low-rank identifications \(B_2=C_2\), \(D_3=A_3\), and
\({}^2D_3={}^2A_3\) introduce no further families.  The rank-one family is
\(\PSL(2,q)\).  Apart from it and the small derived groups, the only remaining
possibilities are
\[
\PSL(3,p^f),\ \PSU(3,p^f)\ (f\le2),\quad
\PSL(4,p),\ \PSU(4,p),\ \PSp(4,p),\ G_2(p),\ Sz(8).
\]

We next exclude all odd characteristics in the six classical or \(G_2\)
families in this display, uniformly in \(p\).  For odd \(q\), put
\(a=v_2(q-1)\) and \(b=v_2(q+1)\).  Then \(\min(a,b)=1\) and \(a+b\ge3\).
The order formulas, including the central divisors
\(\gcd(3,q-1),\gcd(3,q+1),\gcd(4,q-1),\gcd(4,q+1),2,1\), respectively,
give
\[
\begin{array}{c|c|c}
S & v_2(|S|) & \text{lower bound}\\ \hline
\PSL(3,q) & 2a+b & 4\\
\PSU(3,q) & a+2b & 4\\
\PSL(4,q) & 3a+2b+1-\min(2,a) & 7\\
\PSU(4,q) & 2a+3b+1-\min(2,b) & 7\\
\PSp(4,q) & 2a+2b & 6\\
G_2(q) & 2a+2b & 6
\end{array}
\]
For example, these valuations follow from
\(v_2(q^2+1)=1\), \(v_2(q^3-1)=a\), and \(v_2(q^3+1)=b\).
Each of these groups also has defining-characteristic part divisible by
\(p^3\), so \(2^4p^3\mid |S|\), contradicting the second obstruction.

In characteristic \(2\), the remaining field sizes are consequently
\(q=2,4\) in dimension \(3\), and \(q=2\) in the other four families.
The group \(\PSU(4,2)\) has order \(2^6\cdot3^4\cdot5\), so is excluded.
The group \(\PSU(3,2)\) is solvable of order \(72\).  The small isomorphisms
\cite{ATLAS}
\[
        \PSL(3,2)\cong\PSL(2,7),\qquad
        \PSL(4,2)\cong A_8,\qquad
        \PSp(4,2)'\cong A_6
\]
place these cases in earlier rows.  Also,
\(G_2(2)'\cong\PSU(3,3)\) has order \(2^5\cdot3^3\cdot7\), and the Tits
group \({}^2F_4(2)'\) has order \(2^{11}\cdot3^3\cdot5^2\cdot13\);
both are excluded by the displayed obstructions.  Finally,
\({}^2G_2(3)'\cong\PSL(2,8)\) is already included, while \({}^2B_2(2)\)
is solvable \cite{ATLAS,MalleTesterman2011}.  Thus the only new possibilities
are
\[
\begin{array}{c|c}
S & |S|\\ \hline
\PSL(3,4) & 2^6\cdot3^2\cdot5\cdot7\\
\PSU(3,4) & 2^6\cdot3\cdot5^2\cdot13\\
Sz(8) & 2^6\cdot5\cdot7\cdot13
\end{array}
\]
Together with the alternating, sporadic, and rank-one cases above, these give
exactly the list of possibilities asserted in Table~\ref{poss-S}.
\end{proof}

\begin{table}[ht]
    \centering
    \begin{tabular}{@{}clcc@{}}
\toprule
Case & \(S\) & \(|S|\) & \(|\Out(S)|\)\\
\midrule
1  & \(M_{11}\) & \(2^4\cdot3^2\cdot5\cdot11\) & \(1\) \\
2  & \(J_1\) & \(2^3\cdot3\cdot5\cdot7\cdot11\cdot19\) & \(1\) \\
3  & \(A_5\) & \(2^2\cdot3\cdot5\) & \(2\) \\
4  & \(A_6\) & \(2^3\cdot3^2\cdot5\) & \(4\) \\
5  & \(A_7\) & \(2^3\cdot3^2\cdot5\cdot7\) & \(2\) \\
6  & \(A_8\) & \(2^6\cdot3^2\cdot5\cdot7\) & \(2\) \\
7  & \(\PSL(3,4)\) & \(2^6\cdot3^2\cdot5\cdot7\) & \(12\) \\
8  & \(\PSU(3,4)\) & \(2^6\cdot3\cdot5^2\cdot13\) & \(4\) \\
9  & \(Sz(8)\) & \(2^6\cdot5\cdot7\cdot13\) & \(3\) \\
10  & \(\PSL(2,q),\ q=p^f\) &
\(\begin{cases}
q(q^2-1),&p=2,\\
q(q^2-1)/2,&p\text{ odd}
\end{cases}\)
&
\(\begin{cases}
f,&p=2,\\
2f,&p\text{ odd}
\end{cases}\)\\
\bottomrule
    \end{tabular}
    \caption{Possible simple minimal normal subgroups.}
    \label{poss-S}
\end{table}

\begin{lemma}\label{lem:solv-unique-minimal-normal}
Let \(G\) satisfy the hypotheses of Lemma~\ref{lem:solv-minimal-reductions}.
Then \(G\) has a unique minimal normal subgroup.
\end{lemma}

\begin{proof}
Let \(S\) be a minimal normal subgroup.  Suppose first that
\(\pi(S)<\pi(G)\), and let \(\ell\) be a prime divisor of \(|G|\) not dividing
\(|S|\).  If a Sylow \(\ell\)-subgroup \(Q\) does not centralize \(S\), then
conjugation gives a nontrivial image of \(Q\) in \(\Out(S)\).  Thus
\(\ell\mid |\Out(S)|\).

Using Table~\ref{poss-S}, the only cases in which \(\Out(S)\) can contribute a
new prime are \(S\cong Sz(8)\), where the new prime is \(3\), and
\(S\cong \PSL(2,p^f)\), where the new prime divides \(f\).  Since
Table~\ref{poss-order} bounds every exponent in \(|G|\) by \(7\), the latter
case forces \(f=5\) or \(f=7\).  If \(f=7\), then the order patterns force
\(p=2\).  For \(S\cong Sz(8)\) and \(S\cong\PSL(2,128)\), the
\(2\)-parts of \(|S|\) are \(2^6\) and \(2^7\), respectively.
A distinct minimal normal subgroup would centralize \(S\), intersect it
trivially, and contribute at least another factor \(4\), contradicting
Table~\ref{poss-order}.  Thus \(S\) is unique in these two cases.
If \(f=5\), then either \(S\) is already forced to be the unique
minimal normal subgroup, or \(p=2\).  Indeed, if \(p\) is odd, then
\(|\PSL(2,p^5)|\) is divisible by \(p^5\) and by \(4\); another nonabelian
simple minimal normal subgroup would add another factor \(4\), producing
exponents incompatible with Table~\ref{poss-order}.  Thus the only remaining
case with \(f=5\) is
\[
        S\cong \PSL(2,32),\qquad
        |S|=2^5\cdot3\cdot11\cdot31.
\]
In this last case another minimal normal subgroup would have order \(4m\) with
\(m\) odd, squarefree, and not divisible by \(3\).  By
Lemma~\ref{lem:simple-four-squarefree}, the simple groups of order \(4m\) are
\(A_5\) and the groups \(\PSL(2,r)\) with \(r\) prime and
\((r-1)(r+1)/4\) squarefree.  All of them have order divisible by \(3\).  Hence
this last case is impossible.  Consequently, if \(S\) is not unique, then every
prime divisor of \(|G|\) outside \(|S|\) centralizes \(S\).

Assume now, toward a contradiction, that \(G\) has at least two minimal normal
subgroups.  Choose two of them, \(S\) and \(T\).  They centralize one another,
and hence \(S\times T\le G\).  They are not isomorphic; otherwise
Table~\ref{poss-order} is violated by the square exponents coming from two
copies of the same nonabelian simple group.

There cannot be a third minimal normal subgroup.  Indeed, three nonabelian
simple direct factors would force \(2^6\mid |G|\).  If all three factors have
order divisible by \(3\), then \(3^3\mid |G|\), and the simultaneous
divisibility by \(2^6\) and \(3^3\) is excluded by Table~\ref{poss-order}.  If
one factor has order not divisible by \(3\), then it is a Suzuki group by
Proposition~\ref{suzuki-prop}.  Such a factor already contributes \(2^6\), and
the other two nonabelian simple factors contribute another factor \(2^4\).
This would force \(2^{10}\mid |G|\), again impossible by
Table~\ref{poss-order}.

Therefore \(S\) and \(T\) are the only minimal normal subgroups.  The
centralizer of \(S\times T\) is trivial, and the \(N/C\)-theorem gives
\[
        S\times T\le G\le \Aut(S\times T)=\Aut(S)\times\Aut(T),
\]
Here the nonisomorphic nonabelian simple groups \(S\) and \(T\) have
trivial centers and no common nontrivial direct factor, so the last
equality follows from \cite[Theorem~3.2]{aut-direct}.
If \(S\times T<G\) and \(\pi(S\times T)=\pi(G)\), then the minimality of \(G\)
gives
\[
        \sub(G)\ge\sub(S\times T)\ge59\cdot2^{\pi(G)-3},
\]
a contradiction.  Thus \(G/(S\times T)\) introduces a new prime.  This prime
is outside both \(|S|\) and \(|T|\).  Since \(C_G(S\times T)=1\), a Sylow
subgroup for this prime acts nontrivially on \(S\) or on \(T\), and hence gives
a new prime divisor of \(\Out(S)\) or of \(\Out(T)\).  The first paragraph of
the proof then says that the corresponding minimal normal subgroup would have
to be unique, a contradiction.  Hence \(G=S\times T\).

Finally, by the minimal choice of \(G\), every nonsolvable group \(X\) of
smaller order satisfies
\[
        \sub(X)\ge59\cdot2^{\pi(X)-3}.
\]
Thus
\[
        \sub(G)\ge \sub(S)\sub(T)
        \ge 59^2\,2^{\pi(S)+\pi(T)-6}.
\]
Here the first inequality is obtained by taking direct products \(A\times B\)
with \(A\le S\) and \(B\le T\).
Since \(S\) and \(T\) both have even order,
\(\pi(G)\le\pi(S)+\pi(T)-1\), and the last lower bound is at least
\(59\cdot2^{\pi(G)-3}\), a contradiction.  Hence \(G\) has a unique minimal
normal subgroup.
\end{proof}

\begin{table}[ht]
\centering
\caption{Finite simple cases checked from tables of marks.}
\label{tab:solvable-finite-simple-checks}
\begin{tabular}{@{}lcc@{}}
\toprule
Groups \(S\) & largest bound needed & \(\min \sub(S)\)\\
\midrule
\(A_7,A_8,\PSL(3,4),M_{11},J_1,\PSU(3,4),Sz(8)\) & \(472\) & \(3786\)\\
\(\PSL(2,q)\), \(q=11,13,16,19,23,25,27,29,31,32,64\) & \(236\) & \(620\)\\
\bottomrule
\end{tabular}
\end{table}

The second column is chosen large enough for both the simple case and the
almost simple cases in which one outer automorphism prime is added.  The values
in Table~\ref{tab:solvable-finite-simple-checks} are obtained from the
GAP/TomLib tables of marks \cite{GAP,TomLib}: if \(T\) is the table of marks of
\(S\), then
\[
        \sub(S)=\sum \operatorname{LengthsTom}(T),
\]
the sum of the lengths of the conjugacy classes of subgroups.  The verifier
in Online Resource~1 checks every group in both rows separately,
including the displayed minima and the bounds computed from \(|\Aut(S)|\).
The complete successful reference output is archived there and confirms
both rows. The recorded GAP environment also includes AtlasRep
\cite{AtlasRep}, a dependency of TomLib.

\begin{theorem}
\label{thm:solvable}
Let \(G\) be a finite group.  If
\[
        \lambda(G)<\frac{59}{8},
\]
then \(G\) is solvable.
\end{theorem}

\begin{proof}
Put \(t=\pi(G)\).  The hypothesis reads
\[
        \sub(G)<59\cdot2^{t-3}.
\]
For \(t\le2\), the assertion follows from Burnside's
\(p^aq^b\)-theorem.  For \(t=3\), it follows from the theorem of
Das--Mandal that every group with fewer than \(59\) subgroups is solvable
\cite[arXiv version, Theorem~2.1]{das-com}.  Suppose therefore that a counterexample exists,
and choose one of minimal order.  Then \(t\ge4\).

The reductions above apply to \(G\).  By Lemmas~\ref{lem:solv-minimal-reductions}
and~\ref{lem:solv-unique-minimal-normal}, \(G\) has a unique minimal normal
subgroup \(S\), and \(S\) is nonabelian simple.  Hence
Lemma~\ref{almost-simple-lemma} gives
\[
        S\le G\le \Aut(S).
\]

If \(S<G\) and \(\pi(S)=\pi(G)\), then the minimality of \(G\) gives
\[
        \sub(G)\ge\sub(S)\ge59\cdot2^{t-3},
\]
a contradiction.  Thus either \(G=S\), or \(S<G\) and \(\pi(S)<\pi(G)\).

Assume first that \(S<G\) and \(\pi(S)<\pi(G)\).  By Table~\ref{poss-S}, the
extra prime must come from \(\Out(S)\).  Hence \(S\cong Sz(8)\), or
\[
        S\cong \PSL(2,32),\qquad S\cong \PSL(2,128),
        \qquad\text{or}\qquad S\cong \PSL(2,p^5)
\]
with \(p\) odd.  The cases \(S\cong Sz(8)\) and \(S\cong\PSL(2,32)\) are
covered by Table~\ref{tab:solvable-finite-simple-checks}, since
\(\sub(G)\ge\sub(S)\).
For \(S=\PSL(2,128)\), the involutions alone give
\[
        \sub(G)\ge\sub(S)\ge 128^2-1>59\cdot2^2,
\]
while here \(\pi(G)=5\).  Finally, if \(S=\PSL(2,p^5)\) with \(p\) odd, then
\(q=p^5\ge37\), and \(\pi(G)=\pi(S)+1\).  Lemmas~\ref{lem:num-invol}
and~\ref{lem:ineq-primes} give
\[
        \sub(G)\ge\sub(S)\ge \frac{q(q-1)}2
        \ge 59\cdot2^{\pi(S)-2}
        =59\cdot2^{\pi(G)-3},
\]
again a contradiction.

We are left with \(G=S\) simple.  By Table~\ref{poss-S}, either \(S\) is one of
the finite cases in Table~\ref{tab:solvable-finite-simple-checks}, or
\(S\cong\PSL(2,q)\).  The finite cases contradict the displayed table.  For
\(\PSL(2,q)\), the cases \(q=11,13,16,19,23,25,27,29,31,32,64\) are covered by
Table~\ref{tab:solvable-finite-simple-checks}; the cases
\(q=4,5,7,8,9,17\) have \(\pi(S)=3\) and were already covered by the base case.
If \(q=128\), the involution count gives
\[
        \sub(S)\ge q^2-1>59\cdot2^{\pi(S)-3}.
\]
If \(q\) is odd and \(q\ge37\), then Lemmas~\ref{lem:num-invol}
and~\ref{lem:ineq-primes} give
\[
        \sub(S)\ge \frac{q(q-1)}2
        \ge 59\cdot2^{\pi(S)-2}
        >59\cdot2^{\pi(S)-3}.
\]
This final contradiction proves the theorem.
\end{proof}

\begin{remark}\label{rem:solvable-tight}
The constant \(59/8\) in Theorem~\ref{thm:solvable} is tight.  The group
\(A_5\) is not solvable and satisfies
\[
        \lambda(A_5)=\frac{59}{8}.
\]
If \(p_1,\ldots,p_r\) are distinct primes larger than \(5\), then
\[
        A_5\times C_{p_1}\times\cdots\times C_{p_r}
\]
is not solvable and still has normalized value \(\lambda=59/8\).
\end{remark}

\section{\texorpdfstring{Reduced cores below \(\lambda<59/8\)}{Reduced cores below lambda<59/8}}
\label{sec:sub-threshold-structure}

Theorem~\ref{thm:solvable} reduces the study of groups with
\(\lambda(G)<59/8\) to solvable groups.  We first separate their coprime
cyclic direct factors and then give exact criteria for several families.
These criteria do not constitute a classification of all solvable groups
below the threshold.

The decomposition established in Corollary~\ref{cor:sub-block-decomposition}
has the form
\[
        G\cong A\times S\times C,
\]
with pairwise coprime factor orders.  Here \(A\) is cyclic of squarefree
order and does not affect \(\lambda\); \(C\) collects the remaining central
cyclic Sylow subgroups; and \(S\), if nontrivial, has no central cyclic Sylow
subgroup.  The numerical condition is
\[
        \lambda(S)\delta(C)<\frac{59}{8},\qquad
        \delta(C)=\frac{d(|C|)}{2^{\pi(C)}}.
\]

Call a finite group \(K\) squarefree-reduced if it has no central Sylow
subgroup of prime order.  A fundamental block for the subgroup-count threshold
is a nontrivial squarefree-reduced group with no central cyclic Sylow subgroup
and with \(\lambda<59/8\).  Squarefree reduction alone is weaker: it may leave
central cyclic Sylow subgroups of order \(\ell^a\), with \(a\ge2\).

We call a nontrivial finite group externally cyclically indecomposable if it
cannot be written as \(H\times D\), where \(H\ne1\), \(D\ne1\) is cyclic,
and \((|H|,|D|)=1\).  Every fundamental block has this property.  Cyclic
direct factors whose primes already divide the order of the complementary
factor are not removed by the coprime decomposition.

\begin{question}\label{ques:finitely-many-sub-blocks}
Are there only finitely many fundamental blocks \(S\) with
\[
        \lambda(S)<\frac{59}{8}?
\]
\end{question}

\begin{lemma}\label{lem:solv-threshold-reduced-core}
Let \(G\) be a finite group with \(t=\pi(G)\) and
\[
        \sub(G)<59\cdot2^{t-3}.
\]
Then
\[
        G\cong A\times K,
\]
where \(A\) is cyclic of squarefree order, \((|A|,|K|)=1\), the group \(K\) is
squarefree-reduced, and
\[
        \sub(K)<59\cdot2^{\pi(K)-3}.
\]
Conversely, every such direct product satisfies the inequality for \(G\) if
and only if \(K\) satisfies the displayed inequality.
\end{lemma}

\begin{proof}
Suppose that \(P\) is a central Sylow subgroup of prime order \(p\).  By
Schur--Zassenhaus, \(G=P\rtimes H\), where \(H\cong G/P\).  Since \(P\le Z(G)\),
the action is trivial and \(G=P\times H\).  Because the two factors have
coprime orders, every subgroup of \(G\) is a direct product of a subgroup of
\(P\) and a subgroup of \(H\).  Hence
\[
        \sub(G)=2\sub(H),
        \qquad
        \pi(G)=\pi(H)+1.
\]
Thus \(G\) satisfies the strict inequality if and only if \(H\) does.  Iterating
this argument removes all central Sylow subgroups of prime order and gives the
decomposition \(G\cong A\times K\).  The converse follows from the same
multiplicativity.
\end{proof}

\begin{lemma}\label{lem:external-cyclic-normal-form}
Let \(K\) be a finite group.  Let \(C\) be the product of all
central cyclic Sylow subgroups of \(K\).  Then \(C\) is cyclic, this choice of
\(C\) is canonical, and
\[
        K\cong S\times C,
\]
with \((|S|,|C|)=1\).  Moreover, either \(S=1\), or \(S\) has no central
cyclic Sylow subgroup and is externally cyclically indecomposable.  If \(K\) is squarefree-reduced and
\[
        C\cong C_{\ell_1^{a_1}}\times\cdots\times C_{\ell_r^{a_r}},
\]
with distinct primes \(\ell_i\), then every nontrivial exponent satisfies
\(a_i\ge2\).
\end{lemma}

\begin{proof}
If \(K=1\), take \(S=C=1\).  Otherwise, let \(p\) run over the primes for which the Sylow \(p\)-subgroup \(P_p\) of
\(K\) is cyclic and central, and put \(C=\prod P_p\).  The factors have
coprime orders and commute, hence \(C\) is cyclic.  Since \(C\) is a normal
Hall subgroup of \(K\), Schur--Zassenhaus gives a complement \(S\).  Because
\(C\le Z(K)\), this complement is normal, and therefore
\[
        K\cong S\times C.
\]
The subgroup \(C\) is canonical by construction.  Moreover, the complement is
unique: Schur--Zassenhaus conjugacy for complements applies because \(C\) is
solvable, and conjugation by elements of \(C\) is trivial.

If \(S\ne1\) and \(S\) had a central cyclic Sylow subgroup \(Q\), then \(Q\)
would also be a central cyclic Sylow subgroup of \(K\), because \(K=S\times C\)
and \((|S|,|C|)=1\).  This contradicts the definition of \(C\).  Thus \(S\)
has no central cyclic Sylow subgroup.

If \(S\ne1\), then \(S\) is externally cyclically indecomposable.  Indeed, a
decomposition \(S=H\times D\), with \(H\ne1\), \(D\ne1\) cyclic, and
\((|H|,|D|)=1\), would make each Sylow subgroup of \(D\) a central cyclic
Sylow subgroup of \(S\).

If \(K\) is squarefree-reduced and some \(a_i=1\), then the Sylow
\(\ell_i\)-subgroup of \(C\) is central of prime order in \(K\), contradicting
the definition of squarefree-reduced.
\end{proof}

For a cyclic group \(C\) with \(|C|=\prod_i \ell_i^{a_i}\), the factor
\(\delta(C)\) is
\[
        \delta(C)=\prod_i\frac{a_i+1}{2}.
\]

\begin{lemma}\label{lem:sub-cyclic-factor-over-block}
Let \(S\) be a finite group and let \(C\) be cyclic, with
\((|S|,|C|)=1\).  Then
\[
        \lambda(S\times C)=\lambda(S)\delta(C).
\]
Consequently \(S\times C\) satisfies \(\lambda<59/8\)
if and only if
\[
        \lambda(S)\delta(C)<\frac{59}{8}.
\]
\end{lemma}

\begin{proof}
Since the two factors have coprime order, subgroup counts multiply:
\[
        \sub(S\times C)=\sub(S)\sub(C)=\sub(S)d(|C|).
\]
Also \(\pi(S\times C)=\pi(S)+\pi(C)\).  Dividing by
\(2^{\pi(S)+\pi(C)}\) gives the formula.
\end{proof}

\begin{corollary}\label{cor:sub-block-decomposition}
Let \(G\) be a finite group satisfying
\[
        \sub(G)<59\cdot2^{\pi(G)-3}.
\]
Then
\[
        G\cong A\times S\times C,
\]
where \(A\) is cyclic of squarefree order, \(C\) is cyclic,
\[
        (|A|,|S||C|)=1,\qquad (|S|,|C|)=1,
\]
and either \(S=1\) or \(S\) is a fundamental block for the subgroup-count
threshold.  With the convention \(\lambda(1)=1\), the remaining condition is
\[
        \lambda(S)\delta(C)<\frac{59}{8}.
\]
Moreover, if
\[
        C\cong C_{\ell_1^{a_1}}\times\cdots\times C_{\ell_r^{a_r}},
\]
with distinct primes \(\ell_i\), then every nontrivial exponent satisfies
\(a_i\ge2\).

Conversely, every product with these properties satisfies the displayed
subgroup-count inequality.
When \(S=1\), this is the purely cyclic reduced core case; the possible
exponent patterns for \(C\) are exactly the cyclic patterns listed in
Corollary~\ref{cor:nilpotent-sub-primary-blocks}.
\end{corollary}

\begin{proof}
Apply Lemma~\ref{lem:solv-threshold-reduced-core} to write
\[
        G\cong A\times K,
\]
where \(A\) is cyclic squarefree, \((|A|,|K|)=1\), and \(K\) is
squarefree-reduced with \(\lambda(K)<59/8\).  Then apply
Lemma~\ref{lem:external-cyclic-normal-form} to \(K\):
\[
        K\cong S\times C,
\]
where \(C\) is cyclic, \((|S|,|C|)=1\), and either \(S=1\) or \(S\) has no
central cyclic Sylow subgroup.  Since \(K\) is squarefree-reduced, the same is
true of \(S\): otherwise a central Sylow subgroup of prime order in \(S\) would
also be a central Sylow subgroup of prime order in \(K\).  The statement about
the exponents of \(C\) is the last assertion of
Lemma~\ref{lem:external-cyclic-normal-form}.

Finally, Lemma~\ref{lem:sub-cyclic-factor-over-block} gives
\[
        \lambda(K)=\lambda(S)\delta(C),
\]
and the factor \(A\) is neutral for \(\lambda\).  Since
\(\delta(C)\ge1\) and \(\lambda(K)<59/8\), we have
\(\lambda(S)<59/8\).  Therefore, if \(S\ne1\), then \(S\) is a fundamental
block.  This proves the necessity.  The converse follows from the same equality
in Lemma~\ref{lem:sub-cyclic-factor-over-block}.
\end{proof}

\begin{corollary}\label{cor:sub-admissible-cyclic-decorations}
Let \(S\) be fixed with \(\lambda(S)<59/8\), and let
\[
        C\cong C_{\ell_1^{a_1}}\times\cdots\times C_{\ell_r^{a_r}},
        \qquad a_i\ge2,
\]
where the primes \(\ell_i\) are distinct and do not divide \(|S|\).  Then
\[
        S\times C
\]
satisfies \(\lambda<59/8\) if and only if
\[
        \lambda(S)\prod_{i=1}^r\frac{a_i+1}{2}<\frac{59}{8}.
\]
In particular, for a fixed block \(S\), there are only finitely many possible
exponent patterns \((a_1,\ldots,a_r)\) for non-squarefree coprime cyclic
factors.
\end{corollary}

\begin{proof}
This is Lemma~\ref{lem:sub-cyclic-factor-over-block} and the formula
\[
        \delta(C)=\prod_{i=1}^r\frac{a_i+1}{2}.
\]
The finiteness follows because every factor \((a_i+1)/2\) is at least \(3/2\)
and their product is bounded above by \(59/(8\lambda(S))\).
\end{proof}

\begin{corollary}\label{cor:sub-universal-decoration-patterns}
Let \(S\ne1\) be a fundamental block for the threshold \(\lambda<59/8\).  If
\[
        C\cong C_{\ell_1^{a_1}}\times\cdots\times C_{\ell_r^{a_r}},
        \qquad a_1\ge\cdots\ge a_r\ge2,
\]
is a cyclic group of order coprime to \(|S|\) such that \(S\times C\) still
satisfies \(\lambda<59/8\), then one of the following
holds:
\[
\begin{array}{c|c}
r&\text{possible exponent patterns}\\
\hline
0&\text{empty}\\
1&(a),\ 2\le a\le8\\
2&(a,2),\ 2\le a\le5,\quad\text{or }(3,3)\\
3&(3,2,2)\text{ or }(2,2,2).
\end{array}
\]
No admissible decoration has \(r\ge4\).
\end{corollary}

\begin{proof}
Since \(S\) is squarefree-reduced and nontrivial, it is not cyclic of
squarefree order.  Corollary~\ref{cor:initial-threshold-lower-bounds} gives
\[
        \sub(S)\ge6\cdot2^{\pi(S)-2},
\]
and hence \(\lambda(S)\ge3/2\).  Therefore any admissible decoration satisfies
\[
        \prod_{i=1}^r\frac{a_i+1}{2}
        <\frac{59}{8\lambda(S)}
        \le \frac{59}{12}.
\]
Since every \(a_i\ge2\), the factor \((a_i+1)/2\) is at least \(3/2\), which
already excludes \(r\ge4\).  The displayed list is obtained by checking the
remaining cases \(r=0,1,2,3\) in the inequality above.
\end{proof}

\begin{remark}\label{rem:fundamental-block-reduction-sub}
Lemma~\ref{lem:solv-threshold-reduced-core} removes only neutral squarefree
cyclic factors.  Lemma~\ref{lem:external-cyclic-normal-form} then separates the
remaining central cyclic Sylow subgroups into a coprime cyclic factor \(C\).
The remaining factor \(S\), if nontrivial, is a fundamental block.  Cyclic direct factors whose primes already divide \(|S|\) remain inside
the primary part of \(S\).
\end{remark}

\begin{proposition}\label{prop:solv-threshold-terminal-families}
The following criteria characterize the groups satisfying \(\lambda<59/8\)
within each of the stated families.
\begin{enumerate}[label={\rm(\alph*)}]
\item If \(G\) is nilpotent and \(G=\prod_{p\mid |G|}P_p\) is its Sylow
decomposition, then
\[
        \sub(G)<59\cdot2^{\pi(G)-3}
        \quad\Longleftrightarrow\quad
        \prod_{p\mid |G|}\sub(P_p)<59\cdot2^{\pi(G)-3}.
\]
After the squarefree-reduced-core reduction, the possible \(p\)-group factors
are bounded: cyclic factors are \(C_{p^a}\), \(2\le a\le13\); if a
squarefree-reduced non-cyclic
factor has order \(p^a\), then
\[
\begin{array}{c|c}
p & \text{possible }a\\
\hline
2 & 2\le a\le5\\
3 & 2\le a\le4\\
5 & 2\le a\le3\\
7,11 & a=2.
\end{array}
\]

\item If all Sylow subgroups of \(G\) are cyclic, then
\[
        G\cong ZM(m,n,r),
\]
where
\[
        (m,n)=(m,r-1)=1,\qquad r^n\equiv1\pmod m.
\]
In this case \(G\) satisfies \(\lambda<59/8\) if and only if the sum
in Proposition~\ref{prop:zm-counts}
\[
        \sum_{m_1\mid m}\sum_{n_1\mid n}
        \gcd\left(m_1,\frac{r^n-1}{r^{n_1}-1}\right)
        <59\cdot2^{\pi(mn)-3}.
\]

\item More generally, if \(G=N\times ZM(m,n,r)\), where \(N\) is nilpotent and
\((|N|,mn)=1\), then \(G\) satisfies \(\lambda<59/8\) if and only if
\[
        \sub(N)
        \sum_{m_1\mid m}\sum_{n_1\mid n}
        \gcd\left(m_1,\frac{r^n-1}{r^{n_1}-1}\right)
        <59\cdot2^{\pi(N)+\pi(mn)-3}.
\]
\end{enumerate}
\end{proposition}

\begin{proof}
For nilpotent groups, every subgroup is the direct product of its intersections
with the Sylow subgroups.  The bounds for squarefree-reduced \(p\)-groups follow from
Aivazidis--M\"uller \cite[Theorems~A and~B]{AivazidisMuller2020}; the exact
small primary lists can be read from Betz--Nash \cite[Tables~1 and~3]{betz-nash}.
Groups with cyclic Sylow subgroups are precisely the \(ZM\)-groups \cite{Scott},
and the displayed formula is the subgroup-count formula recorded in
Proposition~\ref{prop:zm-counts}.  The mixed direct-product case follows from
coprime direct-product multiplicativity of subgroup counts.
\end{proof}

For the next refinement, call a nontrivial finite group cyclically
indecomposable if it cannot be written as \(H\times D\), where \(H\ne1\)
and \(D\ne1\) is cyclic, without any coprimality restriction.  This is stronger
than external cyclic indecomposability and also excludes cyclic direct factors
within a primary component.

\begin{proposition}\label{prop:terminal-blocks-sub}
For the nilpotent and \(ZM\)-groups in
Proposition~\ref{prop:solv-threshold-terminal-families}, cyclic
indecomposability is characterized as follows.
\begin{enumerate}[label={\rm(\alph*)}]
\item A nontrivial nilpotent group \(N=\prod_{p\mid |N|}P_p\) is cyclically
indecomposable if and only if either \(N\cong C_{p^a}\) for some prime \(p\),
or every nontrivial Sylow subgroup \(P_p\) is non-cyclic and cyclically
indecomposable as a \(p\)-group.

\item Let
\[
        G=ZM(m,n,r)=\langle a,b\mid a^m=b^n=1,\ b^{-1}ab=a^r\rangle
\]
be nonabelian.  For \(q\mid n\), write \(q^{e_q}\Vert n\).  Then \(G\) is
cyclically indecomposable if and only if
\[
        r^{\,n/q^{e_q}}\not\equiv1\pmod m
        \qquad\text{for every }q\mid n.
\]
That is, no Sylow subgroup of the cyclic complement \(\langle b\rangle\)
acts trivially on \(\langle a\rangle\).
\end{enumerate}
In both cases, the numerical condition for lying below \(\lambda<59/8\) is the
corresponding inequality in
Proposition~\ref{prop:solv-threshold-terminal-families}.
This criterion concerns cyclic indecomposability, rather than the
canonical block decomposition above.  For example, a cyclic group
\(C_{p^a}\) is cyclically indecomposable in the strong sense, but in
Lemma~\ref{lem:external-cyclic-normal-form} it belongs to the cyclic factor
\(C\), not to the nontrivial fundamental block \(S\).
\end{proposition}

\begin{proof}
For nilpotent groups, direct factors are obtained prime by prime.  If a Sylow
subgroup \(P_p\) has a nontrivial cyclic direct factor, then \(N\) has one as
well.
Conversely, any cyclic direct factor of \(N\) is the product of cyclic direct
factors from the Sylow subgroups.  Hence a nilpotent group is cyclically
indecomposable precisely in the cases stated.

Now let \(G=ZM(m,n,r)\) be nonabelian.  Since \((m,r-1)=1\), no nontrivial
element of \(\langle a\rangle\) is central.  Thus any cyclic direct factor must
come from the cyclic complement \(\langle b\rangle\).  The Sylow
\(q\)-subgroup of \(\langle b\rangle\) is generated by \(b^{n/q^{e_q}}\), and it
centralizes \(\langle a\rangle\) exactly when
\[
        r^{\,n/q^{e_q}}\equiv1\pmod m.
\]
If this happens, that central Sylow subgroup splits off as a direct cyclic
factor.  Conversely, because all Sylow subgroups of \(G\) are cyclic, any
nontrivial cyclic direct factor must contain the full Sylow subgroup for some
prime \(q\), and hence gives a central Sylow subgroup of the complement.  This
is exactly the displayed congruence.
\end{proof}

\begin{corollary}\label{cor:zm-sub-fundamental-blocks}
Let
\[
        S=ZM(m,n,r)=\langle a,b\mid a^m=b^n=1,\ b^{-1}ab=a^r\rangle
\]
be nonabelian, and for \(q\mid n\) write \(q^{e_q}\Vert n\).  Then \(S\) is a
fundamental block for \(\lambda<59/8\) if and only if
\[
        \sum_{m_1\mid m}\sum_{n_1\mid n}
        \gcd\left(m_1,\frac{r^n-1}{r^{n_1}-1}\right)
        <59\cdot2^{\pi(mn)-3}
\]
and
\[
        r^{\,n/q^{e_q}}\not\equiv1\pmod m
        \qquad\text{for every }q\mid n.
\]
\end{corollary}

\begin{proof}
The displayed subgroup-count inequality is exactly the condition that \(S\)
satisfy \(\lambda<59/8\), by
Proposition~\ref{prop:solv-threshold-terminal-families}{\rm(b)}.

It remains to identify the fundamental-block condition.  Since
\((m,r-1)=1\), no nonidentity element of the cyclic kernel \(\langle a\rangle\)
is central.  All Sylow subgroups of \(S\) are cyclic.  The Sylow
\(q\)-subgroup of the cyclic complement \(\langle b\rangle\) is generated by
\(b^{n/q^{e_q}}\), and it is central precisely when
\[
        r^{\,n/q^{e_q}}\equiv1\pmod m.
\]
Thus the second displayed condition is equivalent to saying that \(S\) has no
central cyclic Sylow subgroup.  Together with the \(\lambda<59/8\) inequality, this is
precisely the definition of a fundamental block in the subgroup-count section.
\end{proof}

\begin{corollary}\label{cor:nilpotent-sub-primary-blocks}
Let \(N\) be nilpotent and satisfy
\[
        \sub(N)<59\cdot2^{\pi(N)-3}.
\]
Let \(K\) be the squarefree-reduced core from
Lemma~\ref{lem:solv-threshold-reduced-core}, and write
\[
        K=P_1\times\cdots\times P_r
\]
as the product of its nontrivial Sylow subgroups.  Then \(r\le4\).  At most
one of the \(P_i\) is non-cyclic; if one is non-cyclic, then \(r\le3\).

If all \(P_i\) are cyclic, say
\[
        P_i\cong C_{p_i^{a_i}},
        \qquad a_1\ge a_2\ge\cdots\ge a_r\ge2,
\]
then the possible exponent patterns are exactly
\[
\begin{array}{ll}
r=0: & \text{empty},\\
r=1: & 2\le a_1\le13,\\
r=2: & (a_1+1)(a_2+1)\le29,\\
r=3: & (a_1,a_2,a_3)=(a,2,2),\ 2\le a\le5,\quad\text{or }(3,3,2),\\
r=4: & (a_1,a_2,a_3,a_4)=(3,2,2,2)\quad\text{or }(2,2,2,2).
\end{array}
\]
\end{corollary}

\begin{proof}
Because \(K\) is squarefree-reduced, no \(P_i\) has prime order.  Hence
\(\sub(P_i)\ge3\), and therefore \(\lambda(P_i)\ge3/2\).  Since
\[
        \lambda(K)=\prod_{i=1}^r\lambda(P_i)<\frac{59}{8},
\]
it follows that \(r\le4\), as \((3/2)^5>59/8\).

For a non-cyclic \(p\)-group \(P\), the quotient \(P/\Phi(P)\) has rank at
least \(2\), so \(P\) has at least \(p+1\) maximal subgroups.  Including
\(1\) and \(P\), we obtain \(\sub(P)\ge p+3\).
If \(K\) had non-cyclic Sylow subgroups for distinct primes \(p<q\),
then coprime multiplicativity, with every other factor contributing at
least \(1\), would give
\[
        \lambda(K)\ge\frac{(p+3)(q+3)}4
        \ge\frac{5\cdot6}{4}
        =\frac{15}{2}>\frac{59}{8},
\]
a contradiction.  Thus at most one Sylow factor is non-cyclic.  Such a
factor contributes at least \(5/2\), and
\[
        \frac52\left(\frac32\right)^3>\frac{59}{8}
\]
shows that \(r\le3\) in that case.

Assume finally that all \(P_i\) are cyclic.  Then
\[
        \lambda(K)=\prod_{i=1}^r\frac{a_i+1}{2}<\frac{59}{8}.
\]
Thus
\[
        \prod_{i=1}^r(a_i+1)<59\cdot2^{r-3}.
\]
The displayed list follows by checking this elementary inequality with
\(a_1\ge\cdots\ge a_r\ge2\).
\end{proof}

\begin{corollary}\label{cor:nilpotent-sub-fundamental-blocks}
Let \(S\ne1\) be nilpotent.  Then \(S\) is a fundamental block for the
subgroup-count threshold if and only if \(S\) is a non-cyclic \(p\)-group with
\[
        \sub(S)\le14.
\]
\end{corollary}

\begin{proof}
In a finite nilpotent group every cyclic Sylow subgroup is central: it
commutes with itself and with all the other Sylow factors.  Hence a
nilpotent fundamental block \(S\) has only non-cyclic Sylow subgroups.
Since \(S\) is its own squarefree-reduced core,
Corollary~\ref{cor:nilpotent-sub-primary-blocks} forces \(S\) to be a
non-cyclic \(p\)-group.  Then \(\lambda(S)<59/8\) is equivalent to
\(\sub(S)<59/4\), and hence to \(\sub(S)\le14\).

Conversely, a non-cyclic \(p\)-group has no central cyclic Sylow subgroup
and is squarefree-reduced.  If it has at most fourteen subgroups, then
\(\lambda(S)\le7<59/8\), so it is a fundamental block.
\end{proof}

\begin{corollary}\label{cor:mixed-sub-fundamental-blocks}
Let
\[
        S=N\times ZM(m,n,r),
\]
where \(N\) is nilpotent, \(ZM(m,n,r)\) is nonabelian, and \((|N|,mn)=1\).
For \(q\mid n\), write \(q^{e_q}\Vert n\).  Then \(S\) is a fundamental block
for the subgroup-count threshold if and only if the following three conditions
hold:
\[
        \sub(N)
        \sum_{m_1\mid m}\sum_{n_1\mid n}
        \gcd\left(m_1,\frac{r^n-1}{r^{n_1}-1}\right)
        <59\cdot2^{\pi(N)+\pi(mn)-3},
\]
every nontrivial Sylow subgroup of \(N\) is non-cyclic, and
\[
        r^{\,n/q^{e_q}}\not\equiv1\pmod m
        \qquad\text{for every }q\mid n.
\]
\end{corollary}

\begin{proof}
The first displayed inequality is the condition \(\lambda<59/8\) for
\(N\times ZM(m,n,r)\), by
Proposition~\ref{prop:solv-threshold-terminal-families}{\rm(c)}.  It remains to
spell out the fundamental-block condition, namely that no central cyclic Sylow
subgroup remains.

In the nilpotent factor \(N\), a cyclic Sylow subgroup is central in \(N\), and
hence central in the direct product \(S\).  Thus the nilpotent factor contributes
no central cyclic Sylow subgroup exactly when all its nontrivial Sylow subgroups
are non-cyclic.  In the \(ZM\)-factor, the kernel \(\langle a\rangle\) has no
nontrivial central element because \((m,r-1)=1\).  The Sylow \(q\)-subgroup of
the complement is central exactly when
\[
        r^{\,n/q^{e_q}}\equiv1\pmod m,
\]
as in Corollary~\ref{cor:zm-sub-fundamental-blocks}.  Therefore the two
structural conditions are exactly equivalent to saying that \(S\) has no
central cyclic Sylow subgroup.  Together with the \(\lambda<59/8\) inequality, this is
the definition of a fundamental block.
\end{proof}

The next section begins the proof of the cyclic-subgroup solvability theorem;
Section~\ref{sec:from-simple} completes it using the simple-group estimates
proved in the appendices.

\section{\texorpdfstring{The cyclic-subgroup solvability threshold \(\eta<4\)}{The cyclic-subgroup solvability threshold eta<4}}\label{sec:cyc-solv}

We now turn to the cyclic-subgroup solvability theorem.

\begin{theorem}\label{thm:main}
Let \(G\) be a finite group.  If
\[
        \eta(G)<4,
\]
then \(G\) is solvable.
\end{theorem}

\begin{remark}\label{rem:cyc-solvable-tight}
The constant \(4\) in Theorem~\ref{thm:main} is sharp.  The group \(A_5\) is
not solvable and satisfies
\[
        \eta(A_5)=4.
\]
Moreover, if \(p_1,\ldots,p_r\) are distinct primes greater than \(5\), then
\[
        G=A_5\times C_{p_1}\times\cdots\times C_{p_r}
\]
is nonsolvable and still has normalized value \(\eta(G)=4\), by the coprime
direct-product formula in Theorem~\ref{thm:product-cyclic-grps}.
\end{remark}

The proof rests on the following estimate for finite nonabelian simple groups.

\begin{theorem}\label{thm:simple}
For every finite nonabelian simple group \(S\),
\[
        \cyc(S)\ge 2^{\pi(\Aut(S))+2}.
\]
\end{theorem}

Appendices~\ref{sec:alternating}--\ref{sec:remaining-cases} establish the
simple-group estimate family by family.  Section~\ref{sec:from-simple}
assembles those estimates and uses the solvable radical and the socle to
complete the proof of Theorem~\ref{thm:main}.

\subsection{Additional preliminaries for the cyclic-subgroup solvability theorem}

For a finite group \(G\), let \(\Rad(G)\) denote the solvable radical, that is,
the largest solvable normal subgroup of \(G\).  Let \(\Soc(G)\) denote the
socle of \(G\), the subgroup generated by the minimal normal subgroups of
\(G\).

\begin{lemma}\label{lem:quotient}
If \(N\triangleleft G\), then
\[
        \cyc(G/N)\le \cyc(G).
\]
\end{lemma}

\begin{proof}
Every cyclic subgroup of \(G/N\) has the form \(\langle gN\rangle\), hence is
the image of the cyclic subgroup \(\langle g\rangle\le G\).  Thus the natural
map from cyclic subgroups of \(G\) to cyclic subgroups of \(G/N\), given by
\(H\mapsto HN/N\), is surjective.
\end{proof}

\begin{lemma}\label{lem:pi-estimates}
For every positive integer \(n\), the following estimates hold.
\begin{enumerate}[label=(\alph*)]
\item If \(n\) is odd, then
\[
        2^{\pi(n)}\le \frac{2}{\sqrt 3}\sqrt n.
\]
\item If \(8\mid n\), then
\[
        2^{\pi(n)}\le \sqrt{\frac{2n}{3}}.
\]
\item For every \(n\ge 1\),
\[
        2^{\pi(n)}\le 2\sqrt n.
\]
\item For every \(n\ge 1\),
\[
        2^{\pi(n)}\le n.
\]
\end{enumerate}
\end{lemma}

\begin{proof}
If \(n\) is odd, then
\[
        4^{\pi(n)}=\operatorname{rad}(n)\prod_{\substack{r\mid n\\r\text{ prime}}}\frac4r.
\]
Among odd primes, the product \(\prod 4/r\) is maximized by including only the
prime \(3\), since all primes \(r\ge5\) contribute a factor at most
\(4/5<1\).  Hence
\[
        4^{\pi(n)}\le \frac43\operatorname{rad}(n)\le \frac43 n,
\]
which gives (a).  If \(8\mid n\), write \(n=2^a m\), where \(a\ge3\) and
\(m\) is odd.  Then
\[
        2^{\pi(n)}=2\cdot 2^{\pi(m)}
        \le 2\cdot \frac{2}{\sqrt3}\sqrt m
        =\frac4{\sqrt3}\sqrt{\frac{n}{2^a}}
        \le \sqrt{\frac{2n}{3}}.
\]
Part (c) follows by applying (a) to the odd part and including the extra factor
\(2\) contributed by the prime \(2\).  Part (d) follows from
\(2^{\pi(n)}\le\operatorname{rad}(n)\le n\).
\end{proof}

\subsection{Simple groups and the reduction}\label{subsec:simple-families}

The proof of Theorem~\ref{thm:simple} is organized family by family.

If \(S\) is nonabelian simple, then \(Z(S)=1\), and the inner automorphism group is naturally identified with \(S\).  Hence
\[
        |\Aut(S)|=|S|\,|\Out(S)|.
\]
Thus each family requires two inputs: a lower bound for \(\cyc(S)\), and enough
information on the prime divisors of \(|\Aut(S)|\).

The passage from Theorem~\ref{thm:simple} to Theorem~\ref{thm:main} uses the
solvable radical and the socle.

Appendix~\ref{sec:alternating} treats \(A_n\).  Appendix~\ref{sec:psl2} treats
the rank-one family \(\PSL(2,q)\).  Appendix~\ref{sec:lie-type} treats the
remaining Lie-type families, apart from a finite list of small cases.
Appendix~\ref{sec:remaining-cases} records the character-table checks
for those cases and for the sporadic groups.

\section{From simple groups to solvability}\label{sec:from-simple}

This section assembles the family estimates proved in
Appendices~\ref{sec:alternating}--\ref{sec:remaining-cases} and then passes from
simple groups to arbitrary finite groups.

\begin{proof}[Proof of Theorem~\ref{thm:simple}]
By the classification of finite simple groups, the finite nonabelian simple
groups are the alternating groups, the simple groups of Lie type, and the 26
sporadic groups \cite{GLS}.  Proposition~\ref{prop:alternating} treats the
alternating groups.  For groups of Lie type, Proposition~\ref{prop:psl2}
treats \(\PSL(2,q)\); Propositions~\ref{prop:linear-unitary}
and~\ref{prop:BCD} treat the remaining classical groups;
Proposition~\ref{prop:suzuki-ree} treats the Suzuki and Ree groups;
Proposition~\ref{prop:exceptional} treats the exceptional groups; and
Proposition~\ref{prop:table-checks} treats the residual small Lie-type cases
and the Tits group.  Finally, Table~\ref{tab:sporadic} treats the sporadic
groups.  Hence every finite nonabelian simple group satisfies
\[
        \cyc(S)
        \ge 2^{\pi(\Aut(S))+2}.
\]
\end{proof}

\subsection{Minimal counterexamples and the socle}

The reduction from Theorem~\ref{thm:simple} to Theorem~\ref{thm:main} uses
only elementary counting, the comparison in Corollary~\ref{lem:coprime-extension},
and the structure of the socle after the solvable radical has been removed.

\begin{lemma}\label{lem:radical}
Let \(G\) be a finite nonsolvable group of minimal order among those satisfying
\[
        \cyc(G)<2^{\pi(G)+2}.
\]
Then
\[
        \Rad(G)=1.
\]
\end{lemma}

\begin{proof}
Suppose not, and let \(R=\Rad(G)\ne1\).  Choose a minimal normal subgroup
\(N\triangleleft G\) with \(N\le R\).  Then \(N\) is elementary abelian, say a
\(p\)-group.  Since \(N\) is solvable and \(G\) is nonsolvable, the quotient
\(G/N\) is nonsolvable.

If \(p\mid |G/N|\), then \(\pi(G/N)=\pi(G)\).  By minimality of \(G\),
\[
        \cyc(G/N)\ge 2^{\pi(G/N)+2}=2^{\pi(G)+2}.
\]
By Lemma~\ref{lem:quotient}, \(\cyc(G/N)\le\cyc(G)\), contradicting the
choice of \(G\).

Assume now that \(p\nmid |G/N|\).  Then \((|N|,|G/N|)=1\), and
\(\pi(G/N)=\pi(G)-1\).  By Lemma~\ref{lem:coprime-extension}, applied to the
extension of \(N\) by \(G/N\),
\[
        \cyc(G)\ge \cyc(N\times G/N)=\cyc(N)\cyc(G/N).
\]
Since \(N\ne1\), we have \(\cyc(N)\ge2\).  Again by minimality,
\[
        \cyc(G/N)\ge 2^{\pi(G/N)+2}=2^{\pi(G)+1}.
\]
Thus
\[
        \cyc(G)\ge 2\cdot 2^{\pi(G)+1}=2^{\pi(G)+2},
\]
again a contradiction.
\end{proof}

A minimal normal subgroup of a finite group is a direct product of isomorphic
simple groups \cite{Isac-ams}; hence, if \(\Rad(G)=1\), then \(\Soc(G)\) is a
direct product of powers of nonabelian simple groups.

\begin{proof}[Proof of Theorem~\ref{thm:main}]
The hypothesis \(\eta(G)<4\) means
\[
        \cyc(G)<2^{\pi(G)+2}.
\]
Suppose, toward a contradiction, that \(G\) is a finite nonsolvable group of
minimal order among those satisfying this inequality.
By Lemma~\ref{lem:radical}, \(\Rad(G)=1\).  Therefore
\[
        \Soc(G)=T_1^{m_1}\times\cdots\times T_r^{m_r},
\]
where the \(T_i\) are pairwise nonisomorphic finite nonabelian simple groups.

The centralizer \(C_G(\Soc(G))\) is trivial.  Indeed, if this centralizer were
nontrivial, then it would contain a minimal normal subgroup \(L\) of \(G\).
Since \(\Rad(G)=1\), the subgroup \(L\) is nonabelian.  But \(L\le \Soc(G)\) and
\(L\le C_G(\Soc(G))\).  Thus \(L\) centralizes itself, forcing \(L\) to be abelian,
a contradiction.  Hence conjugation gives an embedding
\[
        G\hookrightarrow \Aut(\Soc(G))
        \cong \prod_{i=1}^r \Aut(T_i)\wr S_{m_i}.
\]
The wreath product description comes from the fact that automorphisms of
\(T_i^{m_i}\) act on the individual simple factors and may also permute the
\(m_i\) isomorphic factors; factors from different isomorphism classes cannot
be permuted.
For \(1\le i\le r\), put
\[
        a_i=\pi(\Aut(T_i)),\qquad b_i=\pi(m_i!).
\]
The embedding gives
\[
        \pi(G)\le \sum_{i=1}^r a_i+\sum_{i=1}^r b_i.
\]

Since \(\Soc(G)\le G\), every cyclic subgroup of \(\Soc(G)\) is also a cyclic
subgroup of \(G\).  Hence \(\cyc(G)\ge \cyc(\Soc(G))\).  Therefore, by
Theorem~\ref{thm:product-cyclic-grps} and Theorem~\ref{thm:simple},
\[
\begin{aligned}
        \cyc(G)
        &\ge \cyc(\Soc(G)) \\
        &\ge \prod_{i=1}^r \cyc(T_i)^{m_i} \\
        &\ge \prod_{i=1}^r 2^{m_i(a_i+2)}
         =2^{\sum_i m_i(a_i+2)}.
\end{aligned}
\]
We compare exponents.  For every \(i\),
\[
        m_i(a_i+2)-a_i-b_i=(m_i-1)a_i+2m_i-b_i.
\]
If \(m_i=1\), then \(b_i=0\), and this expression equals \(2\).  If
\(m_i\ge2\), then \(a_i\ge3\), by Burnside's \(p^aq^b\)-theorem
\cite{Isac-ams}, and \(b_i=\pi(m_i!)\le m_i\).  Hence
\[
        (m_i-1)a_i+2m_i-b_i\ge 3(m_i-1)+m_i=4m_i-3\ge5.
\]
Therefore
\[
        \sum_i m_i(a_i+2)
        \ge
        \sum_i a_i+\sum_i b_i+2
        \ge \pi(G)+2.
\]
Consequently
\[
        \cyc(G)
        \ge 2^{\pi(G)+2},
\]
contradicting the choice of \(G\).
\end{proof}

\section{\texorpdfstring{Reduced cores below \(\eta<4\)}{Reduced cores below eta<4}}
\label{sec:cyc-threshold-structure}

Theorem~\ref{thm:main} shows that every group with \(\eta(G)<4\) is solvable.
We finish with the coprime cyclic-factor decomposition for this threshold and
exact criteria for selected families.  These are structural reductions, not a
classification of all groups with \(\eta<4\).  Unlike the groups below
\(\lambda<5/2\), those considered here need not have cyclic Sylow subgroups;
for example, \(A_4\) and \(Q_8\) satisfy \(\eta<4\).

We use the squarefree-reduced convention of
Section~\ref{sec:sub-threshold-structure}.  A fundamental block for the
cyclic-subgroup threshold is a nontrivial squarefree-reduced group with no
central cyclic Sylow subgroup and with \(\eta<4\).  The external cyclic normal
form in Lemma~\ref{lem:external-cyclic-normal-form} applies unchanged.
Together with Theorem~\ref{thm:product-cyclic-grps}, it gives
\[
        G\cong A\times S\times C,\qquad \eta(G)=\eta(S)\delta(C),
\]
where the factors have the same roles as in
Corollary~\ref{cor:sub-block-decomposition}.  Thus only the numerical condition
changes, to \(\eta(S)\delta(C)<4\).  The precise statement is
Corollary~\ref{cor:cyc-block-decomposition}.

\begin{question}\label{ques:finitely-many-cyc-blocks}
Are there only finitely many fundamental blocks \(S\) with
\[
        \eta(S)<4?
\]
\end{question}

\begin{lemma}\label{lem:cyc-threshold-reduced-core}
Let \(G\) be a finite group satisfying
\[
        \cyc(G)<2^{\pi(G)+2}.
\]
Then
\[
        G\cong A\times K,
\]
where \(A\) is cyclic of squarefree order, \((|A|,|K|)=1\), the group \(K\) is
squarefree-reduced, and
\[
        \cyc(K)<2^{\pi(K)+2}.
\]
Conversely, every such direct product satisfies the inequality for \(G\) if
and only if \(K\) satisfies the displayed inequality.
\end{lemma}

\begin{proof}
Suppose that \(P\) is a central Sylow subgroup of prime order \(p\).  By Schur--Zassenhaus,
\[
        G=P\rtimes H,
\]
where \(H\cong G/P\).  Since \(P\le Z(G)\), the action is trivial and
\[
        G=P\times H.
\]
The two factors have coprime orders, whence
\[
        \cyc(G)=2\cyc(H),
        \qquad
        \pi(G)=\pi(H)+1.
\]
Thus \(G\) satisfies the strict inequality precisely when \(H\) does.  Iterating
this argument removes all central Sylow subgroups of prime order and gives the
decomposition \(G\cong A\times K\).  The converse follows from
Theorem~\ref{thm:product-cyclic-grps}.
\end{proof}

\begin{lemma}\label{lem:cyc-cyclic-factor-over-block}
Let \(S\) be a finite group and let \(C\) be cyclic, with
\((|S|,|C|)=1\).  Then
\[
        \eta(S\times C)=\eta(S)\delta(C).
\]
Consequently \(S\times C\) satisfies \(\eta<4\)
if and only if
\[
        \eta(S)\delta(C)<4.
\]
\end{lemma}

\begin{proof}
Theorem~\ref{thm:product-cyclic-grps} gives
\[
        \cyc(S\times C)=\cyc(S)\cyc(C)=\cyc(S)d(|C|).
\]
Since \(\pi(S\times C)=\pi(S)+\pi(C)\), division by
\(2^{\pi(S)+\pi(C)}\) gives the formula.
\end{proof}

\begin{corollary}\label{cor:cyc-block-decomposition}
Let \(G\) be a finite group satisfying
\[
        \cyc(G)<2^{\pi(G)+2}.
\]
Then
\[
        G\cong A\times S\times C,
\]
where \(A\) is cyclic of squarefree order, \(C\) is cyclic,
\[
        (|A|,|S||C|)=1,\qquad (|S|,|C|)=1,
\]
and either \(S=1\) or \(S\) is a fundamental block for the cyclic-subgroup
threshold.  With the convention \(\eta(1)=1\), the remaining condition is
\[
        \eta(S)\delta(C)<4.
\]
Moreover, if
\[
        C\cong C_{\ell_1^{a_1}}\times\cdots\times C_{\ell_r^{a_r}},
\]
with distinct primes \(\ell_i\), then every nontrivial exponent satisfies
\(a_i\ge2\).

Conversely, every product with these properties satisfies the displayed
cyclic-subgroup inequality.
When \(S=1\), this is the purely cyclic reduced core case; the possible
exponent patterns for \(C\) are exactly the cyclic patterns listed in
Corollary~\ref{cor:nilpotent-cyc-primary-blocks}.
\end{corollary}

\begin{proof}
Apply Lemma~\ref{lem:cyc-threshold-reduced-core} to write
\[
        G\cong A\times K,
\]
where \(A\) is cyclic squarefree, \((|A|,|K|)=1\), and \(K\) is
squarefree-reduced with \(\eta(K)<4\).  Apply
Lemma~\ref{lem:external-cyclic-normal-form} to \(K\):
\[
        K\cong S\times C,
\]
where \(C\) is cyclic, \((|S|,|C|)=1\), and either \(S=1\) or \(S\) has no
central cyclic Sylow subgroup.  As in the subgroup-count case, the
squarefree-reduced property descends from \(K\) to \(S\), and the exponents in
\(C\) are at least \(2\).

Finally, Lemma~\ref{lem:cyc-cyclic-factor-over-block} gives
\[
        \eta(K)=\eta(S)\delta(C),
\]
and the factor \(A\) is neutral for \(\eta\).  Since \(\delta(C)\ge1\) and
\(\eta(K)<4\), we have \(\eta(S)<4\).  Therefore, if \(S\ne1\), then \(S\) is
a fundamental block for \(\eta<4\).  This proves the necessity.  The converse
follows from the same equality in Lemma~\ref{lem:cyc-cyclic-factor-over-block}.
\end{proof}

\begin{corollary}\label{cor:cyc-admissible-cyclic-decorations}
Let \(S\) be fixed with \(\eta(S)<4\), and let
\[
        C\cong C_{\ell_1^{a_1}}\times\cdots\times C_{\ell_r^{a_r}},
        \qquad a_i\ge2,
\]
where the primes \(\ell_i\) are distinct and do not divide \(|S|\).  Then
\[
        S\times C
\]
satisfies \(\eta<4\) if and only if
\[
        \eta(S)\prod_{i=1}^r\frac{a_i+1}{2}<4.
\]
In particular, for a fixed block \(S\), there are only finitely many possible
exponent patterns \((a_1,\ldots,a_r)\) for non-squarefree coprime cyclic
factors.
\end{corollary}

\begin{proof}
This follows from Lemma~\ref{lem:cyc-cyclic-factor-over-block} and the same
formula for \(\delta(C)\).  Since every \(a_i\ge2\), each factor
\((a_i+1)/2\) is at least \(3/2\), while their product is bounded above by
\(4/\eta(S)\).
\end{proof}

\begin{corollary}\label{cor:cyc-universal-decoration-patterns}
Let \(S\ne1\) be a fundamental block for the threshold \(\eta<4\).  If
\[
        C\cong C_{\ell_1^{a_1}}\times\cdots\times C_{\ell_r^{a_r}},
        \qquad a_1\ge\cdots\ge a_r\ge2,
\]
is a cyclic group of order coprime to \(|S|\) such that \(S\times C\) still
satisfies \(\eta<4\), then one of the following
holds:
\[
\begin{array}{c|c}
r&\text{possible exponent patterns}\\
\hline
0&\text{empty}\\
1&(a),\ 2\le a\le5\\
2&(3,2)\text{ or }(2,2).
\end{array}
\]
No admissible decoration has \(r\ge3\).
\end{corollary}

\begin{proof}
Since \(S\) is squarefree-reduced and nontrivial, it is not cyclic of
squarefree order.  Corollary~\ref{cor:initial-threshold-lower-bounds} gives
\[
        \cyc(S)\ge5\cdot2^{\pi(S)-2},
\]
and hence \(\eta(S)\ge5/4\).  Therefore any admissible decoration satisfies
\[
        \prod_{i=1}^r\frac{a_i+1}{2}
        <\frac4{\eta(S)}
        \le \frac{16}{5}.
\]
Since every \(a_i\ge2\), this excludes \(r\ge3\).  The displayed list follows
by checking \(r=0,1,2\).
\end{proof}

\begin{lemma}\label{lem:cyc-threshold-order-pattern}
Let \(G\) be a finite group satisfying
\[
        \cyc(G)<2^{\pi(G)+2},
\]
and write
\[
        |G|=\prod_{i=1}^t p_i^{a_i}.
\]
Then
\[
        \prod_{i=1}^t(a_i+1)<4\cdot2^t.
\]
This is the same as
\[
        \prod_{i=1}^t\frac{a_i+1}{2}<4.
\]
In particular, after removing the exponents equal to \(1\), the remaining
multiset of exponents is one of the following:
\[
        \varnothing,\qquad
        (a)\ \text{with }2\le a\le6,\qquad
        (2,b)\ \text{with }2\le b\le4,\qquad
        (2,2,2).
\]
\end{lemma}

\begin{proof}
Theorem~\ref{thm:Richards} gives
\[
        \cyc(G)\ge \cyc(C_{|G|})=d(|G|)
        =\prod_{i=1}^t(a_i+1).
\]
The displayed inequality follows from the hypothesis.  It remains only to list
the possible exponent patterns.  If some \(a_i\ge7\), then
\((a_i+1)/2\ge4\), impossible.  If at least four exponents are larger than
\(1\), then the product is at least \((3/2)^4>4\), again impossible.  If three
exponents are larger than \(1\), then any exponent at least \(3\) gives product
at least
\[
        \frac32\cdot\frac32\cdot2>4;
\]
hence the only three-exponent case is \((2,2,2)\).  If two exponents are larger
than \(1\), then the smaller one must be \(2\), because \(2\cdot2=4\) is already
too large.  Writing the larger exponent as \(b\), the condition
\[
        \frac32\cdot\frac{b+1}{2}<4
\]
gives \(b\le4\).  The one-exponent case gives \(a\le6\).
\end{proof}

\begin{proposition}\label{prop:cyc-abelian-hall-formula}
Let \(G=A\rtimes B\), where \(A\) is an abelian normal Hall subgroup of \(G\).
For \(b\in B\), put \(k_b=o(b)\), and define
\[
        N_b(a)=a\,b(a)\,b^2(a)\cdots b^{k_b-1}(a)
        \qquad (a\in A),
\]
where \(b(a)=bab^{-1}\).  Then
\[
        \cyc(G)=
        \sum_{b\in B}\sum_{a\in A}
        \frac{1}{\phi(k_b)\,\phi(o(N_b(a)))}.
\]
Consequently \(G\) satisfies \(\eta<4\) if and
only if the displayed sum is smaller than \(2^{\pi(G)+2}\).
\end{proposition}

\begin{proof}
For \(a\in A\) and \(b\in B\), the image of \(ab\) in \(B\) has order \(k_b\).
Moreover,
\[
        (ab)^{k_b}=N_b(a)\in A.
\]
Since \((|A|,|B|)=1\), the order of \(N_b(a)\) is coprime to \(k_b\).  Hence
\[
        o(ab)=k_b\,o(N_b(a)).
\]
Indeed, if \((ab)^s=1\), then \(b^s=1\), and therefore \(k_b\mid s\).  Writing
\(s=k_bj\), one gets \(N_b(a)^j=1\), and hence \(o(N_b(a))\mid j\).  The
formula now follows from Lemma~\ref{lem:cyclic-subgroup-number-in-terms-order}
and the multiplicativity of \(\phi\) on coprime integers.
\end{proof}

\begin{proposition}\label{prop:cyc-threshold-terminal-families}
The following criteria characterize the groups satisfying \(\eta<4\)
within each of the stated families.
\begin{enumerate}[label={\rm(\alph*)}]
\item If \(G\) is nilpotent and \(G=\prod_{p\mid |G|}P_p\) is its Sylow
decomposition, then
\[
        \cyc(G)<2^{\pi(G)+2}
        \quad\Longleftrightarrow\quad
        \prod_{p\mid |G|}\cyc(P_p)<2^{\pi(G)+2}.
\]
If \(G\) is also squarefree-reduced, then each Sylow factor \(P_p\) satisfies
\(\cyc(P_p)<8\).  Consequently each \(P_p\) is one of
\[
        C_{p^a}\ (2\le a\le6),\quad
        C_2^2,\quad C_4\times C_2,\quad D_8,\quad Q_8,\quad
        C_3^2,\quad C_5^2,
\]
where \(D_8\) denotes the dihedral group of order \(8\).  The remaining
condition is
\[
        \prod_{p\mid |G|}\frac{\cyc(P_p)}2<4.
\]

\item If all Sylow subgroups of \(G\) are cyclic, then
\[
        G\cong ZM(m,n,r),
\]
where
\[
        (m,n)=(m,r-1)=1,\qquad r^n\equiv1\pmod m.
\]
In this case \(G\) satisfies \(\eta<4\) if and only if
\[
        \sum_{\substack{m_1\mid m,\ n_1\mid n\\ m/m_1\mid r^{n_1}-1}}
        \gcd\left(m_1,\frac{r^n-1}{r^{n_1}-1}\right)
        <2^{\pi(mn)+2}.
\]

\item If \(G=A\rtimes B\), where \(A\) is an abelian normal Hall subgroup, then
the exact condition is the inequality in
Proposition~\ref{prop:cyc-abelian-hall-formula}.

\item If \(G=N\times ZM(m,n,r)\), where \(N\) is nilpotent and
\((|N|,mn)=1\), then \(G\) satisfies \(\eta<4\) if and only if
\[
        \cyc(N)
        \sum_{\substack{m_1\mid m,\ n_1\mid n\\ m/m_1\mid r^{n_1}-1}}
        \gcd\left(m_1,\frac{r^n-1}{r^{n_1}-1}\right)
        <2^{\pi(N)+\pi(mn)+2}.
\]
\end{enumerate}
\end{proposition}

\begin{proof}
For nilpotent groups, the Sylow subgroups have pairwise coprime orders, and
Theorem~\ref{thm:product-cyclic-grps} gives the exact product formula for
\(\cyc(G)\).  If \(G\) is squarefree-reduced and nilpotent, then no Sylow subgroup has
prime order.  Moreover,
\[
        \prod_{p\mid |G|}\frac{\cyc(P_p)}{2}<4,
\]
and every factor in this product is at least \(1\).  Hence each
\(\cyc(P_p)<8\).  A cyclic group \(C_{p^a}\) has \(a+1\) cyclic subgroups,
which gives \(2\le a\le6\) in the squarefree-reduced case.  If \(P\) has order \(p^a\)
and is neither cyclic nor generalized quaternion, Theorem~\ref{thm:amiri-jaco}
and Lemma~\ref{lem:strong-dom-implies-less-cyc} give
\[
        \cyc(P)\ge \cyc(C_{p^{a-1}}\times C_p).
\]
A direct count gives
\[
        \cyc(C_{p^{a-1}}\times C_p)=2+p(a-1),
\]
and therefore \(p(a-1)<6\).  Thus the only non-cyclic, non-generalized
quaternion possibilities have \((p,a)=(2,2),(2,3),(3,2)\), or \((5,2)\).  This
gives \(C_2^2\), \(C_4\times C_2\), \(D_8\), \(C_3^2\), and \(C_5^2\); the
elementary abelian group \(C_2^3\) has \(8\) cyclic subgroups and is excluded
by the strict inequality.  Finally, a generalized quaternion group of order
\(2^a\), \(a\ge3\), has
\[
        a+2^{a-2}
\]
cyclic subgroups: the \(a\) subgroups in its cyclic subgroup of index \(2\),
and \(2^{a-2}\) cyclic subgroups of order \(4\) outside it.  The inequality
\(a+2^{a-2}<8\) leaves only \(a=3\), that is, \(Q_8\).

Groups with cyclic Sylow subgroups are precisely the \(ZM\)-groups \cite{Scott},
and the displayed condition is the cyclic-subgroup formula in
Proposition~\ref{prop:zm-counts}.  The abelian-Hall case is exactly
Proposition~\ref{prop:cyc-abelian-hall-formula}.  The mixed direct-product case
follows again from coprime direct-product multiplicativity of cyclic-subgroup
counts.
\end{proof}

\begin{corollary}\label{cor:zm-cyc-fundamental-blocks}
Let
\[
        S=ZM(m,n,r)=\langle a,b\mid a^m=b^n=1,\ b^{-1}ab=a^r\rangle
\]
be nonabelian, and for \(q\mid n\) write \(q^{e_q}\Vert n\).  Then \(S\) is a
fundamental block for the cyclic-subgroup threshold if and only if
\[
        \sum_{\substack{m_1\mid m,\ n_1\mid n\\ m/m_1\mid r^{n_1}-1}}
        \gcd\left(m_1,\frac{r^n-1}{r^{n_1}-1}\right)
        <2^{\pi(mn)+2}
\]
and
\[
        r^{\,n/q^{e_q}}\not\equiv1\pmod m
        \qquad\text{for every }q\mid n.
\]
\end{corollary}

\begin{proof}
The first displayed inequality is the condition \(\eta<4\), by
Proposition~\ref{prop:cyc-threshold-terminal-families}{\rm(b)}.  The second
condition is the same central-Sylow condition as in
Corollary~\ref{cor:zm-sub-fundamental-blocks}: it says exactly that no Sylow
subgroup of the cyclic complement \(\langle b\rangle\) acts trivially on the
cyclic kernel \(\langle a\rangle\).  Since \((m,r-1)=1\), the kernel contains no
nontrivial central element.  Therefore \(S\) has no central cyclic Sylow
subgroup.  This is precisely the additional condition needed for \(S\) to be a
fundamental block below \(\eta<4\).
\end{proof}

\begin{corollary}\label{cor:nilpotent-cyc-primary-blocks}
Let \(N\) be nilpotent and satisfy
\[
        \cyc(N)<2^{\pi(N)+2}.
\]
Let \(K\) be its squarefree-reduced core and write
\[
        K=P_1\times\cdots\times P_r
\]
as the product of its nontrivial Sylow subgroups.  Then \(r\le3\).  At most one
of the \(P_i\) is non-cyclic; if one is non-cyclic, then \(r\le2\).

If all \(P_i\) are cyclic, say
\[
        P_i\cong C_{p_i^{a_i}},
        \qquad a_1\ge a_2\ge\cdots\ge a_r\ge2,
\]
then the possible exponent patterns are exactly
\[
\begin{array}{ll}
r=0: & \text{empty},\\
r=1: & 2\le a_1\le6,\\
r=2: & (a_1,a_2)=(a,2),\quad 2\le a\le4,\\
r=3: & (a_1,a_2,a_3)=(2,2,2).
\end{array}
\]
\end{corollary}

\begin{proof}
For a nontrivial squarefree-reduced \(p\)-group \(P\), one has \(\cyc(P)\ge3\), with
equality only in the cyclic case \(P\cong C_{p^2}\).  Hence every Sylow factor
of \(K\) contributes at least \(3/2\) to \(\eta(K)\), and
\[
        \eta(K)=\prod_i \frac{\cyc(P_i)}2<4
\]
forces \(r\le3\).  If \(P_i\) is non-cyclic, then \(P_i\) has at least two
distinct subgroups of order \(p\), unless it is generalized quaternion.  In the
first case the number of subgroups of order \(p\) is congruent to \(1\) modulo
\(p\), and hence is at least \(p+1\); in the generalized quaternion case,
\(Q_8\) already has five cyclic subgroups.  In all cases
\(\cyc(P_i)\ge4\).  Thus a non-cyclic Sylow factor contributes at least \(2\),
and the inequalities
\[
        2^2\ge4,\qquad
        2\left(\frac32\right)^2>4
\]
show that there is at most one non-cyclic Sylow factor and that, if it exists,
then \(r\le2\).

If all Sylow factors are cyclic, then
\[
        \eta(K)=\prod_i\frac{a_i+1}{2}<4,
\]
or
\[
        \prod_i(a_i+1)<2^{r+2}.
\]
The displayed patterns follow immediately from this inequality with
\(a_1\ge\cdots\ge a_r\ge2\).
\end{proof}

\begin{corollary}\label{cor:nilpotent-cyc-fundamental-blocks}
Let \(S\ne1\) be nilpotent.  Then \(S\) is a fundamental block for the
cyclic-subgroup threshold if and only if \(S\) is a non-cyclic \(p\)-group with
\[
        \cyc(S)<8.
\]
Consequently \(S\) is one of
\[
        C_2^2,\quad C_4\times C_2,\quad D_8,\quad Q_8,\quad
        C_3^2,\quad C_5^2,
\]
where \(D_8\) denotes the dihedral group of order \(8\).
\end{corollary}

\begin{proof}
As in the subgroup-count case, a nilpotent fundamental block has only
non-cyclic Sylow subgroups.  Corollary~\ref{cor:nilpotent-cyc-primary-blocks}
shows that there can be at most one such Sylow factor.  Thus \(S\) is a
non-cyclic \(p\)-group, and the normalized condition \(\eta(S)<4\) is exactly
\(\cyc(S)<8\).

The displayed list is the non-cyclic part of the squarefree-reduced
\(p\)-group list in Proposition~\ref{prop:cyc-threshold-terminal-families}{\rm(a)}.
\end{proof}

\begin{corollary}\label{cor:mixed-cyc-fundamental-blocks}
Let
\[
        S=N\times ZM(m,n,r),
\]
where \(N\) is nilpotent, \(ZM(m,n,r)\) is nonabelian, and \((|N|,mn)=1\).
For \(q\mid n\), write \(q^{e_q}\Vert n\).  Then \(S\) is a fundamental block
for the cyclic-subgroup threshold if and only if the following three conditions
hold:
\[
        \cyc(N)
        \sum_{\substack{m_1\mid m,\ n_1\mid n\\ m/m_1\mid r^{n_1}-1}}
        \gcd\left(m_1,\frac{r^n-1}{r^{n_1}-1}\right)
        <2^{\pi(N)+\pi(mn)+2},
\]
every nontrivial Sylow subgroup of \(N\) is non-cyclic, and
\[
        r^{\,n/q^{e_q}}\not\equiv1\pmod m
        \qquad\text{for every }q\mid n.
\]
\end{corollary}

\begin{proof}
The first displayed inequality is the cyclic-subgroup threshold condition for
\(N\times ZM(m,n,r)\), by
Proposition~\ref{prop:cyc-threshold-terminal-families}{\rm(d)}.  The absence of
central cyclic Sylow subgroups is exactly the same componentwise condition as in
Corollary~\ref{cor:mixed-sub-fundamental-blocks}: cyclic Sylow subgroups of the
nilpotent factor are central, and in the \(ZM\)-factor the only possible central
Sylow subgroups come from the cyclic complement and are detected by the displayed
congruence.  Hence the three stated conditions are precisely the cyclic-subgroup
threshold together with the definition of a fundamental block.
\end{proof}

The criteria above describe fundamental blocks in selected families, but do
not classify all groups below the solvability bounds.  The finiteness
questions in Questions~\ref{ques:finitely-many-sub-blocks}
and~\ref{ques:finitely-many-cyc-blocks}, together with the question about
normalized values in Question~\ref{ques:next-normalized-values}, remain open.

\section*{Acknowledgements}
Hiranya Kishore Dey acknowledges the INSPIRE Faculty Fellowship for support
during the preparation of this work, and thanks the Department of Science and
Technology (DST), India, for funding.  He also thanks the Department of
Mathematics at the Indian Institute of Technology Jammu and the Department of
Mathematics at the Indian Institute of Science Education and Research Kolkata
for their research environments.

Khyati Sharma thanks Dr. Oorna Mitra and the Department of Mathematics, IISER
Kolkata, for financial support through the DST INSPIRE project and for providing
a good working environment during this work.  She also acknowledges the support
of the IISER Berhampur institute post-doctoral fellowship.

Cesar Galindo was partially supported by Grant INV-2025-213-3452 from the
School of Science of Universidad de los Andes.

\subsection*{Data Availability}
The class data and computational output supporting the finite-case checks
are supplied in Online Resource~1.  The underlying character tables and
tables of marks are available in CTblLib and TomLib \cite{CTblLib,TomLib}.

\subsection*{Code Availability}
The GAP code and its output are provided in Online Resource~1.

\subsection*{Competing Interests}
The authors have no competing interests to declare that are relevant to the content of this article.

\appendix

\section{Auxiliary estimates and Lie-type notation}
\label{app:auxiliary-estimates}

This appendix proves the numerical estimate used in the subgroup-count
arguments and fixes the Lie-type notation for the simple-group checks in
Appendices~\ref{sec:psl2} and~\ref{sec:lie-type}.

\subsection{The numerical estimate for \texorpdfstring{\(\PSL(2,q)\)}{PSL(2,q)}}

\begin{proof}[Proof of Lemma~\ref{lem:ineq-primes}]
Write \(q=p^f\), where \(p\) is an odd prime, and put
\[
        M=\frac{q^2-1}{2}.
\]
Then \(N=qM\) and \((q,M)=1\), since \(p\nmid q^2-1\). Hence
\[
        t=\pi(N)=1+\pi(M).
\]
Let \(s=\pi(M)=t-1\), and let \(P_s\) be the product of the first \(s\)
primes.  Since \(M\) has \(s\) distinct prime divisors, we have
\[
        P_s\le M\le q(q-1).
\]

Assume first that \(q\ge7\).  If \(s\ge4\), then
\[
        P_s\ge 2\cdot3\cdot5\cdot7\cdot 2^{s-4}
        > 10\cdot 2^s
        = 5\cdot 2^t.
\]
Thus \(q(q-1)\ge 5\cdot 2^t\).  If \(s\le3\), then for \(q=7,9\) one
has \(t=3\), and the desired inequality is immediate.  For \(q\ge11\), one has
\[
        q(q-1)\ge 110\ge 80\ge 5\cdot 2^t,
\]
because \(t=s+1\le4\).  This proves the first inequality.

Now assume that \(q\ge37\).  If \(s\ge5\), then
\[
        P_s\ge 2\cdot3\cdot5\cdot7\cdot11\cdot 2^{s-5}
        > 59\cdot 2^s
        = 59\cdot 2^{t-1}.
\]
If \(s\le4\), then
\[
        q(q-1)\ge 37\cdot36>59\cdot16\ge 59\cdot2^s
        =59\cdot2^{t-1}.
\]
This proves the second inequality.
\end{proof}

\subsection{Lie-type notation and maximal tori}\label{subsec:lie-notation}

For the Lie-type families, write \(q=p^f\), where \(p\) is prime and
\(f\ge1\).  Let \(\mathbf G\) be a connected simple algebraic group over
\(\overline{\mathbb F}_p\), and let \(F\) be a Steinberg endomorphism, so that
some positive power of \(F\) is a Frobenius endomorphism.  This includes the
exceptional isogenies defining the Suzuki and Ree groups.  Here \(q\) is the
conventional field parameter; for those groups the normalized parameter
attached to \(F\) is \(\sqrt q\); see
\cite[Section~1.4 and Example~1.4.21]{GeckMalle2016}.
The finite groups of Lie type arise from the fixed-point groups
\(\mathbf G^F\), after passing to a derived subgroup and then to a central
quotient when necessary.  In these appendices \(S\) denotes the resulting finite
simple group.

Fix an \(F\)-stable maximal torus \(\mathbf T\) of \(\mathbf G\), and write
\[
        W=N_{\mathbf G}(\mathbf T)/\mathbf T
\]
for the Weyl group, and let \(\sigma\) be the automorphism of \(W\) induced
by \(F\).  The \(\mathbf G^F\)-conjugacy classes of \(F\)-stable maximal tori
are parametrized by the \(W\)-conjugacy classes in the coset \(W\sigma\).
For a torus \(\mathbf T_w\) corresponding to \(w\sigma\), the normalizer
formula used in the proof is
\[
        N_{\mathbf G}(\mathbf T_w)^F/\mathbf T_w^F
        \cong C_W(w\sigma).
\]
This is the parametrization and normalizer formula of Malle--Testerman
\cite[Propositions~25.1--25.3]{MalleTesterman2011}.

When a torus is used inside the finite simple group \(S\), we write
\[
        T=S\cap \mathbf T_w^F
\]
or use the image of \(\mathbf T_w^F\) in the relevant central quotient.
Passing to a central quotient can only decrease the torus order used in the
denominator of our estimates.

A semisimple element \(x\in \mathbf G^F\) is regular if
\[
        C_{\mathbf G}(x)^\circ
\]
is a maximal torus.  If \(x\in\mathbf T\), this means that no root of
\(\mathbf G\) with respect to \(\mathbf T\) vanishes on \(x\); see
Malle--Testerman \cite[Corollary~14.10]{MalleTesterman2011}.  The estimates in
Appendix~\ref{sec:lie-type} use cyclic subgroups generated by regular
semisimple elements lying in suitable maximal tori.

For every torus used in Subsections~\ref{subsec:classical-groups}--\ref{subsec:exceptional-groups}, we set
\[
        m(T)=|C_W(w\sigma)|.
\]
When an estimate uses a larger number, the tables say explicitly that it is a
bound for \(m(T)\).  The symbol \(m(T)\) does not mean the abstract quotient
\(N_S(T)/T\) of the finite torus.  This convention avoids the
finite-normalizer degeneracies discussed by Baykalov
\cite[Corollary~1.5]{Baykalov2024}.

\subsection{The tori used in the proof}\label{subsec:tori-used}

The torus terminology used in Appendices~\ref{sec:psl2} and~\ref{sec:lie-type}
is as follows.  In \(\PSL(2,q)\), the split torus is the image of the diagonal
torus over \(\mathbb F_q\), while the nonsplit torus is the anisotropic torus
obtained from the quadratic extension.  Their projective orders are \(q-1,q+1\)
when \(q\) is even, and \((q-1)/2,(q+1)/2\) when \(q\) is odd.

For the classical groups of higher rank, a Singer torus in \(\PSL(n,q)\) is
the projective image of the determinant-one subgroup of the multiplication
torus \(\mathbb F_{q^n}^{\times}\) on the natural \(n\)-dimensional
\(\mathbb F_q\)-space.  The unitary torus used here is the one-cycle torus,
corresponding to an \(n\)-cycle in the Weyl-group parametrization; we do not
assume that it acts irreducibly on the natural unitary module.  In the
symplectic and orthogonal groups, an irreducible negative torus is the torus
corresponding, in the signed-cycle parametrization of maximal tori, to one
negative cycle of length \(n\); its order is bounded by \(q^n+1\).  Type
\((n-1)^-,1^-\) means one negative cycle of length \(n-1\) and one negative
cycle of length \(1\).  The cyclic decompositions and torus orders used in
Table~\ref{tab:classical-tori} are recorded by Buturlakin--Grechkoseeva
\cite[Theorems~2.1, 2.2, and 3--7]{ButurlakinGrechkoseeva2007}.

For exceptional groups, a Coxeter torus means the maximal torus corresponding
to a Coxeter element of the Weyl group.  In the twisted cases, the same role is
played by a class in the coset \(W\sigma\); we call this a twisted Coxeter torus.
The exceptional torus data used in Appendix~\ref{sec:lie-type} are taken from
the tables of Javed--Parkin--Rowley--Walton cited in
Subsections~\ref{subsec:suzuki-ree} and~\ref{subsec:exceptional-groups}.

\section{Alternating groups}\label{sec:alternating}

We count cyclic subgroups generated by prime cycles in \(A_n\).  The
automorphism groups of alternating groups are taken from Dixon--Mortimer
\cite{DixonMortimer1996}.  Recall that
\[
        \Aut(A_n)\cong \Aut(S_n)\cong S_n
        \qquad (n\ge4,\ n\ne6),
\]
while
\[
        \Aut(A_6)\cong \Aut(S_6)\cong S_6\rtimes C_2.
\]
In particular, for every \(n\ge5\),
\[
        \pi(\Aut(A_n))=\pi(n!).
\]

\begin{proposition}\label{prop:alternating}
For every \(n\ge5\),
\[
        \cyc(A_n)\ge 2^{\pi(\Aut(A_n))+2}.
\]
\end{proposition}

\begin{proof}
For \(n=5\), \(A_5\) has \(32\) cyclic subgroups, and
\(\pi(\Aut(A_5))=\pi(120)=3\); equality follows.  For \(n=6\), one has
\(\cyc(A_6)=167\), while \(\pi(\Aut(A_6))=3\); the desired bound is
\(32\).  The same values can be recovered from the character tables,
using the \textit{Atlas}/GAP data as in Appendix~\ref{sec:remaining-cases}
\cite{ATLAS,GAP,CTblLib}.

Assume \(n\ge7\).  By Bertrand's postulate, choose an odd prime \(r\) with
\[
        \frac n2<r\le n.
\]
Every \(r\)-cycle lies in \(A_n\).  The number of cyclic subgroups generated by \(r\)-cycles is
\[
        N_r=\frac{\binom nr (r-1)!}{r-1}
            =\binom nr (r-2)!
            =\frac{n!}{r(r-1)(n-r)!}.
\]
For \(n\ge22\), we have \(r\ge13\), and the elementary bound \(\pi(n!)\le n/2<r\) gives \(\pi(n!)\le r\).  Hence
\[
        2^{\pi(\Aut(A_n))+2}=2^{\pi(n!)+2}
        \le 4\cdot 2^r.
\]
For \(r\ge13\), \((r-2)!\ge4\cdot2^r\).  Thus
\[
        \cyc(A_n)\ge N_r\ge (r-2)!\ge 2^{\pi(n!)+2}.
\]

For \(7\le n\le21\), the same lower bound \(N_r\) gives:
\[
\begin{array}{c|c|c|c}
 n&r&\text{lower bound}&\text{upper bound}\\
\hline
7&5&126&64\\
8&7&960&64\\
9&7&4320&64\\
10&7&14400&64\\
11,12&11&\ge 362880&\le128\\
13,14,15,16&13&\ge 39916800&\le256\\
17,18&17&\ge 1307674368000&\le512\\
19,20,21&19&\ge 355687428096000&\le1024.
\end{array}
\]
This proves the result.
\end{proof}

\section{The groups \texorpdfstring{\(\PSL(2,q)\)}{PSL(2,q)}}\label{sec:psl2}

Throughout this appendix, \(q=p^f\ge4\).  We use the rank-one structure of
\(\PSL(2,q)\): the unipotent elements come from Sylow \(p\)-subgroups, and
the semisimple elements lie in split or nonsplit maximal tori.  The relevant
order formula, Sylow \(p\)-subgroups, maximal tori, and semisimple elements
are recorded by Malle--Testerman
\cite[Theorem~14.2, Corollaries~24.6 and~24.11, Propositions~25.3 and~26.6]{MalleTesterman2011}.
The split and nonsplit
projective torus orders are listed in Subsection~\ref{subsec:tori-used}.

\begin{lemma}\label{lem:psl2-formula}
Let \(S=\PSL(2,q)\).
\begin{enumerate}[label=(\alph*)]
\item If \(q\) is even, then
\[
\cyc(S)=1+(q^2-1)
+\frac{q(q+1)}2(d(q-1)-1)
+\frac{q(q-1)}2(d(q+1)-1).
\]
\item If \(q\) is odd, put
\[
        a=\frac{q-1}{2},\qquad b=\frac{q+1}{2},
\]
and
\[
\epsilon(n)=\begin{cases}1,&2\mid n,\\0,&2\nmid n.\end{cases}
\]
Then
\[
\begin{aligned}
\cyc(S)=&1+\frac{q^2-1}{p-1}+I(q)\\
&+\frac{q(q+1)}2\bigl(d(a)-1-\epsilon(a)\bigr)\\
&+\frac{q(q-1)}2\bigl(d(b)-1-\epsilon(b)\bigr),
\end{aligned}
\]
where
\[
I(q)=
\begin{cases}
\dfrac{q(q+1)}2,&q\equiv1\pmod4,\\[4pt]
\dfrac{q(q-1)}2,&q\equiv3\pmod4.
\end{cases}
\]
\end{enumerate}
\end{lemma}

\begin{proof}
We use the rank-one data from Malle--Testerman
\cite[Theorem~14.2, Corollaries~24.6 and~24.11, Propositions~25.3 and~26.6]{MalleTesterman2011}.
The nonidentity unipotent elements are precisely the nonidentity elements in
the \(q+1\) Sylow \(p\)-subgroups, each of order \(q\).  The split and nonsplit maximal tori are
cyclic; their projective orders are \(q-1\) and \(q+1\) when \(q\) is even, and
\((q-1)/2\) and \((q+1)/2\) when \(q\) is odd.  Their normalizers have twice
the corresponding torus order, and hence the numbers of split and nonsplit
tori are \(q(q+1)/2\) and \(q(q-1)/2\).  A nontrivial semisimple element has
one of these tori as its centralizer when its order is different from \(2\).
Thus such an element is counted in a unique split or nonsplit torus.  The
involutions are the unique involutions in the even split or nonsplit tori.

If \(q\) is even, the nontrivial unipotent elements are involutions, and there
are \((q+1)(q-1)=q^2-1\) of them.  The split and nonsplit tori have odd
orders \(q-1\) and \(q+1\), respectively.  Counting the nonidentity cyclic
subgroups in these tori gives the two terms involving \(d(q-1)-1\) and
\(d(q+1)-1\), and this proves (a).

If \(q\) is odd, the split and nonsplit tori have orders \((q-1)/2\) and \((q+1)/2\).  The number of cyclic unipotent subgroups of order \(p\) is
\[
        (q+1)\frac{q-1}{p-1}=\frac{q^2-1}{p-1}.
\]
Here \(\PSL(2,q)\) has \(q+1\) Sylow \(p\)-subgroups, each with \(q-1\)
nonidentity elements, and each cyclic subgroup of order \(p\) has \(p-1\)
generators.  The involutions occur in the split tori when \(q\equiv1\pmod4\)
and in the nonsplit tori when \(q\equiv3\pmod4\).  Hence the number of
involutions is \(q(q+1)/2\) in the first case and \(q(q-1)/2\) in the second.
Counting cyclic subgroups inside the split and nonsplit tori, and subtracting
the involution in a torus whenever the torus has even order, gives (b).
\end{proof}

\begin{proposition}\label{prop:psl2}
For every \(q=p^f\ge4\),
\[
        \cyc(\PSL(2,q))\ge 2^{\pi(\Aut(\PSL(2,q)))+2}.
\]
\end{proposition}

\begin{proof}
Put \(S=\PSL(2,q)\).  For \(q>3\), the automorphism description in Malle--Testerman
\cite[Corollary~24.6 and Proposition~24.23]{MalleTesterman2011} gives
\(\Aut(S)\cong P\Gamma L_2(q)\) and
\[
        |\Aut(S)|=f\,q(q^2-1).
\]
Thus the prime divisors of \(|\Aut(S)|\) are precisely those of \(f q(q^2-1)\).
Let
\[
        T=\pi(fq(q^2-1)).
\]
It suffices to prove \(\cyc(S)\ge2^{T+2}\).

First assume \(q=2^f\) is even.  Then
\[
        T\le1+\pi(f)+\pi(q^2-1),
\]
and since \(f\ge2\),
\[
        2^{T+2}\le 8f\,2^{\pi(q^2-1)}.
\]
The number \(q^2-1\) is odd.  Lemma~\ref{lem:pi-estimates} gives
\[
        2^{\pi(q^2-1)}<\frac{2q}{\sqrt3}.
\]
Thus
\[
        2^{T+2}<\frac{16}{\sqrt3}fq.
\]
If \(f\ge6\), then \(q=2^f\ge10f+1\), and therefore
\[
        q^2-1>\frac{16}{\sqrt3}fq.
\]
Since \(\PSL(2,q)\) has \(q^2-1\) involutions, the result follows.  The
remaining even cases are \(q=4,8,16,32\).  For \(q=4\),
\(\PSL(2,4)\cong A_5\), and equality holds.  For \(q=8,16,32\), the
involution count already gives
\[
\begin{array}{c|c|c|c}
q&T&q^2-1&2^{T+2}\\
\hline
8&3&63&32\\
16&4&255&64\\
32&5&1023&128.
\end{array}
\]

Now assume \(q\) is odd.  Again
\[
        T\le1+\pi(f)+\pi(q^2-1),
\]
hence
\[
        2^{T+2}\le8f\,2^{\pi(q^2-1)}.
\]
Since \(8\mid q^2-1\), Lemma~\ref{lem:pi-estimates} gives
\[
        2^{\pi(q^2-1)}<\sqrt{\frac23}\,q.
\]
Hence
\[
        2^{T+2}<8\sqrt{\frac23}\,fq<7fq.
\]
The involutions alone give
\[
        \cyc(\PSL(2,q))\ge\frac{q(q-1)}2.
\]
Thus we are done whenever \(q-1\ge14f\).  The remaining odd simple cases are
\[
        q\in\{5,7,9,11,13,25,27\}.
\]
For these, either Lemma~\ref{lem:psl2-formula} or the involution count gives:
\[
\begin{array}{c|c|c|c}
q&T&\text{lower bound for }\cyc(\PSL(2,q))&2^{T+2}\\
\hline
5&3&32&32\\
7&3&79&32\\
9&3&45&32\\
11&4&244&64\\
13&4&91&64\\
25&4&325&64\\
27&4&351&64.
\end{array}
\]
This proves the proposition.
\end{proof}

\section{Groups of Lie type other than \texorpdfstring{\(\PSL(2,q)\)}{PSL(2,q)}}\label{sec:lie-type}

We next treat the groups of Lie type other than \(\PSL(2,q)\).  The argument is
based on regular semisimple elements in maximal tori.  The notation for
\(\mathbf G\), \(F\), \(W\), \(\mathbf T_w\), and \(m(T)\) was fixed in
Subsection~\ref{subsec:lie-notation}.  Throughout this appendix the order formulas
are those of Malle--Testerman
\cite[Corollary~24.6 and Table~24.1]{MalleTesterman2011}.
For the passage to central quotients we use
\cite[Table~24.2 and Proposition~24.21]{MalleTesterman2011}; the descriptions
of diagonal, field, and graph automorphisms are taken from
\cite[Section~24.2]{MalleTesterman2011}.

\subsection{The regular torus bound}

We first turn the torus notation from Subsection~\ref{subsec:lie-notation} into
a counting estimate.

\begin{lemma}\label{lem:regular-torus-bound}
Let \(S\) be a finite simple group of Lie type, realized in its adjoint form as
a subgroup of \(\mathbf G^F\).  Let \(\mathbf T=\mathbf T_w\) be an
\(F\)-stable maximal torus of \(\mathbf G\), and put
\[
        T=S\cap \mathbf T^F.
\]
Suppose that \(T\) contains a regular semisimple element \(x\), with
\[
        C_{\mathbf G}(x)^\circ=\mathbf T.
\]
Put \(C=\langle x\rangle\).  Then
\[
        \cyc(S)\ge [S:N_S(C)]
        \ge \frac{|S|}{|T|\,|C_W(w\sigma)|}
        \ge \frac{|S|}{|T|\,|W|}.
\]
\end{lemma}

\begin{proof}
If \(g\in N_S(C)\), then \(gxg^{-1}=x^a\) for some \(a\) coprime to
\(|x|\).  Hence \(x^a\) generates the same cyclic subgroup as \(x\), and has
the same connected algebraic centralizer:
\[
        C_{\mathbf G}(x^a)^\circ=C_{\mathbf G}(x)^\circ=\mathbf T.
\]
It follows that
\[
        g\mathbf T g^{-1}
        =C_{\mathbf G}(gxg^{-1})^\circ
        =C_{\mathbf G}(x^a)^\circ
        =\mathbf T.
\]
Thus \(N_S(C)\le S\cap N_{\mathbf G}(\mathbf T)^F\).

The kernel of \(N_{\mathbf G}(\mathbf T)^F\to W\) is \(\mathbf T^F\), and its
intersection with \(S\) is \(T\).  Hence
\[
        |S\cap N_{\mathbf G}(\mathbf T)^F|
        \le |T|\,|C_W(w\sigma)|,
\]
by the normalizer formula from Subsection~\ref{subsec:lie-notation}.  Therefore
\[
        |N_S(C)|\le |T|\,|C_W(w\sigma)|,
\]
which gives the asserted lower bounds for \(\cyc(S)\).
\end{proof}

The existence hypothesis in Lemma~\ref{lem:regular-torus-bound} can be
checked directly when the finite torus \(T\) is cyclic.  If
\(T=\langle x\rangle\) and no root of \(\mathbf G\) with respect to
\(\mathbf T\) is trivial on \(T\), then \(x\) is regular semisimple:
indeed, \(\alpha(x)=1\) would imply \(\alpha(x^j)=1\) for every integer
\(j\), making \(\alpha\) trivial on \(T\).  Thus \(\alpha(x)\ne1\) for
every root \(\alpha\), as required by the root criterion recalled in
Subsection~\ref{subsec:lie-notation}.  For a noncyclic \(T\), noncontainment
in each individual root kernel does not suffice; one must exhibit an
element outside their union.

\subsection{Classical groups}\label{subsec:classical-groups}

Table~\ref{tab:classical-tori} gives the tori used for the classical groups,
using the terminology fixed in Subsection~\ref{subsec:tori-used}.  The normalizer
factors are the algebraic normalizer factors supplied by the Weyl-group
parametrization.  The relevant torus orders and cyclic decompositions for the
finite classical groups are recorded explicitly by Buturlakin--Grechkoseeva
\cite[Theorems~2.1, 2.2, and 3--7]{ButurlakinGrechkoseeva2007}: the linear and unitary rows are the one-cycle cases
in \cite[Theorems~2.1 and~2.2]{ButurlakinGrechkoseeva2007}; the symplectic row is the one negative cycle case
in \cite[Theorem~3]{ButurlakinGrechkoseeva2007}; the odd orthogonal row is the corresponding one negative
cycle case in \cite[Theorem~4]{ButurlakinGrechkoseeva2007}; and the even orthogonal rows follow from
\cite[Theorems~5--7]{ButurlakinGrechkoseeva2007}.

\begin{table}[ht]
\centering
\caption{Maximal tori used for the classical groups.}\label{tab:classical-tori}
\[
\begin{array}{c|c|c|c}
S&T&\text{order/bound for }|T|&\text{bound for }m(T)\\
\hline
\PSL(n,q),\ n\ge3&\text{Singer torus}&\dfrac{q^n-1}{(q-1)(n,q-1)}&n\\[8pt]
\PSU(n,q),\ n\ge3&\text{unitary one-cycle torus}&\dfrac{q^n-(-1)^n}{(q+1)(n,q+1)}&n\\[8pt]
\PSp(2n,q),\ n\ge2&\text{irreducible negative torus}&\le q^n+1&\le2n\\
\POmega(2n+1,q),\ n\ge2&\text{irreducible negative torus}&\le q^n+1&\le2n\\
\POmega^-(2n,q),\ n\ge4&\text{irreducible negative torus}&\le q^n+1&\le2n\\
\POmega^+(2n,q),\ n\ge4&\text{type }(n-1)^-,1^-&\le(q^{n-1}+1)(q+1)&\le4n.
\end{array}
\]
\end{table}

The simple groups with small exceptional isomorphisms are omitted from
Table~\ref{tab:classical-tori} only when the corresponding simple group has
appeared in an earlier appendix, for example \(\PSp(4,2)'\cong A_6\).
The tori in the first five rows are cyclic.  They are nondegenerate by
Baykalov \cite[Corollary~1.5 and Theorem~4.10]{Baykalov2024}: a single
\(n\)-cycle with \(n\ge3\), or a single negative cycle, avoids all the
listed exceptional types, which require positive cycles of length \(1\)
or \(2\).  The rank-two odd orthogonal case follows likewise from
\cite[Theorem~4.7]{Baykalov2024}; in even characteristic the odd orthogonal
case is identified with the symplectic case.  Their generators are therefore
regular by the observation following Lemma~\ref{lem:regular-torus-bound}.
For the last row we construct a regular element in the proof of
Proposition~\ref{prop:BCD}, without assuming that the torus is cyclic.
The remaining small cases outside the numerical estimates are verified in
Proposition~\ref{prop:table-checks}.

\begin{proposition}\label{prop:linear-unitary}
For all simple groups \(\PSL(n,q)\), \(n\ge3\), and \(\PSU(n,q)\), \(n\ge3\),
\[
        \cyc(S)\ge 2^{\pi(\Aut(S))+2}.
\]
\end{proposition}

\begin{proof}
We begin with \(\PSL(n,q)\).  The unitary case is treated after the linear
estimates.

Let \(S=\PSL(n,q)\).  The Singer torus gives
\[
\cyc(S)\ge \frac{|S|}{n|T|}
=\frac{q^{n(n-1)/2}(q-1)\prod_{i=2}^{n-1}(q^i-1)}{n}.
\]
In particular,
\[
        \cyc(S)\ge \frac{q^{n(n-1)}}{n2^{n-1}}.
\]
The outer automorphism group has diagonal part of order dividing \((n,q-1)\),
field part of order \(f\), and graph part of order at most \(2\), by the
automorphism description in Malle--Testerman
\cite[Section~24.2]{MalleTesterman2011}.  Hence
\[
        |\Out(\PSL(n,q))|\le 2(n,q-1)f\le2nf.
\]
Together with the order formula
\cite[Corollary~24.6]{MalleTesterman2011}, which gives
\(|S|<q^{n^2-1}\), this gives, by Lemma~\ref{lem:pi-estimates},
\[
        2^{\pi(\Aut(S))+2}
        \le 8\sqrt{2nf}\,q^{(n^2-1)/2}.
\]
Thus the preceding lower bound proves the result whenever
\[
        q^{(n-1)^2/2}\ge 8n2^{n-1}\sqrt{2nf}.
\]
Since \(f\le2^{f-1}\le q/2\), it suffices that
\[
        q^{(n-1)^2-1}\ge64n^3 4^{n-1}.
\]
At \(q=2,n=6\), this is \(2^{24}\ge64\cdot6^3\cdot4^5\).
The ratio of the left side to the right side increases with \(q\), and,
at \(q=2\), increases with \(n\ge6\), since its successive ratio is
\(2^{2n-3}(n/(n+1))^3>1\).  Thus all \(n\ge6\) are covered.

The remaining linear ranks are \(n=3,4,5\).  For \(n=3\), the exact Singer-torus bound becomes
\[
        \cyc(\PSL(3,q))\ge \frac{q^3(q-1)^2(q+1)}3.
\]
For \(n=3\), the same order and outer automorphism formulas
\cite[Corollary~24.6 and Section~24.2]{MalleTesterman2011} show that the prime divisors of \(|\Aut(\PSL(3,q))|\) are among those of
\[
        6f\,p\,(q-1)(q+1)(q^2+q+1),
\]
where the factors \(2\) and \(3\) already divide \(p(q^2-1)\).
Lemma~\ref{lem:pi-estimates} therefore gives
\[
        2^{\pi(\Aut(\PSL(3,q)))+2}
        \le16f\sqrt{(q^2-1)(q^2+q+1)}
        <16fq(q+1)\le8q^2(q+1).
\]
The torus bound dominates the last expression whenever
\(q(q-1)^2\ge24\), hence for \(q\ge4\).
For \(q=3\), the torus bound is \(144\), while the required bound is at most
\(64\).  The case \(q=2\) is \(\PSL(2,7)\), treated in
Proposition~\ref{prop:psl2}.

For \(n=4\) and \(q\ge4\), the exact torus bound gives
\[
        \cyc(\PSL(4,q))\ge
        \frac{q^6(q-1)(q^2-1)(q^3-1)}4\ge\frac{q^{12}}8,
\]
since \((1-q^{-1})(1-q^{-2})(1-q^{-3})\ge
(3/4)(15/16)(63/64)>1/2\).
The general upper estimate is at most \(16q^8\), using \(f\le q/2\),
and \(q^{12}/8\ge16q^8\) for \(q\ge4\).
This leaves only \(q=2,3\).  The case \(q=2\) is \(\PSL(4,2)\cong A_8\), treated in
Proposition~\ref{prop:alternating}.  For \(q=3\), the bound gives \(75816\),
while the required bound is at most \(64\).

For \(n=5\) and \(q\ge3\), the sufficient condition above follows from
\(q^{15}\ge3^{15}>2\,048\,000=64\cdot5^3\cdot4^4\).
At \(q=2\), the exact Singer-torus bound gives \(64512\), while the required bound is at
most \(128\).  This proves the linear case.

For \(\PSU(n,q)\), the same argument applies, using the corresponding bound
\[
        |\Out(\PSU(n,q))|\le 2(n,q+1)f\le2nf
\]
from the same automorphism formulas \cite[Section~24.2]{MalleTesterman2011}.
The order bound \(|S|<q^{n^2-1}\) follows by pairing the factors
\((1-q^{-2j})(1+q^{-(2j+1)})<1\) in the unitary order formula;
any remaining even-indexed factor is also less than \(1\).
The unitary torus lower bound is at least its linear counterpart, so the
comparisons for \(n\ge4\) apply unchanged.
For \(n=3\), the torus bound is
\[
        \cyc(\PSU(3,q))\ge \frac{q^3(q+1)^2(q-1)}3,
\]
and the corresponding prime-divisor estimate involves only the primes dividing
\[
        6f\,p\,(q-1)(q+1)(q^2-q+1).
\]
As in the linear case, Lemma~\ref{lem:pi-estimates} gives
\[
        2^{\pi(\Aut(\PSU(3,q)))+2}
        \le16f\sqrt{(q^2-1)(q^2-q+1)}<16fq^2\le8q^3.
\]
The torus bound dominates \(8q^3\) because \((q+1)^2(q-1)\ge24\) for
\(q\ge3\).  For \(n=4\), the only small parameters left are \(q=2,3\),
where the unitary one-cycle torus lower bounds are \(1296\) and \(163296\).  For
\(n=5\), the only such parameter is \(q=2\), where the lower bound is
\(248832\).  In these three cases the required bound is at most \(128\).  The
nonsimple case \(\PSU(3,2)\) is excluded.
\end{proof}

\begin{proposition}\label{prop:BCD}
Let \(S\) be a finite simple classical group of type \(B_n\), \(C_n\),
\(D_n^+\), or \(D_n^-\).  Then
\[
        \cyc(S)
        \ge 2^{\pi(\Aut(S))+2}.
\]
\end{proposition}

\begin{proof}
For the simple classical groups of types \(B,C,D\), the diagonal,
field and graph automorphisms give the uniform divisibility
\[
        |\Out(S)|\mid24f,
\]
which also covers triality in type \(D_4\); see Malle--Testerman
\cite[Section~24.2]{MalleTesterman2011}.

First consider \(B_n\) and \(C_n\), with \(n\ge2\).  Put
\[
        P_n(q)=\prod_{i=1}^n(q^{2i}-1).
\]
The order formula \cite[Corollary~24.6]{MalleTesterman2011} has the form
\[
        |S|=\frac1d q^{n^2}P_n(q),\qquad d\le2.
\]
Using the irreducible negative torus, Lemma~\ref{lem:regular-torus-bound} gives
\[
        \cyc(S)
        \ge \frac{q^{n(n-1)}P_n(q)}{8n}.
\]
The prime divisors of \(|\Aut(S)|=|S|\,|\Out(S)|\) are therefore contained
among those of \(p\), \(f\), the fixed integer \(24\), and \(P_n(q)\).  A
conservative estimate is
\[
        2^{\pi(\Aut(S))+2}\le 32fP_n(q).
\]
The required comparison is
\[
        q^{n(n-1)}\ge 256nf.
\]
This holds for \(n\ge4\) and all \(q\ge2\), and for \(n=3\) except possibly
\(q=2,3\).  Those two cases, together with the separate rank-two case
\(\PSp(4,q)\), are treated separately.  For \(\PSp(4,q)\) the negative torus
of order \(q^2+1\) gives
\[
        \cyc(S)\ge \frac{q^4(q^2-1)^2}{8}.
\]
Moreover
\[
        2^{\pi(\Aut(S))+2}\le64fq^2.
\]
This follows from the order formula for \(\PSp(4,q)\), the bound on
\(|\Out(S)|\), and Lemma~\ref{lem:pi-estimates}.  For \(q\ge3\), this lower
bound is sufficient.  The case \(q=2\) gives \(\PSp(4,2)'\cong A_6\), treated
in Proposition~\ref{prop:alternating}.  The cases \(B_3(2),B_3(3),C_3(3)\)
are verified in Proposition~\ref{prop:table-checks}.

Now consider \(D_n^-\), \(n\ge4\).  Let
\[
        P_{n-1}(q)=\prod_{i=1}^{n-1}(q^{2i}-1).
\]
The order formula \cite[Corollary~24.6]{MalleTesterman2011} is
\[
        |S|=\frac1d q^{n(n-1)}(q^n+1)P_{n-1}(q),\qquad d\le4.
\]
The irreducible negative torus gives
\[
        \cyc(S)
        \ge \frac{q^{n(n-1)}P_{n-1}(q)}{8n}.
\]
Using the same divisibility \(|\Out(S)|\mid24f\) and
\[
        2^{\pi(q^n+1)}\le 4q^{n/2},
\]
we obtain
\[
        2^{\pi(\Aut(S))+2}
        \le 128fP_{n-1}(q)q^{n/2}.
\]
The required comparison is
\[
        q^{n(n-1)-n/2}\ge1024nf.
\]
This holds for \(n\ge5\), and for \(n=4\) except possibly \(q=2\).  The
remaining case \(D_4^-(2)\) is verified in
Proposition~\ref{prop:table-checks}.

Finally consider \(D_n^+\), \(n\ge4\).  The order formula
\cite[Corollary~24.6]{MalleTesterman2011} is
\[
        |S|=\frac1d q^{n(n-1)}(q^n-1)P_{n-1}(q),\qquad d\le4.
\]
For the torus of type \((n-1)^-,1^-\), put \(k=n-1\) and
\(e=\gcd(2,q-1)\).  Choose \(\zeta\) of order \(q^k+1\), and set
\(\lambda=\zeta^e\).  In the standard negative-block description take
\(\widetilde x\) with torus coordinates
\[
        (\lambda,\lambda^q,\ldots,\lambda^{q^{k-1}},1).
\]
This element belongs to \(\Omega^+_{2n}(q)\): for odd \(q\) its block
exponents are \((2,0)\), whose sum is even, by
\cite[Lemma~4.6 and the proof of Theorem~4.8]{Baykalov2024}; for even
\(q\) the finite torus already lies in \(\Omega^+_{2n}(q)\), as in
\cite[Section~4.4]{Baykalov2024}.  Since \(k\ge3\), the order of
\(\lambda\) satisfies
\[
        L=\frac{q^k+1}{e}>q^{k-1}+1.
\]
The first \(k\) coordinates are nonidentity, and no two are equal or
mutually inverse: such an equality would give \(L\mid q^j-1\) or
\(L\mid q^j+1\) for some \(1\le j<k\), contradicting the displayed
bound.  In type \(D_n\) the roots evaluate as products or ratios of two
distinct coordinates, or their inverses, so no root takes the value \(1\)
on \(\widetilde x\), including roots involving the last coordinate.
Root values are preserved under the central quotient
\cite[Remark~2.6]{Baykalov2024}; hence the image of \(\widetilde x\) in
\(S\) is regular.  Lemma~\ref{lem:regular-torus-bound} therefore applies,
and the torus bounds are
\[
        |T|\le(q^{n-1}+1)(q+1)\le4q^n,
        \qquad m(T)\le4n.
\]
Therefore
\[
        \cyc(S)
        \ge
        \frac{q^{n(n-2)}(q^n-1)P_{n-1}(q)}{64n}.
\]
Again using \(|\Out(S)|\mid24f\),
\[
        2^{\pi(\Aut(S))+2}
        \le 32f(q^n-1)P_{n-1}(q).
\]
The required comparison is
\[
        q^{n(n-2)}\ge2048nf.
\]
This holds for \(n\ge5\), and for \(n=4\) except possibly \(q=2,3\).  The
cases \(D_4^+(2)\) and \(D_4^+(3)\) are verified in
Proposition~\ref{prop:table-checks}.
\end{proof}

\subsection{Suzuki and Ree groups}\label{subsec:suzuki-ree}

\begin{proposition}\label{prop:suzuki-ree}
The simple groups \({}^2B_2(q)\), \({}^2G_2(q)\), and \({}^2F_4(q)\), including the Tits group \({}^2F_4(2)'\), satisfy
\[
        \cyc(S)\ge 2^{\pi(\Aut(S))+2}.
\]
\end{proposition}

\begin{proof}
We use maximal tori and algebraic normalizer factors for the Suzuki and Ree
groups.  These twisted families are discussed in Malle--Testerman
\cite[Sections~22.1, 22.2, and 25.1]{MalleTesterman2011}.  The order formulas
and outer automorphism groups used here are taken from Malle--Testerman
\cite[Corollary~24.6 and Section~24.2]{MalleTesterman2011}.  The relevant
torus orders and algebraic normalizer factors are listed by
Javed--Parkin--Rowley--Walton
\cite[Tables~10, 11, and 13]{JavedParkinRowleyWalton2024}.  The square roots
in these torus orders are integers because the Suzuki and Ree parameters have odd exponent \(f\).
For the parameters \(f\ge3\) considered below, every finite maximal torus
of these groups is nondegenerate by Baykalov
\cite[Corollary~1.5]{Baykalov2024}; his parameter \(q(\sigma)\) is
\(\sqrt q\) in our notation.  Thus no root is trivial on such a torus,
and each cyclic torus used below has a regular generator by the
observation following Lemma~\ref{lem:regular-torus-bound}.
First let \(S={}^2B_2(q)\), where \(q=2^f\), \(f=2m+1\ge3\).  Then
\[
        |S|=q^2(q^2+1)(q-1),
        \qquad |\Out(S)|=f.
\]
The prime divisors of \(|\Aut(S)|\) are among those of \(2\), \(f\),
\(q-1\), and \(q^2+1\).  Hence
\[
        2^{\pi(\Aut(S))+2}\le 8f(q-1)(q^2+1).
\]
There are cyclic maximal tori of orders
\[
        q\pm\sqrt{2q}+1,
\]
with normalizer quotient of order \(4\).  Let
\[
        t=q-\sqrt{2q}+1.
\]
Then
\[
        \cyc(S)\ge\frac{|S|}{4t}.
\]
Since \(t\le q+1\), the general case follows from
\[
        \frac{q^2(q^2+1)(q-1)}{4(q+1)}
        \ge 8f(q-1)(q^2+1),
\]
that is,
\[
        q^2\ge32f(q+1).
\]
This holds for \(f\ge9\).  The remaining cases \(f=3,5,7\), i.e.\ \(q=8,32,128\), are listed explicitly:
\[
\begin{array}{c|c|c|c}
q&t&|S|/(4t)&2^{\pi(\Aut(S))+2}\\
\hline
8&5&1456&128\\
32&25&325376&64\\
128&113&75427840&256.
\end{array}
\]

Now let \(S={}^2G_2(q)\), where \(q=3^f\), \(f=2m+1\ge3\).  Then
\[
        |S|=q^3(q^3+1)(q-1),\qquad |\Out(S)|=f.
\]
The relevant cyclic tori have orders
\[
        q\pm\sqrt{3q}+1
\]
and normalizer quotient of order \(6\).  Using a torus of order at most \(2q\) gives
\[
        \cyc(S)
        \ge \frac{q^2(q^3+1)(q-1)}{12}.
\]
The prime divisors of \(|\Aut(S)|\) are among those of \(3\), \(f\),
\(q-1\), and \(q^3+1\).  Therefore
\[
        2^{\pi(\Aut(S))+2}
        \le16f(q-1)(q^3+1).
\]
The required comparison is \(q^2\ge192f\), which holds already at \(q=27\).
The group \({}^2G_2(3)\) is not simple and its derived subgroup is
\(\PSL(2,8)\), treated in Proposition~\ref{prop:psl2}.

Finally let \(S={}^2F_4(q)\), where \(q=2^f\), \(f=2m+1\ge3\).  Then
\[
|S|=q^{12}(q^6+1)(q^4-1)(q^3+1)(q-1),
\qquad |\Out(S)|=f.
\]
Choose the cyclic torus \(T\) of order \(q^2+1\), given by the row
\(w=[20,1]\) in \cite[Table~13]{JavedParkinRowleyWalton2024}.
Its generator is regular by the preceding nondegeneracy argument.
Using \(|W(F_4)|=1152\) and retaining the coarse bound
\[
        |T|=q^2+1\le(q+1)^4<2q^4\qquad(q\ge8),
\]
Lemma~\ref{lem:regular-torus-bound} gives
\[
        \cyc(S)\ge \frac{|S|}{2304q^4}.
\]
The prime divisors of \(|\Aut(S)|\) are among those of \(2\), \(f\),
\(q-1\), \(q^3+1\), \(q^4-1\), and \(q^6+1\), and therefore
\[
2^{\pi(\Aut(S))+2}
\le16f(q-1)(q^3+1)(q^4-1)(q^6+1).
\]
The required comparison is \(q^8\ge36864f\), which holds already for \(q=8\).
The exceptional group \({}^2F_4(2)'\) is verified in
Proposition~\ref{prop:table-checks}.
\end{proof}

\subsection{Exceptional groups}\label{subsec:exceptional-groups}

For the exceptional groups we record the torus input explicitly, using
the Coxeter and twisted Coxeter tori defined in Subsection~\ref{subsec:tori-used}.
The parametrization of maximal tori by Weyl-group classes, the order
formula for \(\mathbf T_w^F\), and the algebraic normalizer formula are those of
Malle--Testerman \cite[Propositions~25.1--25.3]{MalleTesterman2011}.  The
relevant exceptional torus data are tabulated by Javed--Parkin--Rowley--Walton
\cite[Tables~4--8, 12, and 14]{JavedParkinRowleyWalton2024}.  In their
notation, the entries used in Table~\ref{tab:exceptional-tori} are
\cite[Table~4]{JavedParkinRowleyWalton2024}, class \(6A\), for \(G_2(q)\);
\cite[Table~5]{JavedParkinRowleyWalton2024}, class \(12A\), for \(F_4(q)\);
\cite[Table~6]{JavedParkinRowleyWalton2024}, class \(12A\), for \(E_6(q)\);
\cite[Table~14]{JavedParkinRowleyWalton2024}, the row \(w=[2,18,6,3]\), with
\[
        T_w(q)=q^6-q^3+1=\Phi_{18}(q),
\]
for \({}^2E_6(q)\); \cite[Table~7]{JavedParkinRowleyWalton2024}, class \(18A\), for \(E_7(q)\); \cite[Table~8]{JavedParkinRowleyWalton2024}, class
\(30D\), for \(E_8(q)\); and \cite[Table~12]{JavedParkinRowleyWalton2024}, the row \(w=[1,2]\), for
\({}^3D_4(q)\), where \(T_w(q)=q^4-q^2+1=\Phi_{12}(q)\).

These entries give cyclic tori in the adjoint fixed-point groups, so
\(T=S\cap\mathbf T_w^F\) is cyclic as well.  For \(q>2\), Baykalov
\cite[Corollary~1.5]{Baykalov2024} shows that \(T\) is nondegenerate,
meaning that no root is trivial on \(T\).  At \(q=2\), the torus estimate
below is used only for \(F_4,E_6,{}^2E_6,E_7,E_8\).  The selected classes
have types \(F_4,E_6,E_6(a_1),E_7,E_8\), respectively, in Baykalov's
notation, and none occurs in his exceptional lists
\cite[Theorem~5.1]{Baykalov2024}.  For \({}^2E_6\), the type
\(E_6(a_1)\) is identified by the entry numbered \(24\) in the
correspondence between \cite[Tables~6 and~14]{JavedParkinRowleyWalton2024}.
Thus \(T\) is nondegenerate in every application below, and its generator
is regular by the observation following Lemma~\ref{lem:regular-torus-bound}.
This verifies the existence hypothesis inside the simple group, including
when it is a proper subgroup of the adjoint fixed-point group.

\begin{table}[ht]
\centering
\caption{Maximal tori used for the exceptional groups.}\label{tab:exceptional-tori}
\[
\begin{array}{c|c|c|c|c}
S&\text{torus class}&|T|\text{ before quotient}&\text{bound for }|T|&\text{bound for }m(T)\\
\hline
G_2(q)&\text{Coxeter}&\Phi_6(q)&\le(q+1)^2&\le 12\\
F_4(q)&\text{Coxeter}&\Phi_{12}(q)&\le(q+1)^4&\le 1152\\
E_6(q)&\text{Coxeter}&\Phi_3(q)\Phi_{12}(q)&\le(q+1)^6&\le 51840\\
{}^2E_6(q)&\text{twisted Coxeter}&\Phi_{18}(q)&\le(q+1)^6&\le 51840\\
E_7(q)&\text{Coxeter}&\Phi_2(q)\Phi_{18}(q)&\le(q+1)^7&\le 2903040\\
E_8(q)&\text{Coxeter}&\Phi_{30}(q)&\le(q+1)^8&\le 696729600\\
{}^3D_4(q)&\text{twisted Coxeter}&\Phi_{12}(q)&\le(q+1)^4&\le 192.
\end{array}
\]
\end{table}

Here \(m(T)=|C_W(w\sigma)|\) is the algebraic normalizer factor.  In
Table~\ref{tab:exceptional-tori} we use only the coarse bound
\(|C_W(w\sigma)|\le |W|\).  Passing from the simply connected or adjoint algebraic group to the
finite simple central quotient can only decrease the order of the torus used in
the denominator of Lemma~\ref{lem:regular-torus-bound}; therefore the upper
bounds for \(|T|\) in Table~\ref{tab:exceptional-tori} remain valid for the simple group.

\begin{proposition}\label{prop:exceptional}
The simple groups
\[
G_2(q),\ F_4(q),\ E_6(q),\ {}^2E_6(q),\ E_7(q),\ E_8(q),\ {}^3D_4(q)
\]
satisfy
\[
        \cyc(S)
        \ge 2^{\pi(\Aut(S))+2}.
\]
\end{proposition}

\begin{proof}
Using the order formulas \cite[Corollary~24.6]{MalleTesterman2011}, write
\[
        |S|=\frac1d q^N P(q),
\]
where \(q^N\) is the unipotent part and \(P(q)\) is the product of the
non-unipotent cyclotomic factors.  Let \(M\) be the bound for \(m(T)\) recorded
in Table~\ref{tab:exceptional-tori}.  Applying
Lemma~\ref{lem:regular-torus-bound} to the corresponding torus gives
\[
        \cyc(S)\ge \frac{q^NP(q)}{d\,M\,(q+1)^\ell},
\]
where \(\ell\) is the absolute rank.  On the other hand, the same references give
the outer automorphism groups as products of bounded diagonal and graph parts
with the field automorphism group of order \(f\).  Combining this with the order formula gives the conservative estimate
\[
        2^{\pi(\Aut(S))+2}\le 64fP(q).
\]
Here the constant \(64\) absorbs the prime \(p\) and the bounded diagonal and graph factors.
Thus it is enough to check
\[
        \frac{q^N}{d\,M\,(q+1)^\ell}\ge 64f.
\]

The numerical data needed for this last inequality are:
\[
\begin{array}{c|c|c|c|c}
S&N&\ell&M&d\\
\hline
G_2(q)&6&2&12&1\\
F_4(q)&24&4&1152&1\\
E_6(q)&36&6&51840&\le3\\
{}^2E_6(q)&36&6&51840&\le3\\
E_7(q)&63&7&2903040&\le2\\
E_8(q)&120&8&696729600&1\\
{}^3D_4(q)&12&4&192&1.
\end{array}
\]

To control the field degree as \(q\) grows, use \(f\le\log_2q\).
For the exponents in the table, \(N-\ell\ge2\), and
\[
        q\frac{d}{dq}\log\!\left(
        \frac{q^N}{(q+1)^\ell\log_2q}\right)
        =N-\frac{\ell q}{q+1}-\frac1{\log q}
        >N-\ell-\frac1{\log2}>0
        \qquad(q\ge2),
\]
where \(\log\) denotes the natural logarithm.
Thus a comparison with \(64\log_2q\) at an initial parameter suffices for
all larger parameters, using the displayed upper bounds for \(d\).
For \(F_4,E_6,{}^2E_6,E_7,E_8\), this comparison already holds at \(q=2\).  In type \(F_4\),
\[
        \frac{2^{24}}{1152\cdot3^4}>64
\]
and direct substitution of the data in the table gives the same conclusion for
the other four families.

For \(G_2(q)\), it reads
\[
        \frac{q^6}{12(q+1)^2}\ge64f.
\]
At \(q=7\), we have \(f=1\) and \(7^6/(12\cdot8^2)>64\).
For \(q\ge8\), the monotonicity above applies, since
\(8^6/(12\cdot9^2)>64\log_2 8\).
The remaining simple cases \(q=3,4,5\) are
verified in Proposition~\ref{prop:table-checks}; \(G_2(2)\) is not simple and
\(G_2(2)'\cong\PSU(3,3)\), treated in Proposition~\ref{prop:linear-unitary}
\cite{ATLAS}.

For \({}^3D_4(q)\), it reads
\[
        \frac{q^{12}}{192(q+1)^4}\ge64f.
\]
At \(q=4\), we have \(4^{12}/(192\cdot5^4)>64\log_2 4\), so the same
monotonicity proves the inequality for all \(q\ge4\).
The remaining cases \(q=2,3\) are verified in
Proposition~\ref{prop:table-checks}.
\end{proof}

The propositions in this appendix reduce the remaining Lie-type cases to the
finite character-table checks recorded in Proposition~\ref{prop:table-checks}.

\section{The remaining finite cases}\label{sec:remaining-cases}

Only finitely many groups are left outside the uniform estimates: a short list
of small Lie-type groups, the Tits group, and the 26 sporadic groups.  These
cases are treated from character tables.  If the table has
conjugacy classes \(C\), with representative order \(o(C)\), then
Lemma~\ref{lem:cyclic-subgroup-number-in-terms-order} gives the exact value
\[
        \cyc(G)=\sum_C\frac{|C|}{\phi(o(C))}.
\]
The character tables used here are the GAP/CTblLib tables \cite{GAP,CTblLib}.
Their sources include the \textit{Atlas of Finite Groups}
\cite{ATLAS,CTblLib}.  No enumeration of group elements is involved.

For each finite case we compare the value of \(\cyc(G)\) with
\[
        B(G,d):=2^{\pi(|G|d)+2},
\]
where \(d\) is chosen to be divisible by the relevant outer automorphism
factor.  Hence \(\pi(\Aut(G))\le \pi(|G|d)\), and proving
\(\cyc(G)\ge B(G,d)\) proves the desired inequality.  For the Lie-type
residual rows in Table~\ref{tab:lie-residual}, the uniform value \(d=24\) is
larger than needed in some cases; this only makes \(B(G,d)\) larger, and \(24\)
is divisible by all outer-automorphism factors occurring there.  The exact
values needed for those cases are recorded in Table~\ref{tab:lie-residual}.

\begin{table}[ht]
\centering
\caption{Residual Lie-type cases checked from character tables.}\label{tab:lie-residual}
\small
\begin{tabular}{@{}lrrr@{}}
\toprule
\(G\) & \(d\) & \(\cyc(G)\) & \(B(G,d)\) \\
\midrule
\(\Sp(6,2)\) & 24 & 414920 & 64 \\
\(\Omega_7(3)\) & 24 & 931607744 & 128 \\
\(\PSp(6,3)\) & 24 & 765946976 & 128 \\
\(\Omega_8^-(2)\) & 24 & 37244792 & 128 \\
\(\Omega_8^+(2)\) & 24 & 42067832 & 64 \\
\(\POmega_8^+(3)\) & 24 & 863873478392 & 128 \\
\(G_2(3)\) & 24 & 1052690 & 64 \\
\(G_2(4)\) & 24 & 49287032 & 128 \\
\(G_2(5)\) & 24 & 851271782 & 128 \\
\({}^3D_4(2)\) & 24 & 34618040 & 64 \\
\({}^3D_4(3)\) & 24 & 1258595228984 & 128 \\
\({}^2F_4(2)'\) & 2 & 4025192 & 64 \\
\bottomrule
\end{tabular}
\end{table}

A standalone GAP verification script, execution instructions, and the complete
successful reference output are provided in Online Resource~1.  The reference
run used GAP~4.12.1, CTblLib~1.3.4, and TomLib~1.2.9 and verified all
finite-case values in Tables~\ref{tab:solvable-finite-simple-checks},
\ref{tab:lie-residual}, and~\ref{tab:sporadic}.  The output records the runtime
versions, table identifiers, class data, and all comparisons.  The logged
class contributions and subgroup-class sums were also checked independently
using exact arithmetic.  The outer-automorphism orders and divisibility bounds
remain literature inputs; they are not inferred from the character tables.

\begin{proposition}\label{prop:table-checks}
The small Lie-type residual cases cited in Appendix~\ref{sec:lie-type} satisfy
\[
        \cyc(S)\ge 2^{\pi(\Aut(S))+2}.
\]
The Tits group \({}^2F_4(2)'\) also satisfies this inequality.
\end{proposition}

\begin{proof}
The required cases are exactly the rows in Table~\ref{tab:lie-residual}.  In
every row the value of \(\cyc(G)\) recorded in the table is larger than
\(B(G,d)\), and \(d\) is divisible by the relevant outer automorphism factor.
Hence
\[
        \cyc(G)\ge B(G,d)\ge 2^{\pi(\Aut(G))+2}.
\]
\end{proof}

\subsection{Sporadic groups}

The remaining finite simple groups are the 26 sporadic groups.  For them, the
\textit{Atlas} records \(|\Out(S)|\in\{1,2\}\) \cite{ATLAS}.  Since all sporadic
groups have even order, passing from \(|S|\) to \(|\Aut(S)|\) introduces no new
prime divisors.  Thus the required bound is \(2^{\pi(S)+2}\).

Table~\ref{tab:sporadic} lists all 26 sporadic groups, grouped by the required
bound.  The last column gives either the exact minimum among the listed groups
in that row, for the first four rows, or a common lower bound, for the last
five rows.  The complete reference output in Online Resource~1 gives the exact
value for each individual group and confirms all nine row bounds.

\begin{table}[ht]
\centering
\caption{Sporadic groups grouped by the required bound.}\label{tab:sporadic}
\small
\begin{tabular}{@{}>{\raggedright\arraybackslash}p{0.53\textwidth}rr@{}}
\toprule
Groups \(S\) & \(2^{\pi(\Aut(S))+2}\) & lower bound for \(\cyc(S)\) \\
\midrule
\(M_{11}, M_{12}, J_2\) & 64 & 2576 \\
\(M_{22}, J_3, HS, McL, He\) & 128 & 111806 \\
\(M_{23}, M_{24}, J_1, Ru, Suz, Co_2, Co_3, Fi_{22}, HN\) & 256 & 40966 \\
\(ON, Co_1, Th\) & 512 & 60030047156 \\
\(Fi_{23}, Ly\) & 1024 & \(>10^{15}\) \\
\(Fi_{24}'\) & 2048 & \(>10^{23}\) \\
\(J_4\) & 4096 & \(>10^{18}\) \\
\(B\) & 8192 & \(>10^{32}\) \\
\(M\) & 131072 & \(>10^{49}\) \\
\bottomrule
\end{tabular}
\end{table}

In every row of Table~\ref{tab:sporadic}, the lower bound for \(\cyc(S)\)
exceeds the required bound.  Hence every sporadic simple group \(S\) satisfies
\[
        \cyc(S)\ge 2^{\pi(\Aut(S))+2}.
\]

\end{document}